\documentclass[draftclsnofoot,onecolumn,journal,12pt]{IEEEtran}

\usepackage{amssymb,graphicx}
\usepackage{latexsym}
\usepackage{mathtools}
\usepackage{cite}
\usepackage{color}
\usepackage{subfigure}
\usepackage{bm,bbm}
\usepackage{tabularx,booktabs}
\usepackage{ltablex}
\usepackage{dsfont}
\usepackage{mathrsfs}
\usepackage{enumitem}
\usepackage{cases}
\usepackage{yhmath}

\newcolumntype{M}[1]{>{\centering\arraybackslash}m{#1}}

\usepackage[linesnumbered,lined,boxed]{algorithm2e}
\usepackage{setspace}
\SetAlCapSkip{1em}
\SetKwComment{Comment}{$\triangleright$\ }{}
\RequirePackage[OT1]{fontenc}
\RequirePackage{amsthm,amsmath}
\RequirePackage[numbers]{natbib}
\RequirePackage[colorlinks,citecolor=blue,urlcolor=blue]{hyperref}
\newcommand*{\fullref}[2]{\hyperref[{#1}]{\ref*{#1} {\it(#2)}}}
\newcommand*{\partialref}[2]{\hyperref[{#1}]{#2}}

\newcommand{\bA}{\bm{A}}

\newcommand{\bF}{\mathbf F}
\newcommand{\bR}{\mathbf R}
\newcommand{\bS}{\bm{S}}
\newcommand{\bZ}{\bm{Z}}

\newcommand{\bs}{\bm{s}}
\newcommand{\bv}{\bm{v}}
\newcommand{\balpha}{\bm{\alpha}}
\newcommand{\db}[1]{\Bar{\Bar{#1}}}

\newcommand{\bz}{\mathbf z}
\newcommand{\bbR}{\mathbb R}
\newcommand{\bbP}{\mathbb P}

\newcommand{\innerproduct}[2]{\bigl\langle #1,#2\bigr\rangle}

\newcommand{\bbE}{\mathbb E}

\newcommand{\boalphaLP}{\overline{\bm{\alpha}}^{\text{LP}}}
\newcommand{\oalphaLP}{\overline{\alpha}^{\text{LP}}}

\newcommand{\sS}{\mathsf S}
\newcommand{\sA}{\mathsf A}

\newcommand{\Lipschitza}{a}

\newcommand{\lPhi}{\Phi}  
\newcommand{\PhiWC}{\Phi^{\text{WC}}}
\newcommand{\lPhih}{\Phi^h}  
\newcommand{\PhiWCh}{\Phi^{\text{WC},h}}

\newcommand{\dblPhih}{\db{\Phi}^h}

\newcommand{\PsiZ}{\Psi^0}
\newcommand{\PsiZWC}{\Psi^{0,\text{WC}}}

\newcommand{\lPsi}{\Psi}

\newcommand{\actionphi}{a^{\phi}}
\newcommand{\actionphiVec}{\bm{a}^{\phi}}
\newcommand{\actionphih}{a^{\phi,h}}
\newcommand{\actionphihVec}{\bm{a}^{\phi,h}}

\newcommand{\actionpsihVec}{\bm{a}^{\psi,h}}

\newcommand{\action}{a}
\newcommand{\actionVec}{\bm a}
\newcommand{\dummyAction}{\action^0}

\newcommand{\deltay}{\Delta_{\mathscr{Y}}}
\newcommand{\deltayh}{\Delta_{\mathscr{Y}}^h}

\newcommand{\indicator}{\mathbb{I}}
\newcommand{\inter}[1]{\text{int}\left\{#1\right\}}

\newcommand{\orule}{\varphi}
\newcommand{\dborule}{\varphi}

\usepackage[scr=boondox]{mathalfa}
\newcommand{\scra}{\mathscr{a}}

\newcommand{\scrp}{\mathscr{p}}
\newcommand{\scru}{\mathscr{u}}

\newcommand{\scrz}{\mathscr{z}}

\newcommand{\scrR}{\mathscr{R}}

\newcommand{\frakR}{\mathfrak{R}}

\newcommand{\calJ}{\mathcal{J}}

\newcommand{\calJext}{\mathcal{J}^+}

\newcommand{\ALP}{FAB }
\newcommand{\GSA}{\text{GSA} }

\newcommand{\GSATs}{GSATs }

\renewcommand{\proof}[1]{\emph{#1}}
\renewcommand{\endproof}{{\hfill $\blacksquare$}}

\newtheorem{theorem}{\bf Theorem}
\newtheorem{lemma}{\bf Lemma}
\newtheorem{proposition}{\bf Proposition}

\newtheorem{assumption}{\bf Assumption}

\usepackage[deletedmarkup=xout]{changes}

\graphicspath{{./figs/}}

\newcounter{condition}
\newenvironment{condition}[1]{\refstepcounter{condition}
\par\medskip
\noindent \textbf{#1:~} \it}{\medskip \newline}

\newcounter{externalTheorem}

\newcommand{\norm}[1]{\left\lVert #1 \right\rVert}

\begin{document}
\title{Weakly-Coupled Multi-Action Restless Bandits - Exponential Convergence in Probability}
\author{Jing~Fu,
Bill~Moran~and~Jos\'e Ni\~no-Mora
\thanks{
Jing Fu is with Department of Electrical and Electronic Engineering, School of Engineering, STEM College, RMIT University, Australia (e-mail: jing.fu@rmit.edu.au).}
\thanks{Bill Moran is with Department of Electrical and Electronic Engineering, the University of Melbourne, VIC 3010, Australia (e-mail:wmoran@unimelb.edu.au).}
\thanks{Jos\'e Ni\~no-Mora is with Department of Statistics, Carlos III University of Madrid, Spain (E-mail: jose.nino@uc3m.es)}
}

\maketitle

\begin{abstract}
We study a finite time horizon Markov decision process (MDP) consisting of several groups of multi-action finite-state restless bandit processes, which are identical within each group.
The bandit processes into different groups can be rather different.
The bandit processes are subject to multiple weakly coupled constraints on their state and action variables. 
In contrast to prior studies that considered only a few specific policies/algorithms, here, we study the behaviours of the general stochastic process
and, most importantly, the design of policies that guarantee its convergence to an ideal trajectory as the problem size increases.
We prove that, for any policy in a rather general class, the resulting stochastic process converges in probability to a deterministic process as the system size (measured by the number of bandits) tends to infinity, at an exponential rate.
Unlike the previous proofs, our exponential convergence does not rely on any non-degenerate assumptions.
It follows that the chosen policy asymptotically approaches optimality (with exponentially diminishing suboptimality) in the size dimension if and only if the deterministic process coincides with optimality. 
We further propose a policy and prove that, in general, it converges in probability to optimality exponentially fast in the system size.
\end{abstract}


\section{Introduction}\label{sec:introduction}
\subsection{Overview and Motivation}\label{subsec:background}

\emph{Restless multi-arm bandit (RMAB)} problem, originally proposed by \cite{whittle1988restless}, is a Markov decision process (MDP) with high-dimensional state and action spaces.
It consists of multiple individual and otherwise independent \emph{bandit processes}, which are again  MDPs but with binary actions \cite{gittins2011multiarmed}, linked through a constraint over all their action variables. 
It has offered a flexible modeling framework applicable to a broad class of dynamic systems and stochastic optimisation problems; see \cite{krishnamurthy2007structured,wang2020whittle,fu2016asymptotic,fu2020energy,avrachenkov2016whittle}, and the review \cite{ninomora2023survey}.

Typically, the complexity stems from the multiplicity of such processes. These RMABs are themselves MDPs whose state space grows exponentially with the number of constituent bandits, rendering them computationally intractable (PSPACE-hard, see   \cite{papadimitriou1999complexity}).
The large state space prevents conventional optimization techniques for MDP, such as value iteration, or exhibits excessively slow convergence of reinforcement learning techniques.

Past efforts resort to scalable heuristics with theoretical performance guarantees.
\cite{whittle1988restless} proposed the famous Whittle index policy that quantifies all states of each of the bandit processes through real-valued \emph{indices} by solving a relaxed version of the original problem and, at each decision epoch, always prioritizes the bandit processes with highest indices.
The Whittle indices are proposed under a non-trivial condition, referred to as \emph{Whittle indexability}, and computed with computational complexity linear to the number of bandit processes, such as the algorithms discussed in \cite{nino2006restless,nino2007dynamic}, which is thus scalable for a general RMAB problem.
After \cite{whittle1988restless} conjectured \emph{asymptotic optimality} of Whittle index policy - it approaches optimality when the scale of the problem, measured by the number of the bandit processes, tends to infinity,
\cite{weber1990index} proved it with further assumed a non-trivial condition related to the existence of a global attractor of its associated averaging process.
The global attractor condition is crucial, although it may naturally exist in special cases, in the sense of showing asymptotic optimality for RMAB with long-run objectives, such as discussions in \cite{weber1990index,verloop2016asymptotically,fu2018restless,fu2020energy}.
In contrast to the challenges of establishing Whittle indexability and the existence of a global attractor, recent work~\cite{brown2020index,gast2023linear,yan2024optimal,fu2024patrolling} explored approximations to fluid models,  which are tractable to solve and serves upper/lower bounds of the maximum/minimum of the original RMAB. 
Such policies, although tractable, exhibit relatively higher computational complexity than that of the Whittle indices but, remarkably, achieve proved asymptotic optimality in general.

The conventional RMAB formulation is limited by its assumptions of binary actions and a single linear constraint on the action variables.
More general and practically relevant multi-action bandit processes (MABPs), featuring richer constraints on both action and state variables, have been studied in \cite{nino2008index,nino2022multigear,fu2018restless,brown2023fluid,fu2024patrolling,gast2024reoptimization,fu2025restless}.
In this paper, an MABP is simply an MDP with finite state and action spaces. We retain the term ``bandit process'' in line with the RMAB literature and to distinguish it from more general MDPs.
The discussions have been further generalized and enriched by building on the valuable work of~\cite{brown2023fluid,fu2024patrolling,gast2024reoptimization}, making them more applicable to practical cases - these  MABPs are coupled through multiple \emph{weakly-coupled} constraints, making the overall problem fall in the scope of weakly-coupled MDPs (see \cite{adelman2008relaxations}).
They have been  analyzed through linear-programming-based techniques with achieved asymptotic optimality in the finite-time-horizon case.
In \cite{fu2024patrolling}, for a special case of the weakly coupled constraints, the performance deviation was proved to converge  in probability at $e^{-O(h)}$, where $h$ is the number of  bandit processes.
In \cite{brown2023fluid}, for the general weakly-coupled bandit processes, under certain \emph{non-degenerate} assumptions for a relaxed version of the problem,
the performance convergence between optimality and the designed scalable policies was also proved to be $O(e^{-h})$ in expectation.
Later in \cite{gast2024reoptimization}, with a slightly different non-degenerate assumption, the performance sub-optimality rate has been substantially improved to $e^{-O(h)}$ in both probability and expectation.

Building on and extending existing results, This paper focuses on the general case where MABPs are coupled through multiple weakly-coupled constraints and no assumptions about non-degenerate conditions.
As demonstrated in \cite[Figure 2]{gast2024reoptimization}, the non-degenerate assumptions are non-trivial and do not hold in general.
We begin by extending the notion of an RMAB. Here, the basic structure consists of a finite collection of groups of multi-action bandit processes. These groups are called \emph{gangs}, and each gang comprises bandit processes that share the same state and action spaces, as well as the same state–action-dependent transition probabilities and reward rates; that is, bandit processes within a gang are stochastically identical. Bandit processes belonging to different gangs, however, may be completely distinct. We study the general setting with weakly coupled constraints and refer to such problems as Weakly Coupled Gangs (WCGs).
A detailed definition and explanation on the relationship between WCGs and conventional RMAB are provided in Section~\ref{subsec:model}.
WCGs encompass conventional RMABs and are at least as difficult to solve.

For the general WCG, unlike the past work~\cite{brown2023fluid,gast2023linear,gast2024reoptimization}, which focused on a few specific algorithms/policies and their performance, we study the behaviors of the general WCG process; most importantly, how to design policies under which the WCG process converges to an ideal trajectory as the problem size increases.

We refer to the scale of the problem; that is $h$, as the \emph{size dimension}, in contradistinction to the time dimension (timeline).

\subsection{Main Results and Roadmap}
\label{sec:Results}

Our main contributions consist of theorems that analyze the behaviors of the general WCG process and  practical methods for which the performance is theoretically guaranteed by the theorems.
They can be summarized into three points.
\begin{enumerate}[label=(\arabic*)]
\item 
We consider the class of all policies, which are well defined and satisfy a \partialref{cond:weak_stab}{Lipschitz–limit regularity} in the limit regime with $h\to \infty$. 
The \partialref{cond:weak_stab}{Lipschitz–limit regularity} specifies how we design the action variables of a policy, but does not impose any intrinsic condition on the WCG problem itself, such as the non-degenerate assumptions in~\cite{brown2023fluid,gast2023linear,gast2024reoptimization}.
We refer to such policies as \emph{Lipschitz-limit} policies and prove that the Whittle index policy~\cite{whittle1988restless} and the policies proposed in the prior work~\cite{brown2023fluid,gast2023linear} fall into this class.
We prove, for the first time, in Theorems~\ref{theorem:convergence-Z} and \ref{theorem:convergence_Z_exp} that 
the WCG process, under any Lipschitz-limit policy, converges in probability to a deterministic trajectory at rate $e^{-O(h)})$.
Such exponential convergence in probability, in contrast to the above-mentioned studies, holds in general without assuming non-degenerate conditions.
The relevant discussions are provided in Section~\ref{sec:convergence}.
\item 
We prove a necessary and sufficient condition for a Lipschitz-limit policy to approach the optimum of the Whittle relaxation of the WCG problem in the asymptotic regime ($h\to \infty$).
Moreover, the convergence (in probability) is at rate $e^{-O(h)})$.
This relaxed optimum coincides with the optimum of the original WCG problem in the asymptotic regime in fairly general settings, such as those discussed in~\cite{brown2020index,brown2023fluid,gast2023linear,gast2024reoptimization,fu2020energy,fu2018restless,fu2024patrolling}; in this case, the provided condition is also necessary and sufficient for a Lipschitz-limit policy to be asymptotically optimal.
The analysis for the necessary and sufficient condition is provided in Section~\ref{subsec:alp:asym_opt_exp}.
\item Under the same resource allocation scenario  as in~\cite{brown2023fluid,gast2024reoptimization}, in Section~\ref{subsec:example}, we propose an algorithm that always lead to an asymptotically optimal policy.
We prove that the suboptimality diminishes in probability exponentially in the size dimension (without requesting non-degenerate assumptions).
The computational complexity of the algorithm is logarithmic-linear in the size dimension - scalable for a large-scale WCG problem. 
Through simulation results, which do not satisfy any non-degenerate condition, we demonstrated that the performance deviation between the proposed algorithm and optimality quickly diminishes in the size dimension, which is consistent with our theoretical conclusions.
\end{enumerate}

\subsection{Relations to the Literature} \label{subsec:rWork}

The past work for weakly coupled systems mainly focused on appropriate relaxation techniques that achieve tighter upper/lower bounds for the original maximum/minimum.
The LP-based relaxation technique was compared with Lagrangian relaxation approaches for weakly coupled dynamic problems by \cite{adelman2008relaxations}  who proved that the LP-based relaxation produced tighter bounds on the performance deviation from optimality, while the Lagrangian relaxation led to lower computational complexity. 
For a decomposable MDP with long-run discounted cumulative reward, \cite{bertsimas2016decomposable} proposed the fluid linear optimization (fluid-LO) formulation, which is a linear programming optimization problem and provides tighter bounds of the original optimality than those of Lagrangian and LP-based relaxations. 
For the case where the MDP actions are subject to weakly coupled constraints, \cite{bertsimas2016decomposable} proved that tighter bound can be achieved by increasing the size of the fluid-LO formulation.
\cite{ye2017weakly} discussed the information relaxation technique for decomposable MDP with a long-run discounted objective and weakly coupled sub-problems, which are also MDPs with the same state and action spaces and transition kernels. 
Such an MDP is usually referred to as a weakly coupled MDP.
Unlike the other relaxations that initiating relaxed versions of optimization problems, information relaxation iterative an information function, which converges to an upper bound, tighter than Lagrangian relaxation, of the value function of the original MDP.  
For weakly coupled MDP with finite-time horizon, \cite{nadarajah2025self} proposed the feasibility network relaxation (FNR) technique that provides an upper bound, tighter than Lagrangian relaxation and the same as LP-based relaxation (referred to as approximate linear programming (ALP) in \cite{nadarajah2025self}), of the original optimality.
In particular, FNR initiates a relaxed problem (linear programming) for which the sum of the numbers of the variables and the constraints is equal or less than that of LP-based relaxation; while, it maintains the same upper bound.
This paper focuses on the weakly coupled MDP with finite time horizon, where we prove that the weakly coupled MDP converges to a deterministic trajectory in the asymptotic regime.
In a general resource allocation scenario, it leads to the convergence, in the asymptotic regime, between its optimality and the upper bound equal to or tighter than that of the LP-based relaxation.

Apart from \cite{gast2023linear,gast2024reoptimization,brown2020index,brown2023fluid} mentioned in Section~\ref{subsec:background}, \cite{hong2024achieving} made an important contribution to the standard RMAB problem that maximizes the long-run average reward. 
They proved the exponentially fast convergence of the proposed policy to optimality in the asymptotic regime, which again relies on a non-degenerate assumption similar to \cite{gast2024reoptimization,brown2023fluid}. 
For the weakly coupled MDP consisting of heterogeneous groups of MDPs (that is, multi-action restless bandit processes) and a finite time horizon, we prove that, under any Lipschits-limit policy, the MDP in general converges in probability to a deterministic trajectory in $e^{-O(h)}$.
It leads to a range of policies, with appropriately adapted actions, that converge exponentially fast to optimality in the asymptotic regime.

The remainder of the paper is organized as follows.
In Section~\ref{sec:model}, we give a detailed definition of the WCG problem. 
In Section~\ref{sec:convergence}, we present our main results: theorems assure fast convergence of the stochastic process in the size dimension. 
Based on the fast convergence in the size dimension,
in Section~\ref{sec:alp}, 
we propose a class of policies, referred to as ALP, for which the performance degradation diminishes exponentially in the size dimension.
We also proposed a specific algorithm, in the general resource allocation scenario, that constructs an asymptotically optimal policy. 
Numerical validation of the algorithm is provided in Section~\ref{subsec:example} for a system that falls out of the scope of non-degenerate assumptions.
In Section~\ref{sec:conclusions}, we present conclusions.


\section{The Problem}\label{sec:model}

The sets of positive and non-negative integers are denoted by $\mathbb{N}_{+}$ and $\mathbb{N}_{0}$ , respectively, and, for any $N\in\mathbb{N}_{+}$,  $[N]$ represents the set $\{1,2,\ldots,N\}$ with $[0]=\emptyset$.
We use $\mathbb{N}_0^{\infty}$, $\mathbb{N}_+^{\infty}$, and $[N]_0$ to represent the sets $\mathbb{N}_0\cup\{\infty\}$, $\mathbb{N}_+\cup\{\infty\}$, and $\{0\}\cup[N]$, respectively.
Similarly,  $\mathbb{R}$, $\mathbb{R}_{+}$ and $\mathbb{R}_{0}$ denote the  sets of all, positive and non-negative reals, respectively.

Define $\db{\bR}$ as the set of all real-valued random variables.
For a finite set $\bS$, let $\bF(\bS)$ and $\db{\bF}(\bS)$ represent the sets of all real valued functions $\bS \rightarrow \mathbb{R}$ and all functions $\bS\rightarrow \db{\bR}$, respectively.
For any $f\in\bF(\bS)$, we define $\lVert f \rVert \coloneqq \max_{s\in\bS}\lvert f(s) \rvert$.

Let $\bm{v}\circ\bm{u}$, $\innerproduct{\bm{v}}{\bm{u}}$, and $\norm{\bm{v}}$ represent the Hadamard product (element-wise multiplication), inner product, and  the maximum norm, respectively.

\subsection{System Model}\label{subsec:model}
We consider a system consisting of $I$ gangs  of restless multi-action bandit (RB) processes, where the $i$th gang  ($i\in [I]$) has $N_i$ stochastically identical   RB processes; that is,  with identical state and action spaces, transition kernels, and identically distributed reward rates. 
An RB process of class $i\in[I]$ is a discrete-time Markov decision process (MDP) with \emph{finite} state and action spaces $\bS_i$ and $\bA_i$, respectively.
Let $s_{i,n}(t)$ and $\action_{i,n}(t)$ ($i\in[I],n\in[N_i]$) represent the state and action variables, respectively, of the $n$th RB process in gang $i$ at time $t\in\mathbb{N}_0$. 
These are  random variables $s_{i,n}(t)$ and $a_{i,n}(t)$  taking  values in $\bS_i$ and $\bA_i$, respectively.
The full state and actions spaces are  $\sS\coloneqq \prod_{i\in [I]}\bS_{i}^{{N_{i}}} $ and $\sA\coloneqq \prod_{i\in I}\bA_{i}^{{N_{i}}} $, respectively.
We observe that the state and action space $\sS$ and $\sA$ are  time-independent. We use the notation
$\bs(t)\coloneqq (s_{i,n}(t): i\in[I], n\in[N_i])\in \sS$ and $\actionVec(t)\coloneqq (\action_{i,n}(t): i\in[I], n\in[N_i])\in \sA$.
At each time $t\in \mathbb{N}_0$, the system controller chooses  the value of the action vector $\actionVec(t)$ as a function of current state $\bs(t)$ and time index $t$.

For a system of this kind, a \emph{policy} $\phi$ is a sequence of mappings $\phi_{t}:\sS\mapsto \sA$ ($t\in \bm{N}_{0}$). 
Evolution of the system under such a policy requires that the action $\actionVec(t)$ at time $t$ is
\begin{equation*}
  \actionVec(t)=\phi_t(\bs(t))\in \sA.
\end{equation*}

Choice of the action $\actionVec(t)=(\action_{i,n}(t))$ provides the transition matrix  $\mathcal{P}_i(a)=\bigl[p_i(s,a,s')\bigr]_{|\bS_i|\times|\bS_i|}\in[0,1]^{|\bS_i|\times|\bS_{i}|}$ ($a\in\bA_i$),
so that  the  individual RB process $\bigl\{s_{i,n}(t),t\in\mathbb{N}_0\bigr\}$ evolves according to  the transition probability $p_i\bigl(s_{i,n}(t),\action_{i,n}(t),s_{i,n}(t+1)\bigr)$ from $s_{i,n}(t)$ to $s_{i,n}(t+1)$. This state transition generates a real-valued, non-negative and bounded   \emph{random reward}, 
$R_i\in\db{\bF}(\bS_i\times \bA_i)$,
with expectation $r_i\bigl(s,\action\bigr)\coloneqq \mathbb{E}\bigl[R_i(s,a)\bigr]$, where $r_i\in\bF(\bS_i\times\bA_i)$.
We define $R_{i,n}(t) \sim R_i\bigl(s_{i,n}(t),\action_{i,n}(t)\bigr)$ that represents the sampled reward for bandit process $n$ at time $t$, where $\sim$ represents equality in distribution.  
We refer to $r_i$ as the \emph{reward function.}
The reward $R_{i,n}(t)$ is dependent on only $s_{i,n}(t)$ and $a_{i,n}(t)$ but independent of state transitions.

To highlight the dependence of the various objects studied on the choice of  policy $\phi$,  we write $s_{i,n}(t)$, $\action_{i,n}(t)$, and $R_{i,n}(t)$,  as $s^{\phi}_{i,n}(t)$, $\actionphi_{i,n}(t)$ and $R^{\phi}_{i,n}(t)$, respectively. We will adopt a similar notation, as required, for other derived entities in relation to the system.

We consider a finite time horizon $T\in\mathbb{N}_+$ of the entire process $\{\bs^{\phi}(t),t\in[T]_0\}$.
The original  RMAB formulation~\cite{whittle1988restless} imposed a simple form of constraint:
\begin{equation}\label{eqn:constraint:rmab}
\sum_{i\in[I],n\in[N_i]}\actionphi_{i,n}(t) = M,~\forall t\in[T]_0,
\end{equation}
where $M\in[N]$, and $\bA_i$ is specified to be $\{0,1\}$ for all $i\in[I]$.
This  couples the RB processes $\bigl\{s_{i,n}(t),t\in\mathbb{N}_0\bigr\}$, enforcing dependencies among $s^{\phi}_{i,n}(t)$ and $\actionphi_{i,n}(t)$, for all $i\in[I]$ and $n\in[N_i]$.
Let 
\begin{equation}\label{eqn:obj_func}
\Gamma(\bs^0) \coloneqq \bbE^{\phi}_{\bs^0}\Bigl[\sum_{t\in[T]_0}\sum_{i\in[I],n\in[N_i]} R^{\phi}_{i,n}(t)\Bigr],
\end{equation}
where $\bs^0\in \sS$ is the initial state, $\bbE^{\phi}_{\bs^0}$ represents the expectation conditioned on given initial state $\bs^0$ and the employed policy $\phi$,
and, of course, the value of $\Gamma(\bs^0)$ is dependent on the action choices.
In \cite{whittle1988restless,weber1990index,nino2007dynamic,verloop2016asymptotically,ninomora2020verification,fu2018restless,brown2020index}, RMAB techniques  deal with maximization of  the long-run average expected reward, $\lim_{T\rightarrow \infty}\Gamma(\bs^0)/T$,  the expected cumulative reward of the entire system $\Gamma(\bs^0)$ and/or  the long-run discounted expected cumulative reward, where the state and action variables are fully observed and the transition matrices and the reward functions are known a priori  (the offline case).

In \cite{weber1990index,verloop2016asymptotically,fu2018restless,brown2020index,fu2024patrolling}, scalable  policies were proposed for large-scale problems and were proved (under certain conditions) to approach optimality as the problem sizes, usually measured by the number of  RB processes, tend to infinity. In particular, the deviations between the performance of the proposed policies and optimality decrease exponentially or polynomially as the problem size increases.

However, many real-world applications do not fall in the scope of the standard RMAB formulation. For instance, stochastic optimization involving geospatial/topology data~\cite{fu2024patrolling}, the  bipartite-queue matching problem~\cite{fu2025restless}, resource allocation with limited capacities~\cite{fu2018restless}, and hospital care scheduling~\cite{daeth2023optimal}.
In this paper, we consider a set of constraints that generalize \eqref{eqn:constraint:rmab}, 
\begin{equation}\label{eqn:constraint:linear}
   \sum_{i\in[I]}\sum_{n\in[N_i]}f_{i,\ell}\bigl(s^{\phi}_{i,n}(t),\actionphi_{i,n}(t)\bigr)\leq 0,~\forall \ell\in[L], t\in[T]_0,
\end{equation}
where $L\in\mathbb{N}_0$, and $f_{i,\ell}:\bS_i\times\bA_i \rightarrow \mathbb{R}$ is a bounded function.
The  inequality in \eqref{eqn:constraint:linear} is for the sake of simplicity. It is straightforward  to modify the results presented here to the case where we replace some of the  inequalities  \eqref{eqn:constraint:linear} by  equalities; that is, having a mix of equalities and inequalities.

We refer to such a process $\bigl\{\bs^{\phi}(t), t\in[T]_0\bigr\}$, consisting of the $I$ classes of RB processes coupled through \eqref{eqn:constraint:linear}, as
\emph{Weakly Coupled Gangs (WCGs)}.
A WCG is a Markov decision process with state space $\sS=\prod_{i\in[I]}\bS_i^{N_i}$ and, of course,  is at least as difficult as a general RMAB problem, when the time horizon $T < \infty$. 
It is mathematically determined by \[\Bigl\{I,\{N_i\}_{i\in[I]},\{\bS_i\}_{i\in[I]},\{\bA_i\}_{i\in[I]}, \{\mathcal{P}_i(a)\}_{i\in[I],a\in\bA_i},\{R_i\}_{i\in[I]},\{f_{i,\ell}\}_{i\in[I],\ell\in[L]}\Bigr\}.\]

Let $\lPhi$ represent the set of all the policies $\phi$ determined by $\phi_t$, $t\in[T]_0$,  for which recall $\actionphiVec(t) = \phi_t(\bs^{\phi}(t))\in\sA$.
Similarly, let $\PhiWC$ represent the set of all the policies $\phi$ determined by $\phi_t$, $t\in[T]_0$, and constrained by \eqref{eqn:constraint:linear}. 
The superscript `WC' stands for the \emph{weakly coupling constraints} 
in \eqref{eqn:constraint:linear}.
It follows that $\PhiWC\subset \lPhi$.

We are interested in finding  a policy $\phi\in \PhiWC$  that achieves  the objective 
\begin{equation}\label{eqn:obj:long-run-average}
\max_{\phi\in\PhiWC} \Gamma^{\phi}(\bs^0) =  \max_{\phi\in\PhiWC}\sum_{t\in[T]_0}\sum_{i\in[I]}\sum_{n\in[N_i]}\sum_{(s,a)\in\bS_i\times\bA_i}r_i(s,a)\mathbb{P}\bigl\{s^{\phi}_{i,n}(t) = s, a^{\phi}_{i,n}(t) = a\bigl|~\bs^{\phi}(0)=\bs^0\bigr\}.
\end{equation}

\section{Fast Convergence in Size Dimension}
\label{sec:convergence}

\cite{brown2023fluid,gast2023linear,gast2024reoptimization} proposed the LP-based approximation method that casts the same action variables as those for an optimal solution of the relaxed problem, solved through linear programming, and then adapts the actions a little to fit the constraints of the original problem.
Such action variables are applicable to the original problem; meanwhile, the resulting stochastic process is ensured to converge to optimality of the relaxed problem, which is an upper/lower bound of the maximum/minimum of the original problem, in the asymptotic regime.
It leads to asymptotic optimality of the proposed policy.

Unlike \cite{brown2023fluid,gast2023linear,gast2024reoptimization},  we do not request \emph{non-degenerate} assumptions for the optimal solution to the relaxed problem (linear programming problem).
We will demonstrate in this section that the asymptotic optimality holds in general to a range of algorithms applicable to the original WCG problem. 
In particular, the deviations between such algorithms and optimality diminish as exponentially as the size of the problem increases.

\subsection{Asymptotic Regime}
\label{subsec:asym_regime}
We start with a discussion of the asymptotic behaviour of the WCG problem.
Consider the effects of scaling the number of copies $N_i$ of each gang $i\in [I]$ by a \emph{scaling parameter} $h$.

Consider a parameter $h\in\mathbb{N}_+$ such that $N_i = h N_i^0$ for some $N_i^0 \in\mathbb{N}_+$ ($i\in[I]$). 
We refer to $h$ as the \emph{scaling parameter} of the WCG system. 
Note that the scaling parameter $h$ is used to scale  $N_i$ ($i\in[I]$), the number of bandit processes of gang $i$.  All the definitions of the WCG system remain the same.  For an
initial state $\bs^0$, we aim to maximize
\begin{equation}\label{eqn:obj:h}
    \max_{\phi\in\PhiWC} \Gamma^{\phi,h}(\bs^0),
\end{equation}
where 
\begin{equation}
   \Gamma^{\phi,h}(\bs^0)\coloneqq 
        \frac{1}{h}\Gamma^{\phi}(\bs^0), 
\end{equation}
and we write $\Gamma^{\phi, \infty}(\bs^0)=\liminf_{h\to \infty} G^{\phi, \infty}(\bs^0)$.
Here,
the superscript $\phi$ and $h$ are attached to indicate their influence on $\Gamma(\bs^0)$.
For a policy $\phi$, since it is applied on each individual member $i$ of the  WCG,   we can extend it to the new system scaled by $h$ in the obvious way, by applying it to each member of the enlarged WCG. To indicate this we write $\bs^{\phi,h}(t)$ and $\actionphihVec(t)$ for the Markov process that results from the application of that policy to the larger system.
Meanwhile, we also write $\lPhih$ and $\PhiWCh$ to substitute $\lPhi$ and $\PhiWC$, respectively, as both sets are defined with given $h\in\mathbb{N}_+$.

We define $\calJ\coloneqq \prod_{i\in[I]}(\bS_i\times\bA_i)$ as the set of all triples (called \emph{gang-state-action (GSA) triples})   of the form  $(i, s,\action)$ with $i\in [I]$, $s\in \bS_i$ and $\action\in \bA_i$.

For $h\in\mathbb{N}_+$, $\phi\in\lPhih$, $(i,s,\action)\in \mathcal{J}$, and $t\in[T]_0$,
define $Z^{\phi,h}_{i,s,\action}(t)$ as the fraction of bandit processes $\bigl\{s^{\phi,h}_{i,n}(t), t\in[T]_0\bigr\}$ ($i\in[I],n\in[N_i]$) that take the value $i,s,\action$ at time $t$; that is, 
\begin{equation}\label{eqn:define_Z}
Z^{\phi,h}_{i,s,\action}(t)\coloneqq \frac{1}{\sum_{i\in[I]}N_i}\biggl|\Bigl\{n\in[N_{i}]~|~s^{\phi,h}_{i,n}(t) = s,\, \actionphih_{i,n}(t)=\action\Bigr\}\biggr|.
\end{equation}
Now  define $\bm{Z}^{\phi,h}(t) \coloneqq \bigl(Z^{\phi,h}_{i,s,\action}(t):(i,s,\action)\in\mathcal{J}\bigr)$, which takes values in a simplex $\Delta_{\calJ}\coloneqq \bigl\{\bz\in [0,1]^{|\calJ|}\bigl|\sum_{(i,s,\action)\in\calJ}z_{i,s,\action} = 1\bigr\}$.
For the maximization in \eqref{eqn:obj:h}, the bandit processes associated with $(i,n)$  gain the same expected reward $r_i(s,a)$ and transition rates $\mathcal{P}_i(a)$ when they are in the same \GSA triple  $(i,s,a)$, $(s,a)\in\bS_i\times\bA_i$.
In this context, for each time $t$,  the expected total reward $\sum_{i\in[I]}\sum_{n\in[N_i]}\mathbb{E}\bigl[R^{\phi,h}_{i,n}(t)\bigr] = \sum_{(i,s,\action)\in\mathcal{J}}\sum_{i\in[I]}N_i Z^{\phi,h}_{i,s,\action}(t)r_{i}(s,\action)$, where $R^{\phi,h}_{i,n}(t)$ is the instantaneous reward generated by process $\bigl\{s^{\phi,h}_{i,n}(t), t\in[T]_0\bigr\}$ with scaling parameter $h$.

The problem described in \eqref{eqn:obj:h} and \eqref{eqn:constraint:linear} is equivalent to 
\begin{equation}\label{eqn:obj:Z}
\max_{\phi\in\lPhih}\sum_{t=0}^T \sum_{(i,s,\action)\in\mathcal{J}}\bbE^{\phi}_{\bs^0}\Bigl[Z^{\phi,h}_{i,s,\action}(t)\Bigr]r_{i}(s,\action),
\end{equation}
subject to 
\begin{equation}\label{eqn:constraint:linear:Z}
 \sum_{(i,s,\action)\in\mathcal{J}}Z^{\phi,h}_{i,s,\action}(t)f_{i,\ell}(s,\action)\leq 0,~\forall \ell\in[L],t\in[T]_0.
\end{equation}
The initial state of $\bigl\{\bm{Z}^{\phi,h}(t),t\in[T]_0\bigr\}$ satisfies $\sum_{\action\in\bA_i} Z^{\phi,h}_{i,s,\action}(0) = y^0_{i,s}$ for all $i\in[I]$ and $s\in\sS_i$, where $y^0_{i,s}$ is the initial state given by
\begin{equation}\label{eqn:define:upsilon}
    y^0_{i,s}\coloneqq \frac{1}{\sum_{i\in[I]}N_i}\Bigl|\bigl\{n\in[N_i]~\bigl|~s^{\phi,h}_{i,n}(0) = s\bigr\}\Bigr|.
\end{equation}
Let $\bm{y}^0\coloneqq (y^0_{i,s}:i\in[I],s\in\sS_i)$, which takes values in a simplex 
$$\deltay\coloneqq \bigl\{\bm{y}\in\mathbb{R}_0^{\sum_{i\in[I]}|\bS_i|}\bigl| \sum_{i\in[I],s\in\bS_i}y_{i,s} = 1\bigr\}.$$
We can rewrite $\Gamma^{\phi,h}(\bs^0)$ as  
\begin{equation}\label{eqn:definition:gamma}
    \Gamma^{\phi,h}(\bm{y}^0)\coloneqq \sum_{i\in[I]}N^0_i\sum_{t=0}^T\sum_{(i,s,\action)\in\mathcal{J}}\bbE^{\phi}_{\bm{y}^0}\Bigl[Z^{\phi,h}_{i,s,\action}(t)\Bigr]r_{i}(s,\action),
\end{equation}
where the initial condition $\bm{y}^0$ satisfies \eqref{eqn:define:upsilon} by plugging in $\bs^{\phi,h}(0)=\bs^0$, and $\bbE^{\phi}_{\bm{y}^0}$ is the expectation with initial condition $\bm{y}^0$.
We say that the process $\bigl\{\bs^{\phi,h}(t),t\in[T]_0\bigr\}$ can be interpreted by $\bigl\{\bm{Z}^{\phi,h}(t),t\in[T]_0\bigr\}$ with initial condition $\bm{y}^0$, and refer to the limit case $h\rightarrow \infty$ as the \emph{asymptotic regime}.

\subsection{Fast Convergence Towards Equilibrium in Size Dimension }

\label{subsec:asym_opt}

The WCG problem is a special MDP for which the sizes of the state and action spaces, $\sS$ and $\sA$, respectively, exponentially increases in the scaling parameter $h\in\mathbb{N}_+$.
Recall, in this paper, the WCG problem is at least as hard as the finite-time-horizon RMAB problem, which is considered to be already hard with intractable optimal solutions~\cite{whittle1988restless,weber1990index,papadimitriou1999complexity}. 
We resort to effective algorithms that are applicable in large-scale WCG systems (large $h$) and, meanwhile, exhibit theoretically bounded performance degradations.

Recall that a policy $\phi\in\lPhih$ is a sequence of mappings $\phi_t:\sS\mapsto \sA $ ($t\in[T]_0$), for which the domain $\sS$ changes as $h$ changes.
We define a family of policies indexed by $h$: 
$\orule \coloneqq \bigl\{\phi^h\in \lPhih: h=1,2,\ldots\bigr\}$, and a set $\lPsi$ of all such $\orule$.
Typically, in the past work, the Whittle index policy~\cite{whittle1988restless}, index-style policies~\cite{fu2016asymptotic,fu2020energy,fu2018restless,fu2024restless}, and most linear-programming-based policies~\cite{fu2024patrolling,gast2023linear,gast2024reoptimization,brown2023fluid} are policy families with some relationship between different $\phi^h$ ($h\in\mathbb{N}_+$).
For any $\orule=(\phi^h:h\in\mathbb{N}_+)\in\lPsi$, we write $\bZ^{\orule,h}(t)\coloneqq \bZ^{\phi^h,h}(t)$, $\bs^{\orule,h}(t)\coloneqq \bs^{\phi^h,h}(t)$, and $\actionVec^{\orule,h}(t)\coloneqq \actionVec^{\phi^h,h}(t)$, and let $\bbE^{\orule}_{\bm{y}^0}$ represent the expectation operation $\bbE^{\phi^h}_{\bm{y}_0}$.

Unlike the past work that focuses on one or a few specified policy families, here, we will consider a set of policy families, which include most of those specified in the past work, and provide sufficient conditions for how to construct policy families (which were referred to as ``policies" in the past work) that are asymptotically optimal.
In this paper, we do not provide conditions for the structure of the WCG problem but, given the problem, we condition on the structure of the designed policy families.

For $\phi\in\lPhih$, we define $\bm{Y}^{\phi,h}(t)\coloneqq (Y^{\phi,h}_{i,s}(t):i\in[I],s\in\bS_i)$, where $Y^{\phi,h}_{i,s}(t)$ is the proportion of bandit processes in the gang-state pair $(i,s)$ at time $t$; that is,
\begin{equation}\label{eqn:define:upsilon}
    Y^{\phi,h}_{i,s}(t) \coloneqq \frac{1}{\sum_{i'\in[I]}N_{i'}}\Bigl\lvert\bigl\{n\in[N_i]~\Bigl|~s^{\phi,h}_{i,n}(t)=s\bigr\} \Bigr\rvert = \sum_{\action\in\bA_i}Z^{\phi,h}_{i,s,\action}(t).
\end{equation}
Similarly, we write $\bm{Y}^{\orule,h}(t)\coloneqq \bm{Y}^{\phi^h,h}(t)$.

Given $h\in\mathbb{N}_+$,  $\bm{Y}^{\orule,h}(t)$ only takes values in $\bm{y}\in\deltay$ such that $h\sum_{i\in[I]}N^0_i \bm{y}$ is a vector of non-negative integers. 
That is, $\bm{Y}^{\varphi,h}(t)$ only takes finitely many possible values 
in $\deltay$. 
Let $\deltayh$ represent such a finite set of all possible values of $\bm{Y}^{\varphi,h}(t)$; that is,
\begin{equation}
    \deltayh \coloneqq \Bigl\{\bm{y}\in\deltay~\Bigl|~\sum_{i\in[I]}N_i \bm{y}\in \mathbb{N}_0^{\sum_{i\in[i]}|\bS_i|}\Bigr\}.
\end{equation}
In this case, for any $\bm{y}\in\deltayh$, an element $y_{i,s}$ can possibly be $0,\frac{1}{\sum_{i\in[I]}N_i},\frac{2}{\sum_{i\in[I]}N_i},\ldots,1$.

For $\orule\in\lPsi$, $h\in\mathbb{N}_+$ and $(i,s,\action)\in\calJ$, we define a function $\alpha^{\orule,h}_{i,s,\action}:\deltay\times [T]_0 \mapsto [0,1]$, which satisfies, for $t\in[T]_0$ and $\bm{y}\in\deltayh$, 
\begin{equation}\label{eqn:define:alpha}
        \alpha^{\orule,h}_{i,s,\action}(\bm{y},t) = \begin{cases}
           \frac{ \mathbb{E}\Bigl[Z^{\orule,h}_{i,s,\action}(t)\Bigl| \bm{Y}^{\orule,h}(t)=\bm{y}\Bigr]}{y_{i,s}} , & \text{if } y_{i,s} > 0,\\
            \min_{\begin{subarray}~i\in[I],s\in\bS_i:\\y_{i,s}\geq \frac{1}{\sum_{i'\in[I]}N_{i'}}\end{subarray}}\alpha^{\orule,h}_{i,s,\action}\Bigl(\bm{y}+\bv_{i,s}\bigl(\frac{1}{\sum_{i'\in[I]}N_{i'}}\bigr)-\bv_{i,s}\bigl(\frac{1}{\sum_{i'\in[I]}N_{i'}}\bigr),t\Bigr), & \text{otherwise},
        \end{cases}
\end{equation}
where $\bv_{i,s}: \bbR\mapsto \bbR^{\sum_{i\in[I]}|\bS_i|}$ such that 
all elements of $\bv_{i,s}(x)$ are zero except the $(i,s)$th one equal to $x$, and if $\bm{y}\in\deltayh$, then $\bm{y}+\bv_{i,s}\bigl(\frac{1}{\sum_{i\in[I]}N_i}\bigr)-\bv_{i,s}\bigl(\frac{1}{\sum_{i\in[I]}N_i}\bigr)$ is also in $\deltayh$. 
The first line in \eqref{eqn:define:alpha} defines $\alpha^{\orule,h}_{i,s,\action}$ for all $\bm{y}\in \deltayh$ with $y_{i,s}>0$.
Then, the second line in \eqref{eqn:define:alpha} can define  $\alpha^{\orule,h}_{i,s,\action}$ for all $\bm{y}\in \deltayh$ with $y_{i,s}=0$.  

Equation~\eqref{eqn:define:alpha} specifies the value of $\alpha^{\orule,h}_{i,s,\action}(\bm{y},t)$ for any $\bm{y}\in\deltayh$. 
For $\bm{y}\notin\deltayh$,  through any linear interpolation over the given points $\bigl(\bm{y},\alpha^{\orule,h}_{i,s,\action}(\bm{y},t)\bigr)\in\deltayh\times [0,1]$, we can construct the function $\alpha^{\orule,h}_{i,s,\action}(\bm{y},t)$ that is Lipschitz continuous in all $\bm{y}\in\deltay$, as shown in the following proposition.

\begin{proposition}\label{prop:stable_cond:1}
Given $h\in\mathbb{N}_+$, for any $\varphi\in\lPsi$ and $(i,s,\action)\in\calJ$, there exists $\alpha^{\orule,h}_{i,s,\action}(\bm{y},t)$ that is Lipschitz continuous in $\bm{y}\in\deltay$ and satisfies \eqref{eqn:define:alpha} for all $\bm{y}\in\deltayh$.
\end{proposition}
The proof of Proposition~\ref{prop:stable_cond:1} is provided in Appendix~\ref{app:linear_interpolation}.

Let $\bm{\alpha}^{\orule,h}(\bm{y},t) \coloneqq (\alpha^{\orule,h}_{i,s,\action}(\bm{y},t):(i,s,\action)\in\calJ)$.
We consider the following condition.

\begin{condition}{Lipschitz-limit regularity (in Size Dimension)}\label{cond:weak_stab}
For a policy family $\orule = (\phi^h: h\in\mathbb{N}_+)\in\lPsi$, 
any $\bm{y}\in\deltay$, $(i,s,\action)\in\calJ$, and $t\in[T]$,
\begin{equation}
 \bm{\alpha}^{\orule}_{i,s,\action}(\bm{y},t)\coloneqq \lim_{h\to \infty}\bm{\alpha}^{\orule,h}_{i,s,\action}(\bm{y},t),
\end{equation} 
exists and is Lipschitz continuous in $\bm{y}$.
\end{condition} 

We refer to a policy family $\orule\in\lPsi$ satisfying~\partialref{cond:weak_stab}{Lipschitz–limit regularity} as a \emph{Lipschitz-limit} family of policies, 
and define $\PsiZ$ as the set of all Lipschitz-limit policy families $\orule\in\lPsi$.
Clearly, $\PsiZ\subset \lPsi$.
We also define a subset $\PsiZWC\subset \PsiZ$ as the set of policy families such that, for any $\orule=(\phi^h:h\in\mathbb{N}_+)\in\PsiZ$, $\phi^h\in\PhiWC$ for all $h\in\mathbb{N}_+$. 

Typically, Whittle index policy (for RMAB case) and LP-based policies~\cite{whittle1988restless,brown2023fluid,gast2023linear} are Lipschitz-limit policy families in $\PsiZWC$. 
We provide proofs for the Lipschitz-limit regularity of these policies in Appendix~\ref{app:stability}.

Here, we start with theorems on a general property of the WCG process: for all Lipschitz-limit policy families $\orule\in\PsiZ$, the stochastic process $\bigl\{\bm{Z}^{\orule,h}(t), t\in[T]_0\bigr\}$ quickly converges to its averaging trajectory $\bigl\{\bbE^{\orule}_{\bm{y}^0}\bm{Z}^{\orule,h}(t), t\in[T]_0\bigr\}$ as $h\rightarrow \infty$. 
Let $\bm{z}^{\orule,h}(t)\coloneqq \bbE^{\orule}_{\bm{y}^0}\bm{Z}^{\orule,h}(t)$.
Such a convergence will later lead to a class of algorithms, proposed and discussed in Section~\ref{sec:alp}, that are scalable to WCG systems with arbitrarily large $h$ and quickly approach optimality as $h\rightarrow\infty$.

\begin{theorem}\label{theorem:convergence-Z}
For any $\orule\in\PsiZ$, $\epsilon>0$, and initial condition $y^0_{i,s}=\sum_{\action\in\bA_i}Z^{\orule,h}_{i,s,\action}(0)$ ($i\in[I]$ and $s\in\bS_i$), 
\begin{enumerate}
    \item the limit $\bm{z}^{\orule}(t)\coloneqq \lim_{h\rightarrow \infty}\bm{z}^{\orule,h}(t)=\lim_{h\rightarrow \infty}\bbE^{\orule}_{\bm{y}^0}\bm{Z}^{\orule,h}(t)$ exists for all $t\in[T]_0$; and
\item $\bm{Z}^{\orule,h}(t)$ converges to $\bm{z}^{\orule}(t)$ in probability for all $t\in[T]_0$; that is, 
\begin{equation}\label{eqn:theorem:convergence-Z}
    \lim_{h\rightarrow\infty} \mathbb{P}\Bigl\{\max_{t=0,1,\ldots,T}\bigl\lVert\bm{Z}^{\orule,h}(t) - \bm{z}^{\orule}(t)\bigr\rVert > \epsilon\Bigr\} = 0.
\end{equation}
\end{enumerate}
\end{theorem}
The proof of Theorem~\ref{theorem:convergence-Z} is provided in Appendix~\ref{app:theorem:convergence-Z}.

Theorem~\ref{theorem:convergence-Z} indicates that, for any policy family $\orule\in\PsiZ$, increasing the size of the WCG problem, measured by $h$, the process $\bigl\{\bm{Z}^{\orule,h}(t), t\in[T]_0\bigr\}$ converges to the deterministic averaging trajectory $\bigl\{\bm{z}^{\orule}(t),t\in[T]_0\bigr\}$ in probability.
More importantly,  based on Freidlin's theorems \cite[Theorem 4.1, Chapter 7]{freidlin2012random}, 
we prove that the convergence speed (in probability) for Lipschitz-limit policy families $\orule\in\PsiZ$ is exponential in $h$.

Theorem~\ref{theorem:convergence-Z} implies that $\Gamma^{\orule,\infty}(\bm{y}^0) \coloneqq \lim_{h\rightarrow\infty}\Gamma^{\orule,h}(\bm{y}^0)$ exists for any $\orule\in\PsiZ$ and initial condition $\bm{y}^0\in \deltay$. We then discuss the convergence speed in the following theorem.

\begin{theorem}\label{theorem:convergence_Z_exp}
For any $\orule\in\PsiZ$,  $\epsilon>0$, and initial condition $y^0_{i,s}=\sum_{\action\in\bA_i}Z^{\orule,h}_{i,s,\action}(0)$ ($i\in[I]$ and $s\in\bS_i$), there exist positive constants $C,H<\infty$ such that, for all $h>H$,
\begin{equation}\label{eqn:theorem:convergence_Z_exp}
    \mathbb{P}\Bigl\{\max_{t=0,1,\ldots,T}\bigl\lVert\bm{Z}^{\orule,h}(t) - \bm{z}^{\orule}(t)\bigr\rVert > \epsilon\Bigr\} \leq e^{-Ch}.
\end{equation}
\end{theorem}
The proof of Theorem~\ref{theorem:convergence_Z_exp} is provided in Appendix~\ref{app:theorem:convergence_Z_exp}.

Theorem~\ref{theorem:convergence_Z_exp} strengthens Theorem~\ref{theorem:convergence-Z}, and, for any Lipschitz-limit $\orule\in\PsiZ$, it ensures a fast convergence speed to the asymptotic regime with respect to $\bm{Z}^{\orule,h}(t)$. We refer to such a convergence as $h\rightarrow \infty$ as the convergence in \emph{size dimension}.
For any Lipschitz-limit $\orule\in\PsiZ$, such exponential convergence in probability does not request any intrinsic properties over the WCG problem, such as the non-degenerate assumptions in~\cite{brown2023fluid,gast2023linear,gast2024reoptimization}.
Recall that $\PsiZ$ is the set of all Lipschitz-limit policy families, including the Whittle index policy~\cite{whittle1988restless} and those proposed in~\cite{brown2023fluid,gast2023linear}, that satisfy a fairly general condition about the action variables in the asymptotic regime (see \partialref{cond:weak_stab}{Lipschitz–limit regularity}).

From Theorem~\ref{theorem:convergence_Z_exp}, 
for any $\orule\in\PsiZ$, it approaches optimality in probability at rate $e^{-O(h)}$ - asymptotic optimality with exponential rate - \emph{if and only if} the deterministic process $\{\bm{z}^{\orule}(t),t\in[T]\}$ coincides with the optimum of the original WCG problem.
We are interested in the questions: how to construct such an asymptotically optimal $\varphi$ that is scalable to the WCG system with large $h$, and how tightly that the convergence of the WCG process affects the overall performance of the potential algorithms?
We will provide a detailed discussion on constructing efficient and effective algorithms in Section~\ref{sec:alp}.

\section{Fluid Approximation}
\label{sec:alp}

\subsection{Asymptotic Optimality in Size Dimension: Exponentially Diminishing suboptimality}\label{subsec:alp:asym_opt_exp}
We follow the general methodology  of \cite{whittle1988restless}: we randomize the policy $\phi\in\lPhih$.  
We will use $\db{\sA}$ to indicate the space of randomized actions; that is, the space of random variables with values in $\sA$. 

For a randomized policy $\phi$ and each time $t\in[T]_0$, $\actionphihVec(t)\coloneqq(\actionphih_{i,n}(t):i\in[I],n\in[N_i])$ is a random action in $\db{\sA}$, with distribution $\bbP\bigl[\actionphih_{i,n}(t)~\bigl|~\bs^{\phi,h}(t)\bigr]$.
Define 
a  simplex  $\Delta_{\bA_i}\coloneqq \bigl\{\balpha: \bA_{i}\to \bbR_{0}|\sum_{a\in\bA_i}\alpha_a = 1\bigr\}$.
Naturally the probability $\bbP\bigl[\actionphih_{i,n}(t)~\bigl|~\bs^{\phi,h}(t)\bigr]$ takes values in $\Delta_{\bA_i}$.
For given $t\in[T]_0$, the random action $\actionphihVec(t)$ and the action probability $\bbP\bigl[\actionphih_{i,n}(t)~\bigl|~\bs^{\phi,h}(t)\bigr]$ are functions of 
current state:
$\sS\rightarrow \db{\sA}$ and $\sS \rightarrow \bigl(\Delta_{\bA_i}\bigr)^{\sum_{i\in[I]}|\bS_i|}$, respectively, where $\db{\sA}$ is the space of the random actions.
Note that such functions are potentially different for different $t\in[T]_0$.
We define $\dblPhih$ as the set of all the policies $\phi$ determined by such random actions $\actionphihVec(t)$ for all $t\in[T]_0$.

We consider the problem, 
\begin{equation}\label{eqn:objective:finite-time:h}
   \sup_{\phi\in\dblPhih} \frac{1}{h}\Gamma^{\phi,h}(\bs^0)=\sum_{i\in[I]}N^0_i\sup_{\phi\in\dblPhih}\sum_{t=0}^T \sum_{(i,s,\action)\in \mathcal{J}}\bbE^{\phi}_{\bm{y}^0}\Bigl[Z^{\phi,h}_{i,s,\action}(t)\Bigr]r_{i}(s,\action),
\end{equation}
subject to 
\begin{equation}\label{eqn:constraint:relax:h}
\sum_{(i,s,\action)\in\mathcal{J}}\bbE^{\phi}_{\bm{y}^0}\Bigl[Z^{\phi,h}_{i,s,\action}(t)\Bigr] f_{i,\ell}(s,\action)\leq0
,~\forall \ell\in[L], t\in[T]_0,
\end{equation}
which is derived by taking expectation operation on both sides of \eqref{eqn:constraint:linear}.

The problem \eqref{eqn:objective:finite-time:h}
subject to \eqref{eqn:constraint:relax:h} is a relaxed version of the problem described in \eqref{eqn:obj:h} and \eqref{eqn:constraint:linear} (or equivalently, the problem described in \eqref{eqn:obj:Z} and \eqref{eqn:constraint:linear:Z}).
We rewrite the relaxed problem described in \eqref{eqn:objective:finite-time:h} and \eqref{eqn:constraint:relax:h} in a linear programming form, 
\begin{equation}\label{eqn:obj:linear programming}
 \Gamma^* \coloneqq \max_{\bm{x}\in[0,1]^{|\mathcal{J}|(T+1)}}\sum_{t\in[T]_0}\sum_{(i,s,\action)\in\mathcal{J}} r_{i}(s,\action)x_{t}(i,s,\action)N_i^0,
\end{equation}
subject to
\begin{eqnarray}
\label{eqn:constraint:linear programming:1}&
\sum_{s'\in\bS_i,\action'\in\bA_i}x_{t}(i,s',\action') p_i(s',\action',s) = \sum_{\action\in\bA_i}x_{t+1}(i,s,\action),&\forall i\in[I],s\in\bS_i,t\in[T-1]_0\\
\label{eqn:constraint:linear programming:2}&\sum_{s\in\bS_i,\action\in\bA_i}x_{t}(i,s,\action) = 1, &~\forall i\in[I],t\in[T]_0,\\
\label{eqn:constraint:linear programming:3}&\sum_{\action\in\bA_i}x_{0}(i,s,\action) = y_{i,s}^0, &~\forall i\in[I],s\in\bS_i,\\
&\sum_{(i,s,\action)\in\mathcal{J}}N_i^0 x_{t}(i,s,\action)f_{i,\ell}(s,\action) \leq 0,&~\forall \ell\in[L],t\in[T]_0,
\label{eqn:constraint:linear programming:4}
\end{eqnarray}
where initial condition $\bm{y}^0\in\deltay$ is given for all $i\in[I]$ and $s\in\bS_i$. 
Constraints~\eqref{eqn:constraint:linear programming:1}-\eqref{eqn:constraint:linear programming:3} ensure that $x_{t}(i,s,\action)=\frac{\sum_{i\in[I]}N^0_i}{N^0_{i}}\bbE^{\phi}_{\bm{y}^0}Z^{\phi,h}_{i,s,\action}(t)$ ($(i,s,\action)\in\mathcal{J}$) for some $\phi\in\dblPhih$, representing the probability that the process $\{s^{\phi,h}_{i,n}(t),t\in\mathbb{N}_0\}$, for any $n\in[N_{i}]$, stays in \GSA triple $(i,s,\action)$.
Note that such a linear problem (relaxed problem) is different from the original WCG problem and is used as an intermediate step to achieve near-optimal policies for the original WCG problem.

For any given $h\in\mathbb{N}_+$ and $\bm{x}\in [0,1]^{|\mathcal{J}|(T+1)}$ satisfying~\eqref{eqn:constraint:linear programming:1}-\eqref{eqn:constraint:linear programming:3}, we can construct a corresponding policy $\phi\in\dblPhih$ by setting the action probability for $(i,s,\action)\in\mathcal{J}$ and $t\in[T]_0$,
\begin{equation}\label{eqn:linear programming:3}
    \bbP\bigl[\actionphih_{i,n}(t)= \action~\bigl|~s^{\phi,h}_{i,n}(t)=s\bigr]
    =\oalphaLP_{i,s,\action}(\bm{x},t) \coloneqq \begin{cases}
        \frac{x_{t}(i,s,\action)}{\sum_{\action'\in\bA_i}x_{t}(i,s,\action')},& \text{if }\sum_{\action'\in\bA_i}x_{t}(i,s,\action') > 0,\\
        \frac{1}{|\bA_i|},&\text{otherwise},
        \end{cases}
\end{equation}
for any $n\in [N_{i}]$ due to the stochastic identity of all bandit processes in the same gang-state pair. 
For the function $\oalphaLP_{i,s,\action}$ defined in \eqref{eqn:linear programming:3}, the overline and superscript of $\oalphaLP_{i,s,\action}$ indicates that it represents the action probability applicable to the relaxed problem (described in \eqref{eqn:objective:finite-time:h} and \eqref{eqn:constraint:relax:h}), and is based on a feasible solution to the linear programming problem (that is, its first argument $\bm{x}$ should satisfies \eqref{eqn:constraint:linear programming:1}-\eqref{eqn:constraint:linear programming:3}), respectively. 
For a given $\bm{x}\in [0,1]^{|\mathcal{J}|(T+1)}$ satisfying \eqref{eqn:constraint:linear programming:1}-\eqref{eqn:constraint:linear programming:3}, we write $\phi(\bm{x})$ as a policy determined by the action probabilities 
$\bbP\bigl[\action^{\phi(\bm{x}),h}_{i,n}(t)= \action~\bigl|~s^{\phi(\bm{x}),h}_{i,n}(t) = s\bigr]=\oalphaLP_{i,s,\action}(\bm{x},t)$, for all $(i,s,\action)\in\mathcal{J}$ and $t\in[T]_0$.
Let $\boalphaLP(\bm{x},t) \coloneqq (\oalphaLP_{i,s,\action}(\bm{x},t) : (i,s,\action)\in\mathcal{J})$.
Since the action probability $\oalphaLP_{i,s,\action}(\bm{x},t)$ is independent to $h$, $\bbP\bigl[\bZ^{\phi(\bm{x}),h}(t+1)~\bigl|~\bZ^{\phi(\bm{x}),h}(t)\bigr]$ is also independent to $h$.
There exists $\bm{z}^{\phi(\bm{x})}(t) \coloneqq  \bm{z}^{\phi(\bm{x}),h}(t) = \bbE^{\phi}_{\bm{y}^0}\bigl[\bm{Z}^{\phi(\bm{x}),h}(t)\bigr]$ independent to $h$.
In this context, we can construct a policy family $\dborule(\bm{x}) \coloneqq (\phi^h: h \in\mathbb{N}_+)$ consisting of a constant sequence $\phi^h = \phi(\bm{x})$ for all $h\in\mathbb{N}_+$.

\begin{proposition}\label{prop:asym_opt:LP}
For any $\orule\in\PsiZ$ and $\bm{x}\in[0,1]^{|\mathcal{J}|(T+1)}$ satisfying \eqref{eqn:constraint:linear programming:1}-\eqref{eqn:constraint:linear programming:3} and $\bm{Y}^{\orule,h}(0)=\bm{y}^0$,  
\begin{equation}\label{eqn:assumption:action}
\lim_{h\rightarrow \infty} \Bigl\lvert \bigl(\alpha^{\orule,h}_{i,s,\action}(\bm{Y}^{\orule,h}(t),t)- \oalphaLP_{i,s,\action}(\bm{x},t)\bigr)Y^{\orule,h}_{i,s}(t)\bigr)\Bigr\rvert =0,~\forall (i,s,\action)\in\calJ,t\in[T]_0,
\end{equation}
if and only if, for any $\epsilon>0$, there exist $C,C_1,H<\infty$ and $C_2>0$ such that, for all $h>H$, 
\begin{equation}\label{eqn:prop:asym_opt:LP:1}
 \bbP\Bigl\{\max_{t\in[T]_0} \bigl\lVert \bm{Z}^{\orule,h}(t)-\bm{z}^{\dborule(\bm{x})}(t)\rVert >\epsilon\Bigr\}\leq e^{-C h},
\end{equation}
and
\begin{equation}\label{eqn:prop:asym_opt:LP:2}
 \max_{t\in[T]_0} \bigl\lVert \bm{z}^{\orule,h}(t)-\bm{z}^{\dborule(\bm{x})}(t)\rVert \leq C_1e^{-C_2h} + \epsilon,
\end{equation}
where recall $\bm{z}^{\orule,h}(t) = \bbE^{\orule}_{\bm{y}^0}[\bm{Z}^{\orule,h}(t)]$. 
\end{proposition}
The proof of Proposition~\ref{prop:asym_opt:LP} is based on Theorem~\ref{theorem:convergence_Z_exp} and is provided in Appendix~\ref{app:asym_opt:LP}. 
For any policy $\phi\in\dblPhih$, the expected total reward $\frac{1}{h}\Gamma^{\phi,h}(\bs^0) = \sum_{t\in[T]_0}\bm{r}\cdot \bm{z}^{\phi,h}(t)$, where $\bs^{\phi,h}(0)=\bs^0$ is the initial state, and
recall $\bm{r} = (r_{i}(s,\action): (i,s,\action)\in\mathcal{J})$.
Based on Proposition~\ref{prop:asym_opt:LP}, if we propose a policy family $\orule\in\PsiZWC$ such that \eqref{eqn:assumption:action} and the original WCG constraints~\eqref{eqn:constraint:linear} are both satisfied ($\orule$ is applicable to the original WCG problem for all $h\in\mathbb{N}_+$), then $\orule$ approaches the same objective function as $\dborule(\bm{x})$ in the asymptotic regime, where the performance deviation exponentially diminishing in probability in $h$.

Moreover, given $\bm{x}=\bm{x}^*$ optimal to the linear programming problem \eqref{eqn:obj:linear programming}-\eqref{eqn:constraint:linear programming:4}, we construct the corresponding optimal policy $\phi(\bm{x}^*)\in\dblPhih$ through \eqref{eqn:linear programming:3}.
Apparently, such $\phi(\bm{x}^*)$ is not necessarily applicable to the original WCG problem, because it may not satisfy \eqref{eqn:constraint:linear}.
From Proposition~\ref{prop:asym_opt:LP}, if there is a policy family $\orule\in\PsiZWC$ such that \eqref{eqn:assumption:action} (plugged in $\bm{x}=\bm{x}^*$) and the original WCG constraints~\eqref{eqn:constraint:linear} are both satisfied, 
$\orule$ is asymptotically optimal with performance deviation exponentially diminishing in probability in $h$.
Because $\orule$ approaches optimality of the relaxed problem, which is an upper bound of the maximum of the original WCG problem described in \eqref{eqn:obj:h} and \eqref{eqn:objective:finite-time:h}.

For such asymptotically optimal $\orule\in\PsiZWC$, condition \eqref{eqn:assumption:action} requests convergence between $\alpha^{\orule,h}_{i,s,\action}(\bm{Y}^{\orule,h}(t), t)$ and $\oalphaLP_{i,s,\action}(\bm{x}^*,t)$  in the asymptotic regime, $h\rightarrow \infty$, whenever $\lim_{h\to\infty}Y^{\orule,h}_{i,s}(t)>0$.

Unlike the past work~\cite{brown2020index,brown2023fluid,fu2024patrolling,gast2023linear,gast2024reoptimization} that focus on specified policies, condition~\eqref{eqn:assumption:action} defines a range of policy families $\orule$  (referred to as `policies' in the past work) that can be asymptotically optimal. 
The specified policies discussed in the past work, which are derived through Whittle relaxation (or equivalently, the linear programming version), are mostly included in the range of policy families constrained by \eqref{eqn:assumption:action}.
Examples for such $\orule$ include the \emph{Lagrangian index policy}~\cite{brown2020index} and the \emph{movement-adapted index policy}~\cite{fu2024patrolling} for the RMAB case.
Also, \cite{gast2023linear} and \cite{brown2023fluid} discussed a problem similar to the WCG problem except that the maximization was conducted over randomized policies in $\dblPhih$. That is, the problem  maximizes \eqref{eqn:objective:finite-time:h} subject to \eqref{eqn:constraint:linear}.
Similarly, for such a maximization over all policies in $\dblPhih$, based on Proposition~\ref{prop:asym_opt:LP}, if a policy $\phi\in\dblPhih$ satisfies the original constraints \eqref{eqn:constraint:linear} and \eqref{eqn:assumption:action} with plugged in $\bm{x}=\bm{x}^*$, then $\phi$ approaches optimality in the asymptotic regime withe performance suboptimality vanishes at rate $O(e^{-h})$.
The \emph{water filling policy}~\cite{gast2023linear} for the RMAB case, and the \emph{fluid-budget balancing policy}~\cite{brown2023fluid} for the general WCG case are examples of such asymptotically optimal $\phi\in\dblPhih$ that satisfy \eqref{eqn:assumption:action}.

We refer to those $\orule\in\PsiZWC$ satisfying \eqref{eqn:constraint:linear} and \eqref{eqn:assumption:action} as the \emph{fluid approximation} (FA) policy families.

Proposition~\ref{prop:asym_opt:LP} proves that FA approaches optimality in probability at rate $e^{-O(h)}$.
When the WCG reduces to a standard RMAB problem, assuming $\bm{x}^*$ is a non-degenerate solution to the linear programming problem, \cite{gast2023linear} provided similar results with exponential convergence between a linear programming-based policy and optimality in the asymptotic regime.
For the general case, under slightly different non-degenerate assumptions, 
\cite{brown2023fluid,gast2024reoptimization} proved that their policies are asymptotically optimal with performance suboptimality $O(e^{-h})$ and  $e^{-O(h)}$, respectively. 
Here, Proposition~\ref{prop:asym_opt:LP} has no request on the form of $\bm{x}^*$ (not necessary to be non-degenerate) and is applicable for general WCG and every BA satisfying \eqref{eqn:assumption:action}.
The condition~\eqref{eqn:assumption:action} is more straightforward and is satisfied by a wider range of policies, including those proposed in \cite{gast2023linear,brown2023fluid}. 
For the general WCG without non-degenerate assumptions, we will further propose an FA policy family satisfying~\eqref{eqn:assumption:action} in Section~\ref{subsec:example}, which, based on Proposition~\ref{prop:asym_opt:LP}, is asymptotically optimal in probability with rate $e^{-O(h)}$.


The asymptotic optimality in the size dimension is marvelous for large-scale WCG systems where the problem size, measured by the scaling parameter $h$, is large. 
If the system is small, then conventional learning and optimization techniques, such as value iteration and Q learning, can achieve solutions without requesting excessive large amount of computational and storage resources. 
When the system becomes large, conventional techniques exhibit the curse of dimensionality with respect to learning and control.
Optimal solutions become intractable and we resort good approximations of optimality.
The WCG model and the BA policy families are applicable to a wide range of application scenarios, including the standard RMAB case~\cite{whittle1988restless,weber1990index,nino2001restless,verloop2016asymptotically,fu2019towards,fu2020energy,brown2020index} and \cite{gast2023linear} and extended restless-bandit-based problems~\cite{fu2018restless,brown2023fluid,fu2024patrolling}.

\subsection{An Algorithm for Fluid Approximation}\label{subsec:example}

In this subsection, we consider a resource allocation scenario specified by the following assumption.
This rather general scenario is the same as the problem model discussed in \cite{brown2023fluid,gast2024reoptimization}.
\begin{assumption}\label{assumption:feasibility}
\begin{enumerate}[label=(\arabic*)]
\item For each $i\in[I]$ and $s\in\bS_i$, there exists $a_i(s)\in\bA_i$ such that, for all $\ell\in[L]$,
$f_{i,\ell}(s,a_i(s))  = \min_{a\in\bA_i}f_{i,\ell}(s,a)$.
\item For any $\bs^{\phi}(t')=\bs\in\sS$,  \eqref{eqn:constraint:linear} is satisfied by plugging in $\action^{\phi}_{i,n}(t) = a_i(s^{\phi}_{i,n}(t))$ for all $i\in[I]$ and $n\in[N_i]$.
\end{enumerate}
\end{assumption}

Note that 
\cite{brown2023fluid} and \cite{gast2024reoptimization} assumed Assumption~\ref{assumption:feasibility} for their model, under which they further request non-degenerate assumptions for achieving exponentially diminishing sub-optimality.

Assumption~\ref{assumption:feasibility} ensures that, with appropriate initial state $\bs^0$, there always exist feasible actions at time $t\in[T]_0$.
It avoids the case where $\PhiWC = \emptyset$.
We need Assumption~\ref{assumption:feasibility} for the feasibility of the proposed algorithm discussed in only Section~\ref{subsec:example} (this subsection), no other place in this paper. 
The main theoretical results, Theorems~\ref{theorem:convergence-Z} and \ref{theorem:convergence_Z_exp} and Proposition~\ref{prop:asym_opt:LP}, do not rely on Assumption~\ref{assumption:feasibility}.

Given Assumption~\ref{assumption:feasibility}, we propose an algorithm that constructs an BA family discussed in Section~\ref{subsec:alp:asym_opt_exp} that satisfy both \eqref{eqn:constraint:linear} and \eqref{eqn:assumption:action}. 
The algorithm is constructed in a similar manner as that of the fluid-budget balancing policy in~\cite{brown2023fluid}. 
Nonetheless, since this paper does not assume the existence of the only ``saturated" constraint $\ell^*\in[L]$ in \eqref{eqn:constraint:linear} (non-degenerate assumption), in the more general case, we provide steps in the proposed algorithm with precise discussions about how to select appropriate alternative actions to fit the constraints in \eqref{eqn:constraint:linear}.

Given $\bm{x}^*$ optimal to the linear programming problem \eqref{eqn:obj:linear programming}-\eqref{eqn:constraint:linear programming:4}, we obtain $\boalphaLP(\bm{x}^*,t)$, through \eqref{eqn:linear programming:3}, for all $t\in[T]_0$. 
Consider a policy $\psi\in\lPhih$, for each $t\in[T]_0$ and observed $\bS^{\psi,h}(t)$, through the following steps.
\begin{enumerate}[label=\roman*)]
\item\label{step:1} Based on the observed $\bS^{\psi,h}(t)$, keep $Y^{\psi,h}_{i,s}(t)$ updated, and, for all $(i,s,\action)\in\mathcal{J}$, set $Z^{\psi,h}_{i,s,\action}(t)$ to be \[\frac{\bigl\lfloor \oalphaLP_{i,s,\action}(\bm{x}^*,t) Y^{\psi,h}_{i,s}(t)h\sum_{i\in[I]}N^0_i\bigr\rfloor}{h\sum_{i\in[I]}N^0_i}.\]
\item \label{step:2}
For all $i\in[I]$ and $s\in\bS_i$, pick up an action $a^0(i,s)\in\bA_i$ and the corresponding \GSA triple $(i,s,a^0(i,s))\in\mathcal{J}$, and set $Z^{\psi,h}_{i,s,a^0(i,s)}(t)$ to be 
\[Y^{\psi,h}_{i,s}(t) - \sum_{\action\in\bA_i:\action\neq a^0(i,s)}Z^{\psi,h}_{i,s,\action}(t).\]
Such $a^0(i,s)$ can be any action in $\bA_i$.
\item \label{step:3}
Explore all the constraints in \eqref{eqn:constraint:linear}. Let $\mathscr{L}(t)$ represent the set of all $\ell\in[L]$ with
\begin{equation}\label{eqn:define_LHS}
    L_{\ell}(t)\coloneqq \sum_{(i,s,\action)\in\mathcal{J}}N^0_{i}Z^{\psi,h}_{i,s,\action}(t)f_{i,\ell}(s,\action) > 0.
\end{equation}
If $\mathscr{L}(t) \neq \emptyset$,
then we order all the \GSA triples $(i,s,\action)\in\mathcal{J}$ according to the ascending order of $\oalphaLP_{i,s,\action}(\bm{x}^*,t)$, denoting the $j$th \GSA triple as $\bigl((i(j),s(j),\action(j)\bigr)$ for $j\in[|\mathcal{J}|]$,  and go to Step~\ref{step:4}; otherwise, jump to Step~\ref{step:done_ell}.
\item \label{step:4} 
Explore all $\ell\in\mathscr{L}(t)$, for each of which we initialize $j=1$ and conduct Step~\ref{step:5}.
\item \label{step:5}
The operations in this step are for a fixed $\ell\in\mathscr{L}(t)$ and a fixed $j\in[|\mathcal{J}|]$. 
If $Z^{\psi,h}_{i(j),s(j),\action(j)}(t) =0$, then keep incrementing $j$ one by one until $Z^{\psi,h}_{i(j),s(j),\action(j)}(t) >0$.
For $(i,s,\action)\in\calJ$ and $\ell\in[L]$, we define 
\begin{equation}\label{eqn:define:M}
    M_{\ell}(i,s,\action)\coloneqq \bigl\lvert \bigl\{\action\in\bA_i~\bigl|~f_{i,\ell}(s,\action)<\action\bigr\}\bigr\rvert +1.
\end{equation}
For the fixed $j$ and $\ell$, we find all the actions 
$a_1,a_2,\ldots,a_{M}\in\bA_{i(j)}$ such that 
$M=M_{\ell}(i(j),s(j),\action(j))$,
$a_{M} = \action(j)$, $a_1 = a^0(i(j),s(j))$, and, for all $\ell'\in[L]$,
\begin{multline}\label{eqn:descending_actions}
f_{i(j),\ell'}(s(j),a_1) \leq 
f_{i(j),\ell'}(s(j),a_2) \leq\ldots \leq f_{i(j),\ell'}(s(j),a_{M-1}) \\< f_{i(j),\ell'}(s(j),a_{M}) = f_{i(j),\ell'}(s(j),\action(j)).
\end{multline}
Based on Assumption~\ref{assumption:feasibility}, such $M=M_{\ell}(i(j),s(j),\action(j))\geq 2$.
Decrementing $m$ from $M-1$, we iteratively 
\begin{itemize}
    \item update $\Delta f_m \coloneqq f_{i(j),\ell}\bigl(s(j),a_{m+1}\bigr)-f_{i(j),\ell}\bigl(s(j),a_m\bigr)$;
    \item update
\begin{equation}\label{eqn:define_deltaZ}
\Delta Z_{m,\ell}\bigl(i(j),s(j),\action(j)\bigr)\coloneqq \begin{cases}
    \Bigl\lceil\frac{L_{\ell}(t)h\sum_{i\in[I]}N^0_i}{N^0_{i(j)}\Delta f_m}\Bigr\rceil\frac{1}{h\sum_{i\in[I]}N^0_i},&\text{if }\Delta f_m > 0,\\
    0, &\text{otherwise};
    \end{cases}
\end{equation}
    \item update $L_{\ell}(t)$ with 
    \begin{equation}\label{eqn:update_LHS}
    L_{\ell}(t) - N_i^0\Delta f_m\min\Bigl\{Z^{\psi,h}_{i(j),s(j),a_{m+1}},\Delta Z_{m,\ell}\bigl(i(j),s(j),\action(j)\bigr)\Bigr\};
    \end{equation}
    \item and then update $Z^{\psi,h}_{i(j),s(j),a_m}$ and $Z^{\psi,h}_{i(j),s(j),a_{m+1}}$ with $$Z^{\psi,h}_{i(j),s(j),a_m}+\min\Bigl\{Z^{\psi,h}_{i(j),s(j),a_{m+1}},\Delta Z_{m,\ell}\bigl(i(j),s(j),\action(j)\bigr)\Bigr\},$$ and $$Z^{\psi,h}_{i(j),s(j),a_{m+1}}-\min\Bigl\{Z^{\psi,h}_{i(j),s(j),a_{m+1}},\Delta Z_{m,\ell}\bigl(i(j),s(j),\action(j)\bigr)\Bigr\},$$ respectively;
\end{itemize}
until $L_{\ell}(t) \leq 0$ or $m=1$. 
If $L_{\ell}(t) \leq 0 $, then go back to Step~\ref{step:4} and explore another not-yet-explored $\ell$ in set $\mathscr{L}(t)$.
If all $\ell\in\mathscr{L}(t)$ have been explored, go to Step~\ref{step:done_ell}.
If $L_{\ell}(t) > 0$ (that is, $m=1$), then find the smallest $j'>j$ such that $Z^{\psi,h}_{i(j), s(j),\action(j)}(t)>0$, assign the value of $j'$ to $j$, and go back to the beginning of Step~\ref{step:5} with the same $\ell$.
\item \label{step:done_ell}
Based on the updated $\bm{Z}^{\psi,h}(t)$, determine the action variables $\bm{a}^{\psi,h}(t)$.
\end{enumerate}
We refer to the $\psi$ output through Steps~\ref{step:1}-\ref{step:done_ell} as the \emph{fluid-approximation balancing (FAB)} algorithm.
We also provide pseudo-code of \ALP in Algorithm~\ref{algo:alp}.
In this paper, we use $\psi$ to specifically represent the policy led by \ALP.

\IncMargin{1em}
\begin{algorithm}
\small 
\linespread{0.5}\selectfont

\SetKwFunction{FALP}{ALPExample}
\SetKwProg{Fn}{Function}{:}{\KwRet}
\SetKwInOut{Input}{Input}\SetKwInOut{Output}{Output}
\SetAlgoLined
\DontPrintSemicolon
\Input{Given $t\in[T]_0$, $\bS^{\psi,h}(t)$, and $\bm{x}^*$ optimal to \eqref{eqn:obj:linear programming}-\eqref{eqn:constraint:linear programming:4}.}
\Output{Action vector $\actionpsihVec(t)$ or, equivalently, $\bm{Z}^{\psi,h}(t)$.}
\Fn{\FALP{}}{
    Update $\bm{Y}^{\psi,h}(t)$ based on $\bS^{\psi,h}(t)$\;
    Update $Z^{\psi,h}_{i,s,\action}(t)\gets \frac{\Bigl\lfloor \oalphaLP_{i,s,\action}(\bm{x}^*,t)Y^{\psi,h}_{i,s}(t)h\sum_{i\in[I]}N^0_i\Bigr\rfloor}{h\sum_{i\in[I]}N^0_i}$ for all $(i,s,\action)\in\mathcal{J}$\label{line:alp_step}\;
    Pick up an action $a^0(i,s)\in\bA_i$ for all $i\in[I]$ and $s\in\bS_i$\;
    For all $i\in[I]$ and $s\in\bS_i$, update
    $Z^{\psi,h}_{i,s,a^0(i,s)}(t) \gets Y^{\psi,h}_{i,s}(t) - \sum_{\action\in \bA_i:\action\neq a^0(i,s)}Z^{\psi,h}_{i,s,\action}(t)$\;
    \label{line:adaption:1}Order all \GSA triples $(i,s,\action)\in\mathcal{J}$ according to the ascending order of $\oalphaLP_{i,s,\action}(\bm{x}^*,t)$, and label the $j$th \GSA triple as $\bigl(i(j),s(j),\action(j)\bigr)$ for $j\in[|\mathcal{J}|]$.\;
    \For{$\ell\in[L]$\label{line:adaption:2}}{
        \If{$L_{\ell}(t) > 0$}{\tcc*{$L_{\ell}(t)$ is defined in \eqref{eqn:define_LHS}}
            Initialize $j=1$\;
            \While {$j < |\mathcal{J}|$\label{line:middle:1}}{
                $M\gets M_{\ell}\bigl(i(j),s(j),\action(j)\bigr)$\;
                Find all $a_1,a_2,\ldots,a_M\in\bA_{i(j)}$ that satisfy \eqref{eqn:descending_actions}.\;
                Initialize $m=M-1$\;
                \While{$m \geq 1$ and $L_{\ell}(t) > 0$\label{line:adaption:inner:1}}{
                    Update $\Delta Z_{m,\ell}\bigl(i(j),s(j),\action(j)\bigr)$ based on \eqref{eqn:define_deltaZ}\;
                    Update $L_{\ell}(t)$ according to \eqref{eqn:update_LHS}\;
                    $Z^{\psi,h}_{i(j),s(j),a_m}(t)\gets Z^{\psi,h}_{i(j),s(j),a_m}(t) + \min\Bigl\{Z^{\psi,h}_{i(j),s(j),a_{m+1}}(t),\Delta Z_{m,\ell}\bigl(i(j),s(j),\action(j)\bigr)\Bigr\}$\;
                    $Z^{\psi,h}_{i(j),s(j),a_{m+1}}(t) \gets Z^{\psi,h}_{i(j),s(j),a_{m+1}}(t) -\min\Bigl\{Z^{\psi,h}_{i(j),s(j),a_{m+1}}(t),\Delta Z_{m,\ell}\bigl(i(j),s(j),\action(j)\bigr)\Bigr\}$\;
                \label{line:adaption:inner:2}}
                \If{$L_{\ell}(t) \leq 0$}{
                    {\bf Break}, go to Line~\ref{break:point}.\;
                }
                Fine the smallest $j'>j$ such that $Z^{\psi,h}_{i(j),s(j),\action(j)}(t)>0$, and $j\gets j'$\;
            }
            \label{break:point}
        }
    }\label{line:adaption:done-1}
    Determine the action vector $\actionpsihVec(t)$ based on the updated $\bm{Z}^{\psi,h}(t)$.\label{line:adaption:done}\;
}
\caption{The decisions of the \ALP algorithm at each time $t$.}\label{algo:alp}
\end{algorithm}
 \DecMargin{1em}

Step~\ref{step:1} aligns the action ratio $Z^{\psi,h}_{i,s,\action}(t)/Y^{\psi,h}_{i,s}(t)$ to be reasonably close to $\oalphaLP_{i,s,\action}(\bm{x}^*,t)$ for all $(i,s,\action)\in\mathcal{J}$ except those with $\action = a^0(i,s)$.
Step~\ref{step:2} ensures that all the bandit processes have been assigned exactly an action; that is, for all $i\in[I]$ and $s\in\bS_i$,
$\sum_{\action\in\bA_i}Z^{\psi,h}_{i,s,\action}(t) = Y^{\psi,h}_{i,s}(t)$.
Steps~\ref{step:1}-\ref{step:2} construct a $\psi$ that is in $\lPhih$ but not necessarily satisfies \eqref{eqn:constraint:linear}.
We refer to such a policy led by Steps~\ref{step:1}-\ref{step:2} as the \emph{linear programming-Approximated (LP-Approx)} policy.

Steps~\ref{step:3}-\ref{step:done_ell} continue adapting the LP-Approx policy to achieve \eqref{eqn:constraint:linear}.
In Steps~\ref{step:5}, in each of the iterations from $m=M-1$ to $1$ (or reach the stop condition), we keep replacing action $a_{m+1}(j,t)$ with less expensive $a_m(j,t)$ so that the summand on the left hand side of \eqref{eqn:constraint:linear} decreases for the given $\ell\in[L]$.
Based on Assumption~\ref{assumption:feasibility}, such replacement will not cause increment on the left hand side of \eqref{eqn:constraint:linear} with other $\ell'\in[L]$.
Also, due to Assumption~\ref{assumption:feasibility}, there must exist some sufficiently ``cheap" actions such that  the iterations in Step~\ref{step:5} achieve the stop condition: $L_{\ell} \leq 0$ for all $\ell\in[L]$.
After Step~\ref{step:done_ell}, the output policy $\psi\in\PhiWCh$ satisfies \eqref{eqn:constraint:linear}.
We refer to the process from Step~\ref{step:3} to \ref{step:done_ell} as the \emph{decision adaption} process and refer to the output policy $\psi$ as the \emph{\ALP} policy.

The LP-Approx policy satisfies \eqref{eqn:assumption:action} but, in general, does not align with \eqref{eqn:constraint:linear}.
For any given $h\in\mathbb{N}_+$, the decision adaption process in Steps~\ref{step:3}-\ref{step:done_ell} usually drives the actions away from \eqref{eqn:assumption:action} but it ensures that \eqref{eqn:constraint:linear} is satisfied and the resulting \ALP is feasible to the original WCG problem described in \eqref{eqn:obj:h} and \eqref{eqn:constraint:linear}.
Fortunately, the proportion of the actions that are changed in the decision adaption process vanishes at rate the asymptotic regime.
In Section~\ref{subsubsec:alp:asym_opt}, we will provide a detailed discussion about how \ALP satisfies \eqref{eqn:assumption:action} in the asymptotic regime and, based on Proposition~\ref{prop:asym_opt:LP}, is asymptotically optimal.

Without Assumption~\ref{assumption:feasibility}, the decision adaption process is not guaranteed to reach actions satisfying \eqref{eqn:constraint:linear}. 
In this case, based on the features of different WCG problems, different decision adaption processes can be considered to reach a feasible policy such that \eqref{eqn:constraint:linear} is always satisfied and the proportion of adapted actions becomes negligible in the asymptotic regime.
For instance, a special case of the WCG problem was discussed in \cite{fu2024patrolling} where Assumption~\ref{assumption:feasibility} is not satisfied.
\cite{fu2024patrolling} proposed an asymptotically optimal algorithm through a similar way: construct an LP-Approx policy and then adapt its actions to achieve \eqref{eqn:constraint:linear} and \eqref{eqn:assumption:action}. It led to a policy satisfying \eqref{eqn:assumption:action} and, based on Proposition~\ref{prop:asym_opt:LP}, is asymptotically optimal.
Proposition~\ref{prop:asym_opt:LP} applies to general WCG, providing a guideline for constructing asymptotically optimal policies.

\subsubsection{Computational Complexity}
To achieve the LP-Approx policy, Steps~\ref{step:1}-\ref{step:2}  take computational complexity $O(|\mathcal{J}|)$.

Steps~\ref{step:3}-\ref{step:done_ell} (decision adaption) correspond to Lines~\ref{line:adaption:1}-\ref{line:adaption:done} in Algorithm~\ref{algo:alp}. 
In Line~\ref{line:adaption:1}, ordering all the \GSA triples has complexity $O(|\mathcal{J}|\log |\mathcal{J}|)$.
From Line~\ref{line:adaption:2} to \ref{line:adaption:done-1}, it includes three layers of loops. 
The inner loop from Line~\ref{line:adaption:inner:1} to \ref{line:adaption:inner:2} in Algorithm~\ref{algo:alp}, iterates $O\Bigl(M_{\ell}\bigl(i(j),s(j),\action(j)\bigr)\Bigr)= O(|\bA_{i(j)}|)$ times, each for which has complexity $O(1)$.
The middle loop, corresponding to Step~\ref{step:5}, in Lines~\ref{line:middle:1}-\ref{break:point}, has worst case complexity $O(\sum_{(i,s,\action)\in\mathcal{J}}|\bA_i|)$.
Then, we obtain the complexity $O(L\sum_{(i,s,\action)\in\mathcal{J}}|\bA_i|)$ for all the three loops.
The computational complexity for the decision adaption is 
$O(|\mathcal{J}|\log|\mathcal{J}|+L\sum_{(i,s,\action)\in\mathcal{J}}|\bA_i|)$, which is the main contributor to the complexity of the entire algorithm.

\subsubsection{Convergence to Asymptotic Optimality}\label{subsubsec:alp:asym_opt}

For all $h\in\mathbb{N}_+$, the \ALP policies form a family of policies, which we refer to as  the \ALP policy family and still represented by $\psi$ with a bit abuse of notation.

\begin{lemma}\label{lemma:action_ALP_satisfied}
For the \ALP policy family, represented by $\psi$,    \eqref{eqn:assumption:action} holds by plugging in $\orule = \psi$ and $\bm{x} = \bm{x}^*$.
\end{lemma}
The proof of Lemma~\ref{lemma:action_ALP_satisfied} is provided in Appendix~\ref{app:lemma:action_ALP_satisfied}.

\begin{lemma}\label{lemma:ALP_stable}
    The \ALP policy family, represented by $\psi$, is Lipschitz-limit; that is, it satisfies \partialref{cond:weak_stab}{Lipschitz–limit regularity} by plugging in $\orule = \psi$.
\end{lemma}
The proof of Lemma~\ref{lemma:ALP_stable} is provided in Appendix~\ref{app:lemma:ALP_stable}.

\begin{proposition}
    The \ALP policy family is asymptotically optimal that satisfies \eqref{eqn:prop:asym_opt:LP:1} and \eqref{eqn:prop:asym_opt:LP:2} by plugging in $\orule=\psi$ and $\bm{x}=\bm{x}^*$.
\end{proposition}
\proof{Proof.}
It is a direct result of Proposition~\ref{prop:asym_opt:LP} and Lemmas~\ref{lemma:action_ALP_satisfied} and \ref{lemma:ALP_stable}.

\endproof

\subsubsection{Simulation (Numerical Example)}\label{susubbsec:simulation}
\begin{figure}[t]
\centering
\includegraphics[width=0.7\linewidth]{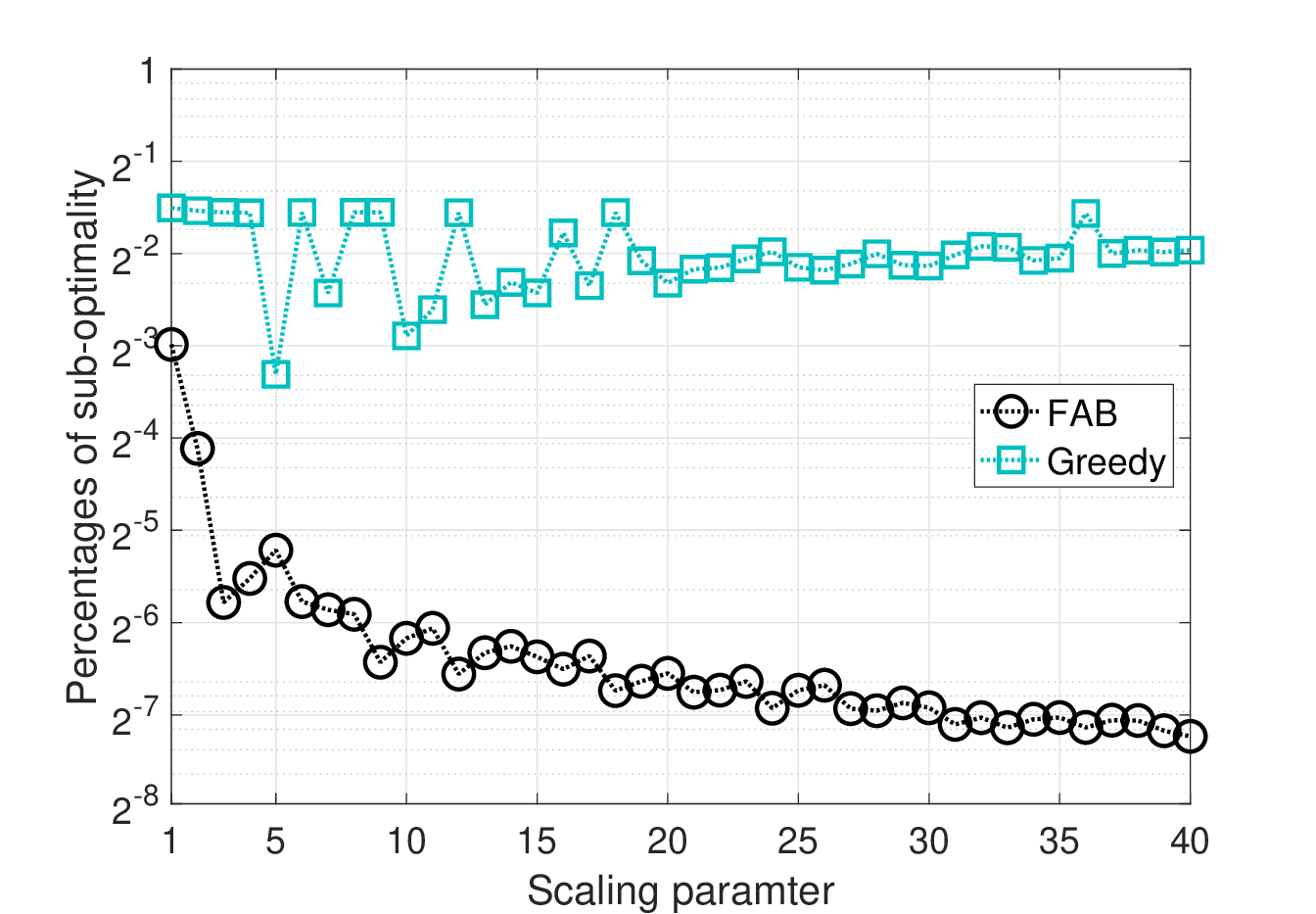}
\caption{Performance deviation of \ALP and Greedy (baseline) to the upper bound of optimality. \label{fig:case-study}}
\end{figure}

We plot in Figure~\ref{fig:case-study} the simulated suboptimality of \ALP against the scaling parameter $h$ (alternatively, the size dimension), where the y-axis is the relative difference between optimality to \eqref{eqn:obj:linear programming}-\eqref{eqn:constraint:linear programming:4} and a given policy $\phi\in\lPhih$,
\[\frac{\Gamma^* - \Gamma^{\phi,h}(\bm{y}^0)}{\Gamma^*}, \]
with given initial state $\bm{y}^0$.
In particular, in Figure~\ref{fig:case-study}, we tested two policies, \ALP and a greedy approach.
The greedy approach is a baseline policy that always selects the \GSA triples $(i,s,\action)$ with the highest $N^0_i r_i(s,\action)$ (the coefficient of $x_{t}(i,s,\action)$ in the objective function \eqref{eqn:obj:linear programming}).
Its actions are adapted through the same decision adaption process as in Steps~\ref{step:3}-\ref{step:done_ell} and hence are feasible to the original WCG problem \eqref{eqn:obj:h} and \eqref{eqn:constraint:linear}.

The simulations presented in Figure~\ref{fig:case-study} considers a WCG system with $I=5$ gangs, time horizon $T=30$, and the number of constraints $L=6$.
The detailed settings are provided in Appendix~\ref{app:simulation:settings}.
It satisfies Assumption~\ref{assumption:feasibility} and \eqref{eqn:assumption:action}, but falls out the scope of the non-degenerate properties requested in the past work~\cite{gast2023linear,brown2023fluid}.

In Figure~\ref{fig:case-study}, we can observe \ALP quickly approaches optimality $\Gamma^*$ (upper bound of maximum of the original WCG problem). The percentage of suboptimality is already less than $1\%$ when $h\geq 23$, for which the confidence intervals of the simulated results are within $\pm 2\%$ of the observed mean.
It consists with Proposition~\ref{prop:asym_opt:LP}.

Such a system simulates a discrete-time, finite-time-horizon version of the resource allocation problem discussed in \cite{fu2018restless}.
There are five resource pools each for which has a finite capacity of resource units that can be used by four services. 
The customers keep arriving into the system seeking to be served by one of the services. 
The controller decides how to assign the arrived customers to the four services or directly block them.
Unlike the long-run case discussed in \cite{fu2018restless}, here, we consider time-varying arrival rates (expected arrival numbers), for which the customers arrive in bulk for each time slot.
When $h\rightarrow \infty$, the expected arrival rates and the service capacities of the resource pools tend to infinity proportionately.
We model the four services and the arrival process of the customers through $I=5$ gangs.

Each service, once serving a customer, will simultaneously occupy numbers of resource units from (part of) the five pools. 
Hence, the number of customers being served by the each service is limited by those of the other services and the capacities of the resource pools, imposing five capacity constraints over the numbers of used resource units at each time $t$.
At each time $t$, there is a service-dependent probability for each customer to finish its service.
When a customer finishes its service, it leaves the system and the corresponding resource units will be immediately released and be reused by future customer-service pairs.
By serving each customer, the system earns money; meanwhile, it is charged at different cost rates by using resource units from different pools. 
Apart from the five constraints for the resource pool capacities, there is one more constraint requesting that the number of customers newly assigned to the services at time $t$ is no larger than the number of new arrivals in the same time slot. 
We refer to Appendix~\ref{app:simulation:settings} for a detailed description of the system settings.

Recall that such a system satisfies Assumption~\ref{assumption:feasibility} and \eqref{eqn:assumption:action} but falls out the scope of the non-degenerate properties requested in the past work~\cite{gast2023linear,brown2023fluid}.
Assumption~\ref{assumption:feasibility} ensures the feasibility of the \ALP algorithm which is asymptotically optimal with exponentially diminishing suboptimality as $h\rightarrow \infty$ - the numerical results in Figure~\ref{fig:case-study} is consistent with Proposition~\ref{prop:asym_opt:LP}. 
Nonetheless, Proposition~\ref{prop:asym_opt:LP} is not limited to Assumption~\ref{assumption:feasibility}, but applicable to general WCG processes. For instance, the multi-agent patrolling problem discussed in \cite{fu2024patrolling} is a special WCG problem which does not satisfy Assumption~\ref{assumption:feasibility}.
The movement-adapted index policy proposed there satisfies \eqref{eqn:assumption:action}, and hence Proposition~\ref{prop:asym_opt:LP} directly applies.

\section{Conclusions}\label{sec:conclusions}

We have proposed a fairly broad class of policies that are specified through their well-defined behaviors and Lipschitz continuity of the action variables in the asymptotic regime.
We have proved that, under any Lipschitz-limit policy family, the WCG process converges to a deterministic process - its averaging process. 
The deviation between the stochastic and the deterministic processes diminishes in probability at rate $e^{-O(h)}$ where $h$ is the scaling parameter representing the size dimension of the WCG system.
Note that the convergence holds in general for all Lipschitz-limit policies without assuming non-degenerate conditions.
It follows with a sufficient and necessary condition for constructing a Lipschitz-limit policy family that is asymptotically optimal - the cumulative reward of the corresponding deterministic process coincides with that of an optimal solution to the original WCG problem.
If a Lipschitz-limit policy family is asymptotically optimal in the size dimension, then its performance deviation, from optimality, vanishes in probability at rate $e^{-O(h)}$.

We have provided a necessary and sufficient condition, related to the form of the action variables, for a Lipschitz-limit policy family to converge in probability (at rate $e^{-O(h)}$) to the optimum of the Whittle relaxation of the WCG problem. 
Recall that, under fairly general scenarios, the relaxed optimum coincides with that of the original WCG problem; in this case, the provided condition is also necessary and sufficient for a Lipschitz-limit policy family to be asymptotically optimal.
Indeed, in the same model setting as in~\cite{brown2023fluid,gast2024reoptimization}, we have proposed the \ALP algorithm that constructs an FA with computational complexity logarithmic-linear in the size dimension (at each time). The algorithm is scalable for systems with large $h$.
We have proved that its performance suboptimality diminishes in probability at rate $e^{-O(h)}$.
Through numerical simulations, which do not satisfy any non-degenerate condition, we have demonstrated the fast convergence of the \ALP algorithm to optimality.

\appendices
\section{Proof of Proposition~\ref{prop:stable_cond:1}}
\label{app:linear_interpolation}

\proof{Proof of Proposition~\ref{prop:stable_cond:1}.}
For some $i_0\in[I]$ and $s_0\in\bS_{i_0}$ and any $\bm{y}\in \deltay\backslash\deltayh$, we select $\sum_{i\in[I]}|\bS_i|$ points $\bm{y}_{i,s}(\bm{y})\in \deltayh$ such that
\begin{equation}
    \bm{y}_{i_0,s_0}(\bm{y}) \coloneqq \sum_{\begin{subarray}~i'\in[I], s'\in \bS_{i'}:\\ (i',s')\neq (i_0,s_0)\end{subarray}} \bv_{i',s'}\Bigl(\frac{\bigl\lfloor h\sum_{i''\in[I]}N^0_{i''} y_{i',s'}\bigr\rfloor}{h\sum_{i''\in[I]}N^0_{i''}}\Bigr) + \bv_{i_0,s_0}\Bigl(1-\sum_{\begin{subarray}~i'\in[I], s'\in \bS_{i'}:\\ (i',s')\neq (i_0,s_0)\end{subarray}} \frac{\bigl\lfloor h\sum_{i''\in[I]}N^0_{i''} y_{i',s'}\bigr\rfloor}{h\sum_{i''\in[I]}N^0_{i''}}\Bigr),
\end{equation}
where recall $\bv_{i,s}(x)$ is a vector of size $\sum_{i\in[I]}|\bS_i|$ with all zero elements except the $(i,s)$th one equal to $x$,
and, for other $(i,s)\in \prod_{i'\in[I]}\bS_{i'}\backslash \{(i_0,s_0)\}$,
\begin{multline}
    \bm{y}_{i,s}(\bm{y}) \coloneqq
        \bv_{i,s}\Bigl(\frac{\bigl\lfloor h\sum_{i'\in[I]}N^0_{i'} y_{i,s} + 1\bigr\rfloor}{h\sum_{i'\in[I]}N^0_{i'}}\Bigr) +\sum_{\begin{subarray}~i'\in[I], s'\in \bS_{i'}:\\ (i',s')\neq (i_0,s_0),\\(i',s')\neq (i,s)\end{subarray}} \bv_{i',s'}\Bigl(y_{i_0,s_0,i',s'}(\bm{y})\Bigr) \\+ \bv_{i_0,s_0}\Bigl(1-\sum_{\begin{subarray}~i'\in[I], s'\in \bS_{i'}:\\ (i',s')\neq (i_0,s_0),\\(i',s')\neq (i,s)\end{subarray}}y_{i_0,s_0,i',s'}(\bm{y}) -  \frac{\bigl\lfloor h\sum_{i'\in[I]}N^0_{i'} y_{i,s} + 1\bigr\rfloor}{h\sum_{i'\in[I]}N^0_{i'}}\Bigr),
\end{multline}
where $y_{i,s,i',s'}(\bm{y})$ is the $(i',s')$th element of $\bm{y}_{i,s}(\bm{y})$.
For all $(i,s)\in \prod_{i'\in[I]}\bS_{i'}\backslash \{(i_0,s_0)\}$, 
\begin{equation}
    \bm{y}_{i,s}(\bm{y})-\bm{y}_{i_0,s_0}(\bm{y}) = \bv_{i,s}(\frac{1}{h\sum_{i'\in[I]}N^0_{i'}}) + \bv_{i_0,s_0}(-\frac{1}{h\sum_{i'\in[I]}N^0_{i'}}),
\end{equation}
are linearly independent.
For any $\bm{y}\in\deltay$, there exist unique $\lambda_{i,s}(\bm{y})\in \bbR_0$ for $(i,s)\in \prod_{i'\in[I]}\bS_{i'}\backslash \{(i_0,s_0)\}$ such that
\begin{equation}
    \bm{y}-\bm{y}_{i_0,s_0}(\bm{y}) = \sum_{(i,s)\in \prod_{i'\in[I]}\bS_{i'}\backslash \{(i_0,s_0)\}} \lambda_{i,s}(\bm{y}) (\bm{y}_{i,s}(\bm{y})-\bm{y}_{i_0,s_0}(\bm{y})).
\end{equation}
Let $\lambda_{i_0,s_0}(\bm{y})\coloneqq 1- \sum_{(i,s)\in \prod_{i\in[I]}\bS_i\backslash \{(i_0,s_0)\}}\lambda_{i,s}(\bm{y})$.

For any $\bm{y}\in\deltay\backslash \deltayh$, we define
\begin{equation}\label{eqn:linear_interpolation}
    \alpha^{\orule,h}_{i,s,\action}(\bm{y},t)\coloneqq \sum_{(i',s')\in \prod_{i''\in[I]}\bS_{i''}} \lambda_{i',s'}(\bm{y}) \alpha^{\orule,h}_{i,s,\action}(\bm{y}_{i',s'}(\bm{y}),t).
\end{equation}
This is a linear interpolation since $\sum_{(i,s)\in \prod_{i'\in[I]}\bS_{i'}} \lambda_{i,s}(\bm{y})=1$.
Also, for any $\bm{y}'\in\deltayh$,
\begin{equation}
    \lim_{\bm{y}\to \bm{y}'} \alpha^{\orule,h}_{i,s,\action}(\bm{y},t) = \alpha^{\orule,h}_{i,s,\action}(\bm{y}',t).
\end{equation}
Such $\alpha^{\orule,h}_{i,s,\action}(\bm{y},t)$ is Lipschitz continuous in $\bm{y}\in\deltay$.

\endproof

\section{Lipschitz-Limit Regularity of Classic Policy Families (Policies)}
\label{app:stability}

Here, we will prove the Lipschitz-limit regularity of Whittle index policy~\cite{whittle1988restless}, relaxed fluid-budget balancing and fluid-budget balancing policies~\cite{brown2023fluid}, and the locally linear policies~\cite{gast2023linear}.

\subsection{RMAB case}
We start with the Whittle index policy and the locally linear policies, as both types of policies are defined when WCG reduces to the RMAB case.
In particular, in this subsection, we consider the special case where $I=1$, $\bA_1=\{0,1\}$, $L=1$, and $f_{i,1}(s,a) = a - \frac{M}{N_1}$ for $M< N_1$ - the WCG becomes standard RMAB.

\begin{proposition}\label{prop:stable:whittle}
The Whittle index policy, as a policy family, if exists, satisfies \partialref{cond:weak_stab}{Lipschitz–limit regularity}.
\end{proposition}
\proof{Proof of Proposition~\ref{prop:stable:whittle}.}
Based on the definition of Whittle index policy~\cite{whittle1988restless}, given the Whittle indices, we rank all the states $s\in\bS_1$ according to the descending order of their Whittle indices.
Tie case is broken arbitrarily.
We use $s_{\iota}$ to represent the $\iota$th state in the ranking.
Given $\bm{Y}^{\text{Whittle},h}(t)=\bm{y}$, the proportion of active bandit processes in state $s\in\bS_1$ (with ranking $\iota$) is 
\begin{equation}\label{eqn:prop:stable:whittle:1}
    \alpha^{\text{Whittle},h}_{1,s,1}(\bm{y},t) = \min\Bigl\{y_{1,s}, \max\bigl\{\frac{M}{N_1} - \sum_{\iota'= 1}^{\iota-1}\alpha^{\text{Whittle},h}_{1,s_{\iota'},1}(\bm{y},t)y_{1,s_{\iota'}},0\bigr\}\Bigr\},
\end{equation}
where, for the first state $s_0$,
$$
\alpha^{\text{Whittle},h}_{1,s_0,1}(\bm{y},t) = \min\Bigl\{y_{(1,s)}, \frac{M}{N_1}\Bigr\}.
$$
The passive proportion of state $s\in\bS_1$ is 
$$ \alpha^{\text{Whittle},h}_{1,s,0}(\bm{y},t) = 1-  \alpha^{\text{Whittle},h}_{1,s,1}(\bm{y},t).$$
In this context, the limit $\lim_{h\to\infty} \alpha^{\text{Whittle},h}_{i,s,\action}(\bm{y},t)$ exists for all $(i,s,\action)\in\calJ$.

By observation, $\lim_{h\to\infty}\alpha^{\text{Whittle},h}_{1,s_0,1}(\bm{y},t)$ is Lipschitz continuous in $\bm{y}$. 
Now we assume that for all $\iota' = 1,2,\ldots,\iota-1$, $\lim_{h\to\infty}\alpha^{\text{Whittle},h}_{1,s_{\iota'},1}(\bm{y},t)$ are Lipschitz continuous in $\bm{y}$.
From \eqref{eqn:prop:stable:whittle:1}, $\alpha^{\text{Whittle},h}_{1,s_{\iota},1}(\bm{y},t)$ is also Lipschitz continuous in $\bm{y}$.
Hence, $\lim_{h\to \infty}\alpha^{\text{Whittle},h}_{1,s,1}(\bm{y},t)$ is Lipschitz continuous in $\bm{y}$ for all $s\in\bS_1$.
It leads to the Lipschitz continuity of $\lim_{h\to \infty}\alpha^{\text{Whittle},h}_{1,s,0}(\bm{y},t)$.
We prove the proporsition.

\endproof

The locally linear policies in \cite{gast2023linear} are defined to satisfy that $y_{1,s}\alpha^{\text{LL},h}_{i,s,\action}\bigl(\bm{y},t\bigr)$ is affine in $\bm{y}$ for all $(i,s,\action)\in\calJ$, $t\in[T]_0$, and $h\in\mathbb{N}_+$, where we use ``LL" in the superscript to indicate a locally linear policy (family).
It leads to the existence and Lipschitz continuity of  $\lim_{h\to\infty}\alpha^{\text{LC},h}_{i,s,\action}\bigl(\bm{y},t\bigr)$.
Such local linearity is an important assumption in \cite[Theorem 2]{gast2023linear} for the exponentially diminishing sub-optimality (in expectation) of the proposed policies (policy families).

\subsection{General case}

Now we consider the policies (policy families) for the general WCG.
Here we use superscript ``RFBB" in variables to indicate the relaxed fluid-budget balancing policy~\cite{brown2023fluid}.

For any $(i,s,\action)\in\calJ$ and function $f_{i,s,\action}:\deltay\times \mathbb{N}_+\mapsto \mathbb{R}$, we define the following conditions.
\begin{enumerate}
    \item $\lim_{h\to \infty}f_{i,s,\action}(\bm{y},h)$ exists and is Lipschitz continuous in $ \bigl\{\bm{y}\in\deltay\bigl| y_{i,s}>0\bigr\}$.\label{cond:deltaf:1}
    \item $\lim_{h\to\infty}f_{i,s,\action}(\bm{y},h)$ is continuous in $\bigl\{\bm{y}\in\deltay\bigl| y_{i,s}=0\bigr\}$.  \label{cond:deltaf:2}
    \item $\lim_{y_{i,s}\to 0}\lim_{h\to\infty}f_{i,s,\action}(\bm{y},h)$ either exists as a finite number or tends to infinity. \label{cond:deltaf:3}
    \item If $\lim_{y_{i,s}\to 0}\lim_{h\to\infty}f_{i,s,\action}(\bm{y},h) < \infty$, then there exists $\delta y >0$ such that, for any $y_{i,s}\in[0,\delta y]$, $ \lim_{h\to\infty}f_{i,s,\action}(\bm{y},h)$ is constant.\label{cond:deltaf:4}
\end{enumerate}

\begin{proposition}\label{prop:stable:fluid-balance}
The relaxed fluid-budget balancing policy (family) proposed in \cite{brown2023fluid} satisfy \partialref{cond:weak_stab}{Lipschitz–limit regularity}.
In particular, Conditions~\ref{cond:deltaf:1}-\ref{cond:deltaf:4} hold when plugging in $f_{i,s,\action}(\bm{y},h)=\alpha^{\text{RFBB},h}_{i,s,\action}(\bm{y},t)$ ($(i,s,\action)\in\calJ$).
\end{proposition}
\proof{Proof of Proposition~\ref{prop:stable:fluid-balance}.}
To describe the relaxed fluid-budget balancing policy proposed in  \cite{brown2023fluid}, we need to firstly choose an $\ell^*\in[L]$ and define the following variable,
for any $(i,s,\action)\in\calJ$, $h\in\mathbb{N}_+$, $t\in[T]_0$, and $\bm{y}\in\deltay$,
\begin{equation}\label{eqn:prop:stable:fluid-balance:1}
    \frakR^h_{i,s,\action}\bigl(\bm{y},t\bigr) \coloneqq\begin{cases}
        \Bigl\lfloor \sum_{i'\in[I]}N_{i'} y_{i,s} \eta_{t}(i,s,\action)\Bigr\rfloor, & \text{if }\action \neq \underline{\action}(i,s),\\
        \sum_{i'\in[I]}N_{i'} y_{i,s} - \sum_{\action'\in\bA_i:\action' \neq \underline{\action}(i,s)}\Bigl\lfloor \sum_{i'\in[I]}N_{i'} y_{i,s} \eta_{t}(i,s,\action')\Bigr\rfloor, &\text{otherwise},
    \end{cases}
\end{equation}
where $\eta_{t}(i,s,\action)\in[0,1]$ ($(i,s,\action)\in\calJ$, $t\in[T]_0$) and $\underline{\action}(i,s)\coloneqq \arg\min_{\action\in\bA_i: \eta_{(i,s,a),t}>0}f_{i,\ell^*}(s,\action)$ are some pre-determined constants independent to $\bm{Y}^{\text{RGBB},h}(t)=\bm{y}$ and $h$.
We have 
\begin{equation}\label{eqn:prop:stable:fluid-balance:2}
    \lim_{h\to\infty} \frac{\frakR_{i,s,\action}^h(\bm{y},t)}{\sum_{i'\in[I]}N_{i'} y_{i,s}} = 
\begin{cases}
    \eta_{t}(i,s,\action), &\text{if }\action \neq\underline{\action}(i,s),\\
    1 - \sum_{\action'\in\bA_i:\action' \neq \underline{\action}(i,s)} \eta_{t}(i,s,\action'), &\text{otherwise},
\end{cases}
\end{equation}
which is independent to $\bm{y}$.

For RFBB, we need to further describe some parameters   that are independent to $\bm{Y}^{\text{RGBB},h}(t)$ and $h$:
\begin{itemize}
    \item a set of gang-state pairs $\mathscr{Y}^0_t\subset\{(i,s):i\in[I],s\in\bS_i\}$, and
    \item two sets of GSA triples
    $$
        \mathscr{J}^{0,+}_t\coloneqq \Bigl\{(i,s,\action)\in\calJ~\Bigl|~(i,s)\in\mathscr{Y}^0_t, f_{i,\ell^*}\bigl(s,\action\bigr) < f_{i,\ell^*}\bigl(s,\bar{\action}(i,s)\bigr)\Bigr\},
    $$
    and
    $$
        \mathscr{J}^{0,-}_t\coloneqq \Bigl\{(i,s,\action)\in\calJ~\Bigl|~(i,s)\in\mathscr{Y}^0_t, f_{i,\ell^*}\bigl(s,\action\bigr) > f_{i,\ell^*}\bigl(s,\underline{\action}(i,s)\bigr)\Bigr\},
    $$
    where $\bar{\action}(i,s)\coloneqq \arg\max_{\action\in\bA_i: \eta_{t}(i,s,\action)>0}f_{i,\ell^*}(s,\action)$.
\end{itemize}
Also, as described in \cite{brown2023fluid}, for the RFBB in this proof,  we order the \GSA triples and let $\zeta(i,s,\action) \in [|\calJ|]$ represent the rank of \GSA $(i,s,\action)\in\calJ$.
Let $(i_1,s_1,\action_1)$ represent the first \GSA in such an order.

For $\bm{Y}^{\text{RGBB},h}(t)=\bm{y}$, 
and $(i,s,\action)\in\calJ$,
\begin{multline}\label{eqn:prop:stable:fluid-balance:4}
\alpha^{\text{RFBB},h}_{i,s,\action}(\bm{y},t) =\\ \begin{cases}
    \frac{1}{\sum_{i'\in[I]}N_{i'} y_{i,s}} \max\Bigl\{\frakR^h_{i,s,\action}(\bm{y},t) + \frac{\Delta f^h_{i,s,\action}(\bm{y},t)}{f_{i,\ell^*}(s,\bar{\action}(i,s)) - f_{i,\ell^*}(s,\action)},0\Bigr\}, &\text{if } 
    \Delta f^h_{i_1,s_1,\action_1}(\bm{y},t) <0, \\&~~(i,s,\action)\in \mathscr{J}^{0,+}_t,\\
    \frac{1}{\sum\limits_{i'\in[I]}N_{i'} y_{i,s}} \max\Bigl\{\frakR^h_{i,s,\action}(\bm{y},t) + \frac{\Delta f^h_{i,s,\action}(\bm{y},t)}{f_{i,\ell^*}(s,\underline{\action}(i,s)) - f_{i,\ell^*}(s,\action)},0\Bigr\}, &\text{if } 
    \Delta f^h_{i_1,s_1,\action_1}(\bm{y},t) >0, \\&~~(i,s,\action)\in \mathscr{J}^{0,-}_t,\\
    \frac{1}{\sum_{i'\in[I]}N_{i'} y_{i,s}} \Bigl(\frakR^h_{i,s,\action}(\bm{y},t) 
    &\\~~~~~~+\sum_{\begin{subarray}~(i',s',\action')\in\mathscr{J}^{0,+}_t:\\i'=i,\\s'=s\end{subarray}}\min\Bigl\{\frakR^h_{i,s,\action'}(\bm{y},t),  \frac{-\Delta f^h_{i,s,\action'}(\bm{y},t)}{f_{i,\ell^*}(s,\bar{\action}(i,s)) - f_{i,\ell^*}(s,\action')}\Bigr\} \Bigr), &\text{if } \Delta f^h_{i_1,s_1,\action_1}(\bm{y},t) <0, \\&~~\action = \bar{\action}(i,s),\\&~~(i,s)\in\mathscr{Y}^0_t,\\    
    \frac{1}{\sum_{i'\in[I]}N_{i'} y_{i,s}} \Bigl(\frakR^h_{i,s,\action}(\bm{y},t) 
    &\\~~~~~~+\sum_{\begin{subarray}~(i',s',\action')\in\mathscr{J}^{0,-}_t:\\i'=i,\\s'=s\end{subarray}}\min\Bigl\{\frakR^h_{i,s,\action'}(\bm{y},t),  \frac{-\Delta f^h_{i,s,\action'}(\bm{y},t)}{f_{i,\ell^*}(s,\underline{\action}(i,s)) - f_{i,\ell^*}(s,\action')}\Bigr\} \Bigr), 
    &\text{if } \Delta f^h_{i_1,s_1,\action_1}(\bm{y},t) >0, \\&~~\action = \underline{\action}(i,s),\\&~~(i,s)\in\mathscr{Y}^0_t,\\
    \frac{\frakR^h_{i,s,\action}(\bm{y},t) }{\sum_{i\in[I]}N_i y_{i,s}}, &\text{otherwise}, 
\end{cases}
\end{multline}
where 
\begin{multline}\label{eqn:prop:stable:fluid-balance:5}
\Delta f^h_{i,s,\action}(\bm{y},t)\\\coloneqq 
\begin{cases}
 \sum_{(i',s',\action')\in\calJ}f_{i',\ell^*}(s',a')\frakR^h_{i',s',\action'}(\bm{y},t), &\text{if } \zeta(i,s,\action) =1 \\&~~~~\text{ or } \Delta f^h_{i_1,s_1,\action_1}(\bm{y},t)=0,\\
 \sum_{\begin{subarray}~(i',s',\action')\in \mathscr{J}^{0,+}_t:\\\zeta(i',s',\action')<\zeta(i,s,\action)\end{subarray}}\min\Bigl\{\frakR^h_{i',s',\action'}(\bm{y},t)\bigl(f_{i',\ell^*}(s',\bar{\action}(i',s'))-f_{i',\ell^*}(s',\action')\bigr), 
 &\\~~~~-\Delta f^h_{i',s',\action'}(\bm{y},t)\Bigr\} 
+\sum_{(i',s',\action')\in\calJ}f_{i',\ell^*}(s',\action')\frakR^h_{i',s',\action'}(\bm{y},t), &\text{if }\Delta f^h_{i_1,s_1,\action_1}(\bm{y},t) < 0,\\&~~~~\text{ and } \zeta(i,s,\action)>1,\\
 \sum_{\begin{subarray}~(i',s',\action')\in \mathscr{J}^{0,-}_t:\\\zeta(i',s',\action')<\zeta(i,s,\action)\end{subarray}}\max\Bigl\{\frakR^h_{i',s',\action'}(\bm{y},t)\bigl(f_{i',\ell^*}(s',\underline{\action}(i',s'))-f_{i',\ell^*}(s',\action')\bigr) ,
 &\\~~~~- \Delta f^h_{i',s',\action'}(\bm{y},t)\Bigr\} 
 +\sum_{(i',s',\action')\in\calJ}f_{i',\ell^*}(s',\action')\frakR^h_{(i',s',\action')}(\bm{y},t), &\text{if }\Delta f^h_{i_1,s_1,\action_1}(\bm{y},t) > 0, \\&~~~~\text{ and }\zeta(i,s,\action) >1.
\end{cases}
\end{multline}
Based on the above definition, for any $\zeta(i,s,\action) < \zeta(i',s',\action')$, $\norm{\Delta f^h_{i,s,\action}(\bm{y},t)} \geq \norm{\Delta f^h_{i',s',\action'}(\bm{y},t)}$, and they are either all non-positive or all non-negative.

We define, for $(i,s,\action)\in\calJ$, $\bar{\Delta}f_{i,s,\action}(\bm{y},t)\coloneqq \lim_{h\to\infty}\frac{\Delta f^h_{i,s,\action}(\bm{y},t)}{\sum_{i'\in[I]}N_{i'} y_{i,s}}$.
From \eqref{eqn:prop:stable:fluid-balance:5}, \begin{equation}\label{eqn:prop:stable:fluid-balance:5-1}
    \bar{\Delta}f_{i_1,s_1,\action_1}(\bm{y},t) = \frac{1}{y_{i_1,s_1}}\Bigl(\sum_{(i',s',\action')\in\calJ}y_{i',s'}f_{i',\ell^*}(s',\action') \lim_{h\to\infty}\frac{\frakR^h_{i',s',\action'}(\bm{y},t)}{\sum_{i''\in[I]}N_{i''} y_{i',s'}}\Bigr).
\end{equation}
It exists and is Lipschitz continuous in $\bm{y}\in\deltay$ with $y_{i_1,s_1}> 0$, and is continuous as $y_{i_1,s_1}\to 0$.
When $y_{i_1,s_1}\to 0$, $\bar{\Delta}f_{i_1,s_1,\action_1}(\bm{y},t)$ tends to either infinity or zero, of which the sign is the same as that in the brackets at the right hand side of \eqref{eqn:prop:stable:fluid-balance:5-1}.

Assume that, for all $(i,s,\action)$ with $\zeta(i,s,\action)=1,2,\ldots, \zeta'-1$, 
Conditions~\ref{cond:deltaf:1}-\ref{cond:deltaf:4} are satisfied by plugging in $f_{i,s,\action}(\bm{y},h)=\frac{\Delta f^h_{i,s,\action}(\bm{y},t)}{\sum_{i'\in[I]}N_{i'} y_{i,s}}$.
In this case, based on the definition in \eqref{eqn:prop:stable:fluid-balance:5} and the identical signs of all $\Delta f^h_{i',s',\action'}(\bm{y},t)$ ($(i',s',\action')\in\calJ$), by replacing $(i,s,\action)$ with the $\zeta'$th \GSA $(i',s',\action')$ (that is, $\zeta(i',s',\action')=\zeta'$), $f_{i',s',\action'}(\bm{y},h)=\frac{\Delta f^h_{i',s',\action'}(\bm{y},t)}{\sum_{i''\in[I]}N_{i''} y_{i',s'}}$ also satisfies the aforementioned Conditions~\ref{cond:deltaf:1}-\ref{cond:deltaf:4}.
That is, for all $(i,s,\action)\in\calJ$, Conditions~\ref{cond:deltaf:1}-\ref{cond:deltaf:4} hold with $f_{i,s,\action}(\bm{y},h)=\frac{\Delta f^h_{i,s,\action}(\bm{y},t)}{\sum_{i'\in[I]}N_{i'} y_{i,s}}$.

Together with the definition in \eqref{eqn:prop:stable:fluid-balance:4}, 
if $\bar{\Delta} f_{i_1,s_1,\action_1}(\bm{y},t) \neq 0$, then, for all $(i,s,\action)\in\calJ$, $\alpha^{\text{RFBB}}_{i,s,\action}(\bm{y},t)=\lim_{h\to\infty}\alpha^{\text{RFBB},h}_{i,s,\action}(\bm{y},t)$ is Lipschitz continuous in $\bigl\{\bm{y}\in\deltay~\bigl|~\bar{\Delta} f_{i_1,s_1,\action_1}(\bm{y},t) \neq 0\bigr\}$.

It remains to discuss the case within $\bigl\{\bm{y}\in\deltay~\bigl|~\bar{\Delta} f_{i_1,s_1,\action_1}(\bm{y},t) =0\bigr\}$. 
In this case, 
\begin{multline}\label{eqn:prop:stable:fluid-balance:6}
 \lim_{\begin{subarray}~h\to\infty:\\\bar{\Delta}f_{i_1,s_1,\action_1}(\bm{y},t) <0\end{subarray}}\alpha^{\text{RFBB},h}_{i,s,\action}(\bm{y},t) \\=
 \begin{cases}
     \max\Bigl\{\eta_{t}(i,s,\action)+\frac{\bar{\Delta}f_{i,s,\action}(\bm{y},t)}{f_{i,\ell^*}(s,\bar{\action}(i,s)) - f_{i,\ell^*}(s,\action)},0\Bigr\},& \text{if }(i,s,\action) \in\mathscr{J}^{0,+}_t,\\
     \eta_{t}(i,s,\action)  +\sum_{\begin{subarray}~(i',s',\action')\in\mathscr{J}^{0,+}_t:\\i'=i,\\s'=s\end{subarray}}\min\Bigl\{\eta_{t}(i,s,\action),  \frac{-\bar{\Delta} f_{i,s,\action'}(\bm{y},t)}{f_{i,\ell^*}(s,\bar{\action}(i,s)) - f_{i,\ell^*}(s,\action')}\Bigr\}, & \text{if }\action=\bar{\action}(i,s),\\&~~~~ (i,s)\in\mathscr{Y}^0_t,\\
     \eta_{t}(i,s,\action),&\text{otherwise}.
 \end{cases}
\end{multline}
Together with $\norm{\bar{\Delta}f_{i,s,\action}(\bm{y},t)} \leq \lim_{h\to\infty}\norm{\Delta f^h_{i_1,s_1,\action_1}(\bm{y},t)}  = 0$, $\lim_{\begin{subarray}~h\to\infty:\\\bar{\Delta} f_{i_1,s_1,\action_1}(\bm{y},t) \downarrow 0\end{subarray}}\alpha^{\text{RFBB},h}_{i,s,\action}(\bm{y},t)=\eta_{t}(i,s,\action)$.
Similarly, 
\begin{multline}\label{eqn:prop:stable:fluid-balance:7}
 \lim_{\begin{subarray}~h\to\infty:\\\bar{\Delta} f_{i_1,s_1,\action_1}(\bm{y},t) >0\end{subarray}}\alpha^{\text{RFBB},h}_{i,s,\action}(\bm{y},t)\\ =
 \begin{cases}
     \max\Bigl\{\eta_{t}(i,s,\action)+\frac{\bar{\Delta}f_{i,s,\action}(\bm{y},t)}{f_{i,\ell^*}(s,\underline{\action}(i,s)) - f_{i,\ell^*}(s,\action)},0\Bigr\},& \text{if }(i,s,\action) \in\mathscr{J}^{0,-}_t,\\
     \eta_{t}  (i,s,\action)+\sum_{\begin{subarray}~(i',s',\action')\in\mathscr{J}^{0,-}_t:\\i'=i,\\s'=s\end{subarray}}\min\Bigl\{\eta_{t}(i,s,\action),  \frac{-\bar{\Delta} f_{i,s,\action'}(\bm{y},t)}{f_{i',\ell^*}(s',\underline{\action}(i',s')) - f_{i',\ell^*}(s',\action')}\Bigr\}, & \text{if }\action=\underline{\action}(i,s),\\&~~~~ (i,s)\in\mathscr{Y}^0_t,\\
     \eta_{t}(i,s,\action),&\text{otherwise},
 \end{cases}
\end{multline}
leading to 
$$
\lim_{\begin{subarray}~h\to\infty:\\\bar{\Delta}f_{i_1,s_1,\action_1}(\bm{y},t) \downarrow 0\end{subarray}}\alpha^{\text{RFBB},h}_{i,s,\action}(\bm{y},t)=\eta_{t}(i,s,\action)=\lim_{\begin{subarray}~h\to\infty:\\\bar{\Delta}f_{i_1,s_1,\action_1}(\bm{y},t) \uparrow 0\end{subarray}}\alpha^{\text{RFBB},h}_{i,s,\action}(\bm{y},t).$$
That is,  $\alpha^{\text{RFBB}}_{i,s,\action}(\bm{y},t)=\lim_{h\to\infty}\alpha^{\text{RFBB},h}_{i,s,\action}(\bm{y},t) = \eta_{t}(i,s,\action)$ exists and is continuous for the points in $\bigl\{\bm{y}\in\deltay~\bigl|~\bar{\Delta} f_{i_1,s_1,\action_1}(\bm{y},t) =0\bigr\}$.

We recall Conditions~\ref{cond:deltaf:1}-\ref{cond:deltaf:4} for $f_{i,s,\action}(\bm{y},h)=\frac{\Delta f^h_{i,s,\action}(\bm{y},t)}{\sum_{i'\in[I]}N_{i'} y_{i,s}}$, which guarantee the Lipschitz continuity of $\bar{\Delta}f_{i,s,\action}(\bm{y},t)$ in all $\bm{y}\in\bigl\{\bm{y}'\in\deltay~\bigl|~\bar{\Delta} f_{i_1,s_1,\action_1}(\bm{y}',t) =0\bigr\}$.
It further leads to the Lipschitz continuity of $\alpha^{\text{RFBB}}_{i,s,\action}(\bm{y},t)$ in $\bm{y}\in\bigl\{\bm{y}'\in\deltay~\bigl|~\bar{\Delta} f_{i_1,s_1,\action_1}(\bm{y}',t) =0\bigr\}$.

It remains to show that Conditions~\ref{cond:deltaf:1}-\ref{cond:deltaf:4} also hold for plugging in $f_{i,s,\action}(\bm{y},h)=\alpha^{\text{RFBB},h}_{i,s,\action}(\bm{y},t)$.
The Lipschitz continuity and the boundness of $\lim_{h\to\infty}\alpha^{\text{RFBB},h}_{i,s,\action}(\bm{y},t)$ ensure Conditions~\ref{cond:deltaf:1}-\ref{cond:deltaf:3}.
Since $\lim_{h\to\infty}\frac{\frakR^h_{i,s,\action}(\bm{y},t)}{\sum_{i'\in[I]}N_{i'} y_{i,s}}$ is constant independent to $\bm{y}$, by \eqref{eqn:prop:stable:fluid-balance:4}, Condition~\ref{cond:deltaf:4} is satisfied with $f_{i,s,\action}(\bm{y},h)=\alpha^{\text{RFBB},h}_{i,s,\action}(\bm{y},t)$ if it is satisfied with $f_{i,s,\action}(\bm{y},h)=\frac{\Delta f^h_{i,s,\action}(\bm{y},t)}{\sum_{i'\in[I]}N_{i'} y_{i,s}}$.
It proves the proposition.

\endproof

We use ``FBB" to represent the fluid-budget balancing policy family.
Note that, in \cite{brown2023fluid}, FBB was proposed under an assumption: for any $i\in[I]$ and $s\in\bS_i$, there exists $\underline{\action}^*(i,s)\in\bA_i$ such that, for any $\ell\in[L]$, $ f_{i,\ell}(s,\underline{\action}^*(i,s)) =\min_{\action\in \bA_i} f_{i,\ell}(s,\action)$. 
\begin{proposition}\label{prop:stable:fluid-balance:plus}
The fluid-budget balancing policy (
family) proposed in \cite{brown2023fluid} satisfy \partialref{cond:weak_stab}{Lipschitz–limit regularity}.
\end{proposition}
\proof{Proof of Proposition~\ref{prop:stable:fluid-balance:plus}.}
FBB is a policy family adapted from RFBB.
Similar to the $\Delta f^h_{i,s,\action}$ defined in \eqref{eqn:prop:stable:fluid-balance:5}, for FBB, we define
\begin{multline}\label{eqn:prop:stable:fluid-balance:plus:2}
    \Delta g^h_{i,s,\action,\ell}(\bm{y},t)\\ 
    \coloneqq \begin{cases}   
    \max\Bigl\{0,\sum_{(i',s',\action')\in\calJ}f_{i',\ell}(s',\action')\sum_{i''\in[I]}N_{i''} y_{i',s'} \alpha^{\text{RFBB},h}_{i',s',\action'}(\bm{y},t)\Bigr\}, 
    &\!\!\!\!\!\!\!\!\!\!\!\!\!\!\!\!\!\!\!\!\!\!\!\!\!\!\!\!\text{if }\zeta(i,s,\action) =1,\\
    \sum_{\begin{subarray}~(i',s',\action')\in\calJ:\\\zeta(i',s',\action')<\zeta(i,s,\action)\end{subarray}}\min\Bigl\{\sum_{i\in[I]}N_i y_{i',s'} \alpha^{\text{RFBB},h}_{i',s',\action'}(\bm{y},t),  
    \max_{
    \begin{subarray}~\ell'\in[L]:\\(i',s',\action')\in\mathscr{J}^-_{\ell'}\end{subarray}}\frac{\Delta g^h_{i',s',\action',\ell'}(\bm{y},t)}{f_{i',\ell'}(s',\action')-f_{i',\ell'}(s',\underline{\action}^*(i',s'))}\Bigr\}&\\~~~~
    \times\bigl(f_{i',\ell}(s',\underline{\action}^*(i',s'))-f_{i',\ell}(s',\action')\bigr)     
&\\~~~~ +\sum_{(i',s',\action')\in\calJ}f_{i',\ell}(s',\action')\sum_{i''\in[I]}N_{i''} y_{i',s'} \alpha^{\text{RFBB},h}_{i',s',\action'}(\bm{y},t), 
&\!\!\!\!\!\!\!\!\!\!\!\!\!\!\!\!\!\!\!\!\!\!\!\!\!\!\!\!\text{otherwise},\\
    \end{cases}
\end{multline}
where $\mathscr{J}^-_{\ell}\coloneqq \Bigl\{(i,s,\action)\in\calJ:  f_{i,\ell}(s,\action) > f_{i,\ell}(s,\underline{\action}^*(i,s))\Bigr\}$, and $\max_{\ell'\in[L]:(i',s',\action')\in\mathscr{J}^-_{\ell}}$ returns zero if $\{\ell'\in[L]:(i',s',\action')\in\mathscr{J}^-_{\ell}\} = \emptyset$.
Based on Proposition~\ref{prop:stable:fluid-balance}, it can be obtained iteratively from $(i_1,s_1,\action_1)$ to the last \GSA that Conditions~\ref{cond:deltaf:1}-\ref{cond:deltaf:4} hold with substituted $f_{i,s,\action}(\bm{y},h)=\frac{\Delta g^h_{i,s,\action,\ell}(\bm{y},t)}{\sum_{i'\in[I]}N_{i'} y_{i,s}}$ ($(i,s,\action)\in\calJ$).

Let $\bar{\Delta} g_{i,s,\action,\ell}(t)\coloneqq \lim_{h\to\infty} \frac{\Delta g^h_{i,s,\action,\ell}(\bm{y},t)}{\sum_{i'\in[I]}N_{i'} y_{i,s}}$.
Also, based on the definition in \eqref{eqn:prop:stable:fluid-balance:plus:2}, $\max_{\ell\in[L]}\Delta g^h_{i,s,\action,\ell}(\bm{y},t)$ is always non-negative.

From the definition of FBB in \cite{brown2023fluid}, we have
\begin{multline}
    \alpha^{\text{FBB},h}_{i,s,\action}(\bm{y},t)\\ = \begin{cases}
        \max\Bigl\{\alpha^{\text{RFBB},h}_{i,s,\action}(\bm{y},t) - \frac{1}{\sum_{i'\in[I]}N_{i'} y_{i,s}}\max_{\ell\in[L]:(i,s,\action)\in\mathscr{J}^-_{\ell}}\frac{\Delta g^h_{i,s,\action,\ell}(\bm{y},t)}{f_{i,\ell}(s,\action)-f_{i,\ell}(s,\underline{\action}^*(i,s))},0\Bigr\}, &\text{if }\action\neq \underline{\action}^*(i,s),\\
        \alpha^{\text{RFBB},h}_{i,s,\action}(\bm{y},t) &\\~~+ \frac{1}{\sum_{i'\in[I]}N_{i'} y_{i,s}}\sum_{\action'\in\bA_i}\min\Bigl\{\max_{\ell\in[L]:(i,s,\action')\in\mathscr{J}^-_{\ell}}\frac{\Delta g^h_{i,s,\action',\ell}(\bm{y},t)}{f_{i,\ell}(s,\action')-f_{i,\ell}(s,\underline{\action}^*(i,s))},
        &\\~~~~~~~~~~~~~~~~~~~~~~~~~~~~~~~~~~~~~~~~~\sum_{i'\in[I]}N_{i'} y_{i,s}\alpha^{\text{RFBB},h}_{i,s,\action'}(\bm{y},t)\Bigr\},&\text{otherwise}.
    \end{cases}
\end{multline}
That is, 
\begin{multline*}
\lim_{h\to\infty}\alpha^{\text{FBB},h}_{i,s,\action}(\bm{y},t) \\= \begin{cases}
  \max\Bigl\{\alpha^{\text{RFBB}}_{i,s,\action}(\bm{y},t)-  \max_{\ell\in[L]:(i,s,\action)\in\mathscr{J}^-_{\ell}}\frac{\bar{\Delta} g_{i,s,\action,\ell}(t)}{f_{i,\ell}(s,\action)-f_{i,\ell}(s,\underline{\action}^*(i,s))},0\Bigr\}, &\text{if }\action\neq \underline{\action}^*(i,s),\\
  \alpha^{\text{RFBB}}_{i,s,\action}(\bm{y},t) 
  &\\~~+ \sum_{\action'\in\bA_i}\min\Bigl\{\max_{\ell\in[L]:(i,s,\action')\in\mathscr{J}^-_{\ell}}\frac{\bar{\Delta} g_{i,s,\action',\ell}(t)}{f_{i,\ell}(s,\action')-f_{i,\ell}(s,\underline{\action}^*(i,s))},\alpha^{\text{RFBB},h}_{i,s,\action'}(\bm{y},t)\Bigr\}&\text{otherwise},
\end{cases}
\end{multline*}
which, based on Proposition~\ref{prop:stable:fluid-balance} and Conditions~\ref{cond:deltaf:1}-\ref{cond:deltaf:4} for plugging in $f_{i,s,\action}(\bm{y},h)=\Delta g^h_{i,s,\action}(\bm{y},t)$, is Lipschitz continuous in $\bm{y}\in\deltay$. It proves the proposition.

\endproof

\section{Proof of Theorem~\ref{theorem:convergence-Z}}\label{app:theorem:convergence-Z}

We extend $\calJ$, the set of all GSA triples, to $\calJext \coloneqq \bigl\{(i,s,\action): i\in[I], s\in \bS_i, \action\in \bA_i\cup\{\dummyAction\}\bigr\}$, where $\dummyAction$ is a dummy action only used for completing the proof.
Define a gang-state-action-time (GSAT) $(i,s,\action,t)\in \calJext\times [T]_0$.
We  denote the set $\mathscr{K}\coloneqq \calJext\times [T]_0$ of all such \GSATs of the form $\kappa = (i,s,\action,t)$. 
For a \GSA $(i,s,\action)\in\calJext$, we rewrite it in the form as $\zeta=(i,s,\action)$, and let $i=i^\zeta$,  $s=s^\zeta$, and  $a=\action^{\zeta}$ to indicate the gang,  state, and action, respectively,  from a given $\zeta\in \calJext$.
Similarly, let $i^{\kappa}$, $\action^{\kappa}$, $s^{\kappa}$, and $t^{\kappa}$ represent the gang, action, state, and time, respectively, for $\kappa\in\mathscr{K}$.
We will also, when needed, use $\zeta^{\kappa}$ to represent the \GSA $(i^{\kappa},s^{\kappa},\action^{\kappa})$.

We define independent (and independent of all other random variables previously defined) Poisson random processes
$\tau \mapsto M_t(\zeta,\zeta',n,\tau)$ ($ \zeta,\zeta'\in \calJext,\, n\in [N_{i^{\zeta}}]$) with (exponentially distributed) inter-event times with rate
\begin{equation}\label{eqn:app:convergence_Z:2}
 \scrp_t(\zeta,\zeta')\coloneqq 
 \begin{cases}
    \mathbb{P}\Bigl\{s^{\orule,h}_{i^{\zeta},n}(t+1) = s^{\zeta'}~\Bigl|~s^{\orule,h}_{i^{\zeta},n}(t) = s^{\zeta},\action^{\orule,h}_{i^{\zeta},n}(t) = \action^{\zeta}\Bigr\}=p_{i^{\zeta}}(s^{\zeta},\action^{\zeta},s^{\zeta'}), &\text{if } \action^{\zeta'} = \dummyAction,\action^{\zeta}\neq \dummyAction,
    \\&~~\text{ and } i^{\zeta}=i^{\zeta'},\\
     0,&\text{otherwise}.
\end{cases}
\end{equation} 
For $t\in[T]_0$, define 
\begin{equation}
  \label{eq:summed_ms}
    M_t(\tau)\coloneqq \sum_{\zeta,\zeta'\in\mathcal{J}}\sum_{n\in[N_{i^\kappa}]}M_t(\zeta,\zeta',n,\tau),
\end{equation}
where $\tau_t(m)\coloneqq \inf \{\tau\geq 0~|~M_t(\tau)\geq m\}$
is the  occurrence time of the $m$th event in this  process, and $\tau_t(0) \equiv 0$.
The $m$th event in this process occurs in one of the summands labeled by $(\zeta, \zeta',n)$ in~\eqref{eq:summed_ms}, and to indicate this we write  $\bigl(\kappa_t(m),\zeta_t(m),n_t(m)\bigr) = \bigl((i^{\zeta},s^{\zeta},\action^{\zeta},t),\zeta',n\bigr)$, where $\kappa_t(m) = (i^{\zeta},s^{\zeta},\action^{\zeta},t)$.

Let $\scrR\coloneqq \bigcup_{n\in\mathbb{N}_0} [2n,2n+1)$, which is the set of all intervals from an even integer to its next odd integer.
For $\kappa, \kappa'\in\mathscr{K}$, $n\in[N_{i^{\kappa}}]$, we define
  \begin{equation}
    \label{eq:xi_definition_new}
\xi_\tau(\kappa,\kappa',n)\coloneqq
\sum_{\substack{m: (\kappa_t(m),\zeta_t(m),n_t(m)) = (\kappa,\zeta^{\kappa'},n)\\ t^{\kappa'}=t^{\kappa}+1}}\frac{\indicator\Biggl\{\tau\notin\scrR, \tau \in \Bigl[\max\bigl\{\tau_{t^{\kappa}}(m-1), \lfloor\tau\rfloor\big\},\tau_{t^{\kappa}}(m)\Bigr)\Biggr\}}{\tau_{t^{\kappa}}(m)-\tau_{t^{\kappa}}(m-1)},
  \end{equation}
where $\indicator\{\cdot\}$ is the indicator function.
We recall that an empty sum evaluates to $0$. This is a cadlag process for each $\kappa,\kappa',n$ and is zero unless  $t^{\kappa'}=t^{\kappa}+1$.  
We consider the vector random process $\tau\mapsto \bm{\xi}_{\tau}(\kappa,n) = (\xi_\tau(\kappa,\kappa',n):\kappa'\in \mathscr{K})$, for $\kappa\in \mathscr{K}$. 
We define another deterministic cadlag process 
\begin{equation} \label{eq:xi_definition_new:deter}
    \xi_{\tau}(0) \coloneqq \indicator\Bigl\{\tau\in \scrR\Bigr\}.
\end{equation}

For all fixed  $\kappa\in \mathscr{K}$, the random vectors $\bm{\xi}_{\tau}(\kappa,n)$ are identically distributed as   $\tau\geq 0$ varies,   and the trajectory $\xi_{\tau}(\kappa,\kappa',n)$ is continuous in $\tau\geq 0$ except, possibly,  at  finitely many points of discontinuity of  the first kind in  every finite time interval.

For any two  ``jump'' points $\tau_{t}(m_{1})$ and $\tau_{t}(m_{2})$ ($m_{2}>m_{1}$)
  \begin{multline}
    \label{eq:xi_integral}
   \int_{\tau_t(m_{1})}^{\tau_t(m_{2})}\xi_\tau(\kappa,\kappa',n)\, d\tau \\= \Bigl|\bigl\{m: (\kappa_{t^{\kappa}}(m),\zeta_{t^{\kappa}}(m),n_{t^{\kappa}}(m) = (\kappa,{\zeta^{\kappa'}},n), m_{1}<m\le m_{2},t^{\kappa'}={t^{\kappa}}+1,\tau_t(m)\notin \scrR\bigr\}\Bigr|
  \end{multline}
Note that the endpoints of this integral are random. 

Such a $\xi_{\tau}(\kappa,\kappa',n)$ is considered as an event that may trigger a state transition of the corresponding process, when $\tau\notin\scrR$. 
The periodic $\tau\in\scrR$ is set as silent periods without events.

We use the word ``potential" because a bandit process may not be in the
\GSA triple $\zeta^{\kappa}$ at time $t^{\kappa}$, causing fewer real transitions. 
Let $\bm{\xi}^h_\tau\coloneqq (\xi_{\tau}(0);\xi_{\tau}(\kappa,\kappa',n): \kappa,\kappa'\in\mathscr{K}, n\in[N_{i^{\kappa}}])$ which takes values in $\mathbb{R}_0^{|\mathscr{K}|\sum_{\kappa\in\mathscr{K}}N_{i^{\kappa}}+1}$, where recall $N_i = hN_i^0$.

For $\orule\in\PsiZ$, $\zeta\in\calJext$, $t\in[T]_0$, and $\bm{\scrz}\in \mathbb{R}^{|\mathscr{K}|}$, we define
\begin{equation}\label{eqn:define:alpha:x}
\scra^{\orule}_{\zeta}(\bm{\scrz},t) \coloneqq
\begin{cases}
    \alpha^{\orule}_{\zeta}(\scru(\bm{\scrz},t),t), & \text{if } \scru(\bm{\scrz},t) \in \deltay, \action^{\zeta}\neq \dummyAction,\\
    0, &\text{if } \action^{\zeta}=\dummyAction,\\
    \text{anything within $[0,1]$ and making $\scra$ Lipschitz continuous in $\bm{\scrz}$}, &\text{otherwise},
\end{cases}
\end{equation}
where 
\begin{equation}\label{eqn:define:scru}
    \scru(\bm{\scrz},t)\coloneqq \Bigl(\sum_{\begin{subarray}~\kappa\in\mathscr{K}:\\t^{\kappa}=t,\\i^{\kappa}=i,\\s^{\kappa}=s\end{subarray}}\scrz_{\kappa} ~:~i\in[I], s\in\bS_i\Bigr).
\end{equation}
For $\orule\in\PsiZ$ and $\zeta\in\calJ$, such $\scra^{\orule}_{\zeta}$ is Lipschitz continuous in its first argument due to the Lipschitz continuity of $\alpha^{\orule}_{\zeta}$ in the first argument (guaranteed by \partialref{cond:weak_stab}{Lipschitz–limit regularity}).

For given $h\in\mathbb{N}_+$, $\bm{x}\in\mathbb{R}^{|\mathscr{K}|}$, $\bm{\xi}=(\xi_0;\xi(\kappa,\kappa',n):\kappa,\kappa'\in\mathscr{K},n\in[N_{i^{\kappa}}])\in\mathbb{R}^{|\mathscr{K}|\sum_{\kappa'\in\mathscr{K}}N_{i^{\kappa'}}+1}$, and $\kappa,\kappa'\in\mathscr{K}$, define
\begin{multline}\label{eqn:app:convergence_Z:3}
Q^{h,\Lipschitza}(\kappa,\kappa',\bm{x},\bm{\xi}) \coloneqq 
\indicator\biggl\{\sum_{\begin{subarray}~\kappa''\in \mathscr{K}:\\t^{\kappa''} <t^{\kappa}\end{subarray}}\lvert x_{\kappa''}\rvert =0,\sum_{\begin{subarray}~\kappa'' \in \mathscr{K}:\\t^{\kappa''} = t^{\kappa},\\\action^{\kappa''}=\dummyAction\end{subarray}}\lvert x_{\kappa''}\rvert=0\biggr\}\sum_{n\in\bigl[ h\bigl\lceil  \lvert x_{\kappa}\rvert/h\bigr\rceil\bigr]}\xi(\kappa,\kappa',n) \\
+ \indicator\biggl\{\sum_{\begin{subarray}~\kappa''\in \mathscr{K}:\\t^{\kappa''} <t^{\kappa}\end{subarray}}\lvert x_{\kappa''}\rvert =0,\sum_{\begin{subarray}~\kappa'' \in \mathscr{K}:\\t^{\kappa''} = t^{\kappa},\\\action^{\kappa''}=\dummyAction\end{subarray}}\lvert x_{\kappa''}\rvert>0, i^{\kappa} = i^{\kappa'},s^{\kappa}=s^{\kappa'}, \action^{\kappa} = \dummyAction, t^{\kappa}=t^{\kappa'}\biggr\}\\
\times\sum_{\begin{subarray}~\kappa''\in\mathscr{K}:\\i^{\kappa''}=i^{\kappa},\\s^{\kappa''}=s^{\kappa},\\t^{\kappa}=t^{\kappa''}\end{subarray}}\lvert x_{\kappa''}\rvert\scra^{\orule}_{\zeta^{\kappa'}}(\frac{\bm{x}}{h\sum_{i\in[I]}N_i^0},t^{\kappa'})\xi(0)
+ f^{h,\Lipschitza}_{\kappa,\kappa'}(\lvert \bm{x}\rvert,\bm{\xi}),
\end{multline}
where $\lvert \bm{x} \rvert \coloneqq (\lvert x_{\kappa}\rvert: \kappa\in\mathscr{K})$, recall $\scra^{\orule}_{\zeta}$  is the fraction of bandit processes in state $s^{\zeta}$ to take action $\action^{\zeta}$ defined in \eqref{eqn:define:alpha:x}, and if $x_{\kappa} = 0$, $\bigl[ h\lceil  x_{\kappa}/h\rceil\bigr] = \emptyset$.

We consider a class of functions $f\in\mathscr{F}$ of variables $h, \Lipschitza, \kappa, \kappa', \bm{\xi}, \lvert x\rvert$, and written as 
$f^{h,\Lipschitza}_{\kappa,\kappa'}(\lvert \bm{x}\rvert,\bm{\xi})$, which satisfy: 
\begin{enumerate}
\item $\bm{\xi}\mapsto f^{h,\Lipschitza}_{\kappa,\kappa'}(\lvert \bm{x}\rvert,\bm{\xi})$ is linear for all $h, \Lipschitza, \kappa, \kappa',\lvert x\rvert$;
\item $\Lipschitza\mapsto f^{h,\Lipschitza}_{\kappa,\kappa'}(\lvert \bm{x}\rvert,\bm{\xi})$ is continuous in $a\in (0,1)$ for all $h, \kappa, \kappa',\lvert x\rvert$;
\item For such $f\in \mathscr{F}$, we must have
$\lim_{\Lipschitza \downarrow 0} f^{h,\Lipschitza}_{\kappa,\kappa'}(\lvert \bm{x}\rvert,\bm{\xi})=0$;
\item $\lim_{\Lipschitza\downarrow0}d~ f^{h,\Lipschitza}_{\kappa,\kappa'}(\lvert\bm{x}\rvert,\bm{\xi})/d\Lipschitza = 0$;
  \item for given $\Lipschitza\in(0,1)$, $Q^{1,\Lipschitza}
  (\kappa,\kappa',\bm{x},\bm{\xi})$ is Lipschitz in $\bm{x}$ and $\bm{\xi}$ separately; that is, there exist $A_{\bm{x},\kappa,\kappa'}>0$ and $B_{\bm{\xi},\kappa, \kappa'}>0$ such that
        \begin{equation}
            \label{eq:lipsq}
\begin{aligned}
  \lvert Q^{1,\Lipschitza}(\kappa,\kappa',\bm{x},\bm{\xi})-Q^{1,\Lipschitza}(\kappa,\kappa',\bm{x},\bm{\xi'})\rvert&\le A_{\bm{x},\kappa,\kappa'}\norm{\bm{\xi}-\bm{\xi'}}\\
  \lvert Q^{1,\Lipschitza}(\kappa,\kappa',\bm{x},\bm{\xi})-Q^{1,\Lipschitza}(\kappa,\kappa',\bm{x'},\bm{\xi})\rvert &\le B_{\bm{\xi},\kappa,\kappa'}\norm{\bm{x}-\bm{x'}}.
 \end{aligned} 
\end{equation}     
\end{enumerate}

As in \cite{fu2018restless,fu2020energy,fu2024patrolling}, such functions $f^{h,\Lipschitza}_{\kappa,\kappa'}(\bm{x},\bm{\xi})$ can be constructed by incorporating the Dirac delta function.

We provide an example of such $f^{h,\Lipschitza}_{\kappa,\kappa'}$ in Appendix~\ref{app:continuity_Q}, together with proofs for Conditions 1-5.

For $h\in\mathbb{N}_+$, $\kappa\in\mathscr{K}$, $\bm{x}\in\mathbb{R}^{|\mathscr{K}|}$, $\bm{\xi}\in\mathbb{R}^{|\mathscr{K}|\sum_{\kappa'\in\mathscr{K}}N_{i^{\kappa'}}+1}$, define a function
\begin{equation}\label{eqn:app:convergence_Z:4}
b^{h,\Lipschitza}_{\kappa}(\bm{x},\bm{\xi})\coloneqq \sum_{\kappa'\in\mathscr{K}}\Bigl(Q^{h,\Lipschitza}(\kappa',\kappa,\bm{x},\bm{\xi})-Q^{h,\Lipschitza}(\kappa,\kappa',\bm{x},\bm{\xi})\Bigr),
\end{equation}
which, for $h=1$, is separate Lipschitz continuous in $\bm{x}$ and $\bm{\xi}$.
Let $b^{h,\Lipschitza}(\bm{x},\bm{\xi})\coloneqq \bigl(b^{h,\Lipschitza}_{\kappa}(\bm{x},\bm{\xi}):\kappa\in\mathscr{K}\bigr)$, 
and, based on \eqref{eqn:app:convergence_Z:3} and \eqref{eqn:app:convergence_Z:4}, there exists matrix 
$\tilde{\mathcal{Q}}^{h,\Lipschitza}(\bm{x})\in\mathbb{R}^{|\mathscr{K}|\times\bigl(|\mathscr{K}|\sum_{\kappa'\in\mathscr{K}}N_{i^{\kappa'}}\bigr)}$ such that $b^{h,\Lipschitza}(\bm{x},\bm{\xi})=\tilde{\mathcal{Q}}^{h,\Lipschitza}(\bm{x})\bm{\xi}$.

For $U\in\mathbb{R}_0\cup\{\infty\}$, we define an adapted version of $b^h(\bm{x},\bm{\xi})$ as
\begin{equation}\label{eqn:app:convergence_Z:5}
b^{h,\Lipschitza,U}(\bm{x},\bm{\xi})\coloneqq b^{h,\Lipschitza,U}\bigl(\bm{x},\min\bigl\{\bm{\xi},U\bigr\}\bigr)=\tilde{\mathcal{Q}}^{h,\Lipschitza}(\bm{x})\min\{\bm{\xi},U\},
\end{equation}
where $\min\bigl\{\bm{\xi},U\bigr\} \coloneqq \bigl(\min\{\xi(\kappa,\kappa',n),U\}: \kappa,\kappa'\in\mathscr{K},n\in[N_{i^{\kappa}}]\bigr)$.
Let $b^{\Lipschitza,U}(\bm{x},\bm{\xi})$ represent the special case $b^{h,\Lipschitza,U}(\bm{x},\bm{\xi})$ with $h=1$.
It is obvious that $b^{h,\Lipschitza,\infty}(\bm{x},\bm{\xi})\coloneqq\lim_{U\rightarrow \infty}b^{h,\Lipschitza,U}(\bm{x},\bm{\xi}) = b^{h,\Lipschitza}(\bm{x},\bm{\xi})$.

\begin{lemma}\label{lemma:continuity_bU}
For $\Lipschitza\in(0,1)$, and $U\in\mathbb{R}_0$,     $b^{\Lipschitza,U}(\bm{x},\bm{\xi})$ is jointly Lipschitz continuous in $\bm{x}$ and $\bm{\xi}$.
\end{lemma}
The proof of Lemma~\ref{lemma:continuity_bU} is provided in Appendix~\ref{app:continuity_bU}.

For the special case with $h=1$, $\Lipschitza \in(0,1)$, and $U\in\mathbb{R}_0$, we construct a trajectory $\bm{X}^{\sigma,\Lipschitza,U}_{\tau}$ on $\tau\geq 0$ that satisfies
\begin{equation}\label{eqn:app:convergence_Z:6}
\dot{\bm{X}}^{\sigma,\Lipschitza,U}_\tau = b^{\Lipschitza,U}(\bm{X}^{\sigma,\Lipschitza,U}_{\tau},\bm{\xi}^{1}_{\tau/\sigma}),
\end{equation}
with given $\bm{X}^{\sigma,\Lipschitza,U}_0=\bm{x}_0\coloneqq (x_{\kappa,0}:\kappa\in\mathscr{K})$, where  $\bm{\xi}^1_{\tau}$ is the special case of $\bm{\xi}^{h}_{\tau}$ with $h=1$.
\begin{lemma}\label{lemma:existence_ODE}
For any given $\bm{x}_0\in\mathbb{R}^{|\mathscr{K}|}_0$, $\bar{T}>0$, $\Lipschitza \in(0,1)$, and $U\in\mathbb{R}_0$, the solution to \eqref{eqn:app:convergence_Z:6} uniquely exists in $\tau\in[0,\bar{T}]$.
\end{lemma}
The proof for Lemma~\ref{lemma:existence_ODE} is provided in Appendix~\ref{app:existence_ODE}.

Recall that $b^{h,\Lipschitza,U}(\bm{x},\bm{\xi}) = \tilde{\mathcal{Q}}^{h,\Lipschitza}(\bm{x})\min\{U,\bm{\xi})$, where $\tilde{\mathcal{Q}}^{h,\Lipschitza}(\bm{x})$ is independent from $\bm{\xi}$.
It is relatively straightforward to see 
\begin{equation}\label{eqn:app:existence_ODE_limit:2}
\lim_{a\downarrow 0}\lim_{U\to \infty} \tilde{\mathcal{Q}}^{h,\Lipschitza}(\bm{x}) \min\{U,\bm{\xi}\}
= \lim_{a\downarrow 0}\tilde{\mathcal{Q}}^{h,\Lipschitza}(\bm{x}) \bm{\xi} = \lim_{U\to \infty} \lim_{a\downarrow 0}\tilde{\mathcal{Q}}^{h,\Lipschitza}(\bm{x}) \min\{U,\bm{\xi}\}.
\end{equation}
We define 
\begin{equation}\label{eqn:app:existence_ODE_limit:3}
\beta^h(\bm{x},\bm{\xi}) \coloneqq \lim_{\begin{subarray}~a\downarrow 0\\U\to \infty\end{subarray}}b^{h,\Lipschitza,U}(\bm{x},\bm{\xi}) = \lim_{a\downarrow 0}\tilde{\mathcal{Q}}^{h,\Lipschitza}(\bm{x}) \bm{\xi}\eqqcolon \tilde{\Theta}^h(\bm{x})\bm{\xi},
\end{equation}
and
\begin{equation}\label{eqn:app:existence_ODE_limit:5}
    \beta(\bm{x},\bm{\xi}) \coloneqq \beta^h(\bm{x},\bm{\xi})|_{h=1} =\tilde{\Theta}^h(\bm{x})\bigl|_{h=1}\bm{\xi} \eqqcolon \tilde{\Theta}(\bm{x})\bm{\xi}.
\end{equation}

Let $\bm{X}^{\sigma}_{\tau}$ represent the solution to 
\begin{equation}\label{eqn:app:existence_ODE_limit:7}
\dot{\bm{X}}^{\sigma}_{\tau} = \beta(\bm{X}^{\sigma}_{\tau},\bm{\xi}^{1}_{\tau/\sigma})=\tilde{\Theta}(\bm{X}^{\sigma}_{\tau})\bm{\xi}^{1}_{\tau/\sigma},  
\end{equation}
with initial condition $\bm{X}^{\sigma}_0 = \bm{x}_0$.

\begin{lemma}\label{lemma:existence_ODE_limit:h=1}
For any given $\bm{x}_0\in\mathbb{R}_0^{|\mathscr{K}|}$, and $\bar{T}>0$, there exists a unique solution $\bm{X}^{\sigma}_{\tau}$ for $\tau\in[0,\bar{T}]$ to \eqref{eqn:app:existence_ODE_limit:7}.
\end{lemma}
The proof of Lemma~\ref{lemma:existence_ODE_limit:h=1} is provided in Appendix~\ref{app:existence_ODE_limit:h=1}.

From \eqref{eqn:app:convergence_Z:3} and \eqref{eqn:app:convergence_Z:4}, for any $t\in[T]$, $\tau \geq 0$, and $\kappa\in\mathscr{K}$ with $t^{\kappa} = t$, 
\begin{itemize}
    \item if $\sum_{\kappa'\in\mathscr{K}:t^{\kappa'} = t}X^{\sigma}_{\kappa',\tau} = 0$, then for all $\kappa'\in\mathscr{K}$, $\lim_{a\downarrow 0}Q^{h,\Lipschitza}(\kappa,\kappa',\bm{X}^{\sigma}_{\tau},\bm{\xi}) = 0$; and
    \item if $\sum_{\kappa'\in\mathscr{K}:t^{\kappa'} = t-1}X^{\sigma}_{\kappa',\tau} = 0$, then for all $\kappa'\in\mathscr{K}$, $\lim_{a\downarrow 0}Q^{h,\Lipschitza}(\kappa',\kappa,\bm{X}^{\sigma}_{\tau},\bm{\xi}) = 0$.
\end{itemize}
In other words, if $\sum_{\kappa'\in\mathscr{K}: t^{\kappa'} \leq t}X^{\sigma}_{\kappa',\tau} = 0$, $\beta^h(\bm{X}^{\sigma}_{\tau'},\bm{\xi}^{h}_{\tau'/\sigma}) = 0$ for all $\tau' \geq \tau$, and hence $\sum_{\kappa'\in\mathscr{K}: t^{\kappa'} \leq t}X^{\sigma}_{\kappa',\tau'} = 0$ for all $\tau'\geq \tau$.
There exists a minimum $\tau$, for all $t\in[T]$, such that 
$\sum_{\kappa'\in\mathscr{K}: t^{\kappa'} \leq t}X^{\sigma}_{\kappa',\tau'} = 0$ for all $\tau'\geq \tau$.
For $t\in[T]$, we define 
\begin{equation}\label{eqn:app:convergence_Z:7}
    \tau^{\sigma}(t) \coloneqq \inf\Bigl\{\tau\in\mathbb{R}_0~\Bigl|~\sum_{\kappa\in\mathscr{K}:t^{\kappa}<t
    }X^{\sigma}_{\kappa,\tau} = 0\Bigr\},
\end{equation}
and $\tau^{\sigma}(0)=0$.

\begin{lemma}\label{lemma:convergence_b}
For any $U\in\mathbb{R}_0\cup\{\infty\}$, $\delta>0$, $\tau\in\mathbb{R}_0$, and $\bm{x}\in\mathbb{R}_0^{|\mathscr{K}|}$, there exists $\bar{b}^{\Lipschitza,U}(\bm{x})\coloneqq\mathbb{E}b^{\Lipschitza,U}(\bm{x},\bm{\xi}^{1}_{\tau}) = \tilde{\mathcal{Q}}^{1,\Lipschitza}(\bm{x})\mathbb{E}\min\{U,\bm{\xi}^{1}_{\tau}\}$
such that
\begin{equation}\label{eqn:lemma:convergence_b}
\lim_{\bar{T}\rightarrow +\infty}\mathbb{P}\biggl\{\Bigl\lVert \frac{1}{\bar{T}}\int_{\tau}^{\tau+\bar{T}}b^{\Lipschitza,U}(\bm{x},\bm{\xi}^{1}_{\tau'}) d\tau' -\bar{b}^{\Lipschitza,U}(\bm{x})\Bigr\rVert > \delta\biggr\}=0,
\end{equation}
uniformly for all $\tau\geq 0$. We recall that, since $\bm{\xi}^{1}_{\tau}$ is identically distributed for all $\tau\geq 0$, $\bar{b}^{\Lipschitza,U}(\bm{x})$ is independent for $\tau\geq 0$.
\end{lemma}

\proof{Proof of Lemma~\ref{lemma:convergence_b}.}
Let $\bar{b}^{\Lipschitza,U}(\bm{x})= \tilde{\mathcal{Q}}^{1,\Lipschitza}(\bm{x})\bm{\lambda}^U$ for $\bm{\lambda}^U \coloneqq \mathbb{E}\min\{U,\bm{\xi}^{1}_{\tau}\}$.
For any $U<\infty$, $\delta>0$, $\tau,\bar{T}\in\mathbb{R}_0$, and $\bm{x}\in\mathbb{R}_0^{|\mathscr{K}|}$,
\begin{multline}\label{eqn:lemma:convergence_b:1}
\mathbb{P}\biggl\{\Bigl\lVert \frac{1}{\bar{T}}\int_{\tau}^{\tau+\bar{T}}b^{\Lipschitza,U}(\bm{x},\bm{\xi}^{1}_{\tau'}) d\tau' -\bar{b}^{\Lipschitza,U}(\bm{x})\Bigr\rVert > \delta\biggr\} \\
=\mathbb{P}\biggl\{\Bigl\lVert \tilde{\mathcal{Q}}^{1,\Lipschitza}(\bm{x})\Bigl(\frac{1}{\bar{T}}\int_{\tau}^{\tau+\bar{T}}\min\bigl\{\bm{\xi}^{1}_{\tau'},U\bigr\} d\tau' -\bm{\lambda}^U\Bigr)\Bigr\rVert> \delta\biggr\}.
\end{multline}
Since $\bm{\xi}^{1}_{\tau}$ is identically distributed for all $\tau \geq 0$, 
\[\Bigl\lVert\frac{1}{\bar{T}}\int_{\tau}^{\tau+\bar{T}}\min\{U,\bm{\xi}^{1}_{\tau'}\} d\tau' -\bm{\lambda}^U\Bigr\rVert
 \sim \Bigl\lVert\frac{1}{\bar{T}}\int_{0}^{\bar{T}}\min\{U,\bm{\xi}^{1}_{\tau'}\} d\tau' -\bm{\lambda}^U\Bigr\rVert\]
where the  right hand side is independent from $\tau$.
Based on the Law of Large Numbers, we obtain \eqref{eqn:lemma:convergence_b} uniformly for all $\tau\geq0$.

For the case with $U\rightarrow\infty$, any $\delta>0$, $\tau,\bar{T}\in\mathbb{R}_0$, and $\bm{x}\in\mathbb{R}_0^{|\mathscr{K}|}$, we obtain
\begin{multline}\label{eqn:lemma:convergence_b:2}
\mathbb{P}\biggl\{\Bigl\lVert \frac{1}{\bar{T}}\int_{\tau}^{\tau+\bar{T}}b^{\Lipschitza,\infty}(\bm{x},\bm{\xi}^{1}_{\tau'}) d\tau' -\bar{b}^{\Lipschitza,\infty}(\bm{x})\Bigr\rVert > \delta\biggr\} \\
\leq \mathbb{P}\biggl\{\Bigl\lVert \tilde{\mathcal{Q}}^{1,\Lipschitza}(\bm{x})\frac{1}{\bar{T}}\Bigl\lfloor\int_0^{\bar{T}}\bm{\xi}^{1}_{\tau'} d\tau'\Bigr\rfloor -\bar{b}^{\Lipschitza,\infty}(\bm{x})\Bigr\rVert + \frac{o(\bar{T})}{\bar{T}}> \delta\biggr\}\\
=\mathbb{P}\biggl\{\Bigl\lVert \tilde{\mathcal{Q}}^{1,\Lipschitza}(\bm{x})\Bigl(\frac{1}{\bar{T}}\Bigl\lfloor\int_0^{\bar{T}}\bm{\xi}^{1}_{\tau'} d\tau'\Bigr\rfloor -\bm{\lambda}^{\infty}\Bigr)\Bigr\rVert + \frac{o(\bar{T})}{\bar{T}}> \delta\biggr\},
\end{multline}
where, for any vector $\bm{v}\in\mathbb{R}^M$, $\lfloor \bm{v}\rfloor \coloneqq \bigl(\lfloor v_m\rfloor: m\in[M]\bigr)$, $\bar{b}^{\Lipschitza,\infty} \coloneqq \lim_{U\rightarrow \infty}\bar{b}^{\Lipschitza,U}$, and $\bm{\lambda}^{\infty}\coloneqq \lim_{U\rightarrow \infty}\bm{\lambda}^U$. 
Based on the definition of $\bm{\xi}^{h}_{\tau}$, for any $\tau,\bar{T}\in\mathbb{R}_0$, each element of $\bigl\lfloor \int_0^{\bar{T}}\bm{\xi}^{1}_{\tau'}d\tau'\bigr\rfloor$ either is a deterministic value $\lfloor \bar{T}/2 \rfloor$ or follows a Poisson distribution with expectation $\bar{T}\lambda(\kappa,\kappa',n)$, where 
$\lambda(\kappa,\kappa',n) \leq     \scrp_{t^{\kappa}}(\zeta^{\kappa},\zeta^{\kappa'})$ for all $n\in[N_{i^{\kappa}}]$.
Let $\bm{\lambda}\coloneqq (\lfloor \bar{T}/2\rfloor; \lambda(\kappa,\kappa',n):\kappa,\kappa'\in\mathscr{K},n\in[N_{i^{\kappa}}])$.

From the Law of Large Numbers, plugging  $\bm{\lambda}^{\infty}=\bm{\lambda}$ in \eqref{eqn:lemma:convergence_b:2}, we obtain 
\begin{multline}\label{eqn:lemma:convergence_b:3}
\lim_{\bar{T}\rightarrow +\infty}\sup_{\tau\geq 0}\mathbb{P}\biggl\{\Bigl\lVert \frac{1}{\bar{T}}\int_{\tau}^{\tau+\bar{T}}b^{\Lipschitza,\infty}(\bm{x},\bm{\xi}^{1}_{\tau'}) d\tau' -\bar{b}^{\Lipschitza,\infty}(\bm{x})\Bigr\rVert > \delta\biggr\}\\
\leq\lim_{\bar{T}\rightarrow +\infty}\mathbb{P}\biggl\{\Bigl\lVert\tilde{\mathcal{Q}}^{1,\Lipschitza}(\bm{x})\Bigl(\frac{1}{\bar{T}}\Bigl\lfloor\int_0^{\bar{T}}\bm{\xi}^{1}_{\tau'}d\tau'\Bigr\rfloor - \bm{\lambda}\Bigr)\Bigr\rVert+\frac{o(\bar{T})}{\bar{T}} > \delta\biggr\}=0.
\end{multline}
That is, for $U\rightarrow \infty$, any $\delta>0$ and $\bm{x}\in\mathbb{R}_0^{|\mathscr{K}|}$, there exists $\bar{b}^{\Lipschitza,\infty}(\bm{x}) = \tilde{\mathcal{Q}}^{1,\Lipschitza}(\bm{x})\bm{\lambda}$ such that
\eqref{eqn:lemma:convergence_b} is satisfied uniformly for all $\tau\geq 0$. It proves the lemma.

\endproof

For $U\in\mathbb{R}_0\cup\{\infty\}$, define $\bar{\bm{x}}^{\Lipschitza,U}_{\tau}$ as the unique solution of \begin{equation}\label{eqn:averaging_ODE}
    \dot{\bar{\bm{x}}}^{\Lipschitza,U}_{\tau} = \bar{b}^{\Lipschitza,U}(\bar{\bm{x}}_{\tau}),
\end{equation}
with given $\bar{\bm{x}}^{\Lipschitza,U}_0\in\mathbb{R}^{|\mathscr{K}|}$.

\begin{lemma}\label{lemma:existence_ODE:averaging}
For any given $\bm{x}_0\in\mathbb{R}_0^{|\mathscr{K}|}$ and $\bar{T}>0$, there exists a unique solution $\bar{\bm{x}}^{\Lipschitza,U}_{\tau}$ for $\tau\in[0,\bar{T}]$ to \eqref{eqn:averaging_ODE}.
\end{lemma}
The proof of Lemma~\ref{lemma:existence_ODE:averaging} is provided in Appendix~\ref{app:existence_ODE:averaging}. 

From \cite[Theorem 2.1 in Chapter 7]{freidlin2012random}, if $\bar{\bm{x}}^{\Lipschitza,U}_t$ and $\bar{b}^{\Lipschitza,U}(\bm{x})$ exist and satisfy \eqref{eqn:lemma:convergence_b} uniformly for $\tau\geq 0$, and $\lVert b^{\Lipschitza,U}(\bm{x},\bm{\xi}^{1}_t)\rVert^2 < \infty$ for all $\bm{x}\in\mathbb{R}_0^{|\mathscr{K}|}$, then, for any $0<\bar{T}<\infty$ and $\delta >0$,
\begin{equation}\label{eqn:app:convergence_Z:8}
\lim_{\sigma\downarrow 0} \mathbb{P}\Bigl\{\sup_{0\leq \tau\leq \bar{T}}\bigl\lVert \bm{X}^{\sigma,\Lipschitza,U}_{\tau} - \bar{\bm{x}}^{\Lipschitza,U}_\tau\bigr\rVert > \delta\Bigr\}=0,
\end{equation}
where the initial point $\bm{X}^{\sigma,\Lipschitza,U}_0 = \bar{\bm{x}}^{\Lipschitza,U}_0$ is given.
From Lemma~\ref{lemma:convergence_b}, \eqref{eqn:app:convergence_Z:8} is achieved for all given $\Lipschitza\in(0,1)$ and $U\in\mathbb{R}_0$. 

Given the linearity of $b^{\Lipschitza,U}(\bm{x},\bm{\xi})$ in $\bm{\xi}$, we obtain
\begin{multline}\label{eqn:app:existence_ODE_limit:25}
\lim_{a\downarrow 0}\lim_{U\rightarrow \infty}\bar{b}^{\Lipschitza,U}(\bm{x}) = \lim_{a\downarrow 0}\lim_{U\rightarrow \infty}\mathbb{E}b^{\Lipschitza,U}(\bm{x},\bm{\xi}^{1}_{\tau}) = \lim_{a\downarrow 0}\lim_{U\rightarrow \infty}\tilde{\mathcal{Q}}^{1,\Lipschitza}(\bm{x})\mathbb{E}\min\{U,\bm{\xi}^{1}_{\tau}\}\\
    =\lim_{U\rightarrow \infty}\lim_{a\downarrow 0}\tilde{\mathcal{Q}}^{1,\Lipschitza}(\bm{x})\mathbb{E}\min\{U,\bm{\xi}^{1}_{\tau}\}
    =\lim_{U\rightarrow \infty}\lim_{a\downarrow 0}\mathbb{E}b^{U,\Lipschitza}(\bm{x},\bm{\xi}^{1}_{\tau})
    =\lim_{U\rightarrow \infty}\lim_{a\downarrow 0}\bar{b}^{\Lipschitza,U}(\bm{x}).
\end{multline}
Define
\begin{equation}\label{eqn:app:existence_ODE_limit:26}
    \bar{\beta}(\bm{x}) \coloneqq \lim_{\begin{subarray}~U\rightarrow \infty\\a\downarrow 0\end{subarray}}\bar{b}^{\Lipschitza,U}(\bm{x}_{\tau}).
\end{equation}

Let $\bar{\bm{x}}_{\tau}$ represent the solution to 
\begin{equation}\label{eqn:app:existence_ODE_limit:27}
    \dot{\bar{\bm{x}}}_{\tau} = \bar{\beta}(\bm{x}),
\end{equation}
with given $\bar{\bm{x}}_0 = \bm{x}_0$.

\begin{lemma}\label{lemma:existence_ODE_limit:averaging}
For any given $\bar{\bm{x}}_0 = \bm{x}_0\in\mathbb{R}_0^{|\mathscr{K}|}$, there exists a unique solution $\bar{\bm{x}}_{\tau}$ for $\tau\in[0,\bar{T}]$ to \eqref{eqn:app:existence_ODE_limit:27}.    
\end{lemma}
The proof of Lemma~\ref{lemma:existence_ODE_limit:averaging} is provided in Appendix~\ref{app:existence_ODE_limit:averaging}.

\begin{lemma}\label{lemma:continuity_trajectory_a:stochastic}
For any $\bar{T}>0$ and $\tau\in[0,\bar{T}]$, given the trajectory $\bm{\xi}^1_{\tau}$ for $\tau\in [0,\sigma\bar{T})$,
\begin{equation}\label{eqn:continuity_trajectory_a:stochastic:1}
    \lim_{\begin{subarray}~\Lipschitza\downarrow 0\\U\rightarrow \infty\end{subarray}} \bm{X}^{\sigma,a,U}_{\tau} = \bm{X}^{\sigma}_{\tau}, 
\end{equation}
and
\begin{equation}\label{eqn:continuity_trajectory_a:averaging:1}
    \lim_{\begin{subarray}~\Lipschitza\downarrow 0\\U\rightarrow \infty\end{subarray}} \bar{\bm{x}}^{\sigma,a,U}_{\tau} = \bar{\bm{x}}^{\sigma}_{\tau}. 
\end{equation}  
\end{lemma}
The proof of Lemma~\ref{lemma:continuity_trajectory_a:stochastic} is provided in Appendix~\ref{app:continuity_trajectory_a}.

From Lemma~\ref{lemma:continuity_trajectory_a:stochastic}, for any $\epsilon>0$, there exist $a_0>0$ and $U_0<\infty$ such that, for all $a\leq a_0$ and $U\geq U_0$,
\begin{multline}\label{eqn:app:swapping_limits:3} 
    \norm{\bm{X}^{\sigma}_{\tau} - \bar{\bm{x}}_{\tau}} 
    \leq \norm{\bm{X}^{\sigma}_{\tau} - \bm{X}^{\sigma,a_0,U_0}_{\tau}} + \norm{\bm{X}^{\sigma,a_0,U_0}_{\tau} - \bar{\bm{x}}^{a_0,U_0}_{\tau}}
    +\norm{\bar{\bm{x}}^{a_0,U_0}_{\tau} - \bar{\bm{x}}_{\tau}}\\
    \leq 2\epsilon + \norm{\bm{X}^{\sigma,a_0,U_0}_{\tau} - \bar{\bm{x}}^{a_0,U_0}_{\tau}}.
\end{multline}
Together with \eqref{eqn:app:convergence_Z:8}, for any $\delta>0$, there exist  $0<\epsilon < \delta/2$ and $a_0>0$ such that 
\begin{equation}\label{eqn:app:convergence_Z:9}
    \lim_{\sigma\downarrow 0}\mathbb{P}\Bigl\{\sup_{0\leq \tau \leq \bar{T}}\norm{\bm{X}^{\sigma}_{\tau} - \bar{\bm{x}}_{\tau}} > \delta\Bigr\} \leq \lim_{\sigma\downarrow 0}\mathbb{P}\Bigl\{\sup_{0\leq \tau \leq \bar{T}}\norm{\bm{X}^{\sigma,a_0,U_0}_{\tau} - \bar{\bm{x}}^{a_0,U_0}_{\tau}} > \delta-2\epsilon\Bigr\} =0.
\end{equation}

Let $\bm{X}^h_{\tau}$ represent the solution to
\begin{equation}\label{eqn:app:existence_ODE_limit:1}
    \dot{\bm{X}}^h_{\tau} = \frac{1}{h}b^h(h\bm{X}^h_{\tau},\bm{\xi}^{h}_{\tau}),
\end{equation}
with initial condition $\bm{X}^h_0 = \bm{x}_0$.
\begin{lemma}\label{lemma:existence_ODE_limit:h>1}
For any given $\bm{x}_0\in\mathbb{R}_0^{|\mathscr{K}|}$ and $\bar{T}>0$, there exists a unique solution $\bm{X}^h_{\tau}$ for $\tau\in[0,\bar{T}]$ to \eqref{eqn:app:existence_ODE_limit:1}.
\end{lemma}
The proof of Lemma~\ref{lemma:existence_ODE_limit:h>1} is provided in Appendix~\ref{app:existence_ODE_limit:h>1}.

We define an appropriate initial condition $\bm{x}_0\in\mathbb{R}^{|\mathscr{K}|}_0$ such that 
\begin{equation}\label{eqn:sim:1}
        x_{0,\kappa} = \begin{cases}
            \bm{Z}^{\orule,1}_{\zeta^{\kappa}}(0)\sum_{i\in[I]}N^0_i, &\text{if } t^{\kappa} = 0, \\
            0, &\text{otherwise},
        \end{cases}
\end{equation}

Given $\bm{X}^{\sigma}_0=\bm{X}^h_0=\bar{\bm{x}}_0= \sum_{i\in[I]}N_i^0\bm{\scrz}_0\coloneqq \bm{x}_0$ that satisfies \eqref{eqn:sim:1} and takes values in $\mathscr{N}_0^{|\mathscr{K}|}$, consider a trajectory \[\bm{\mathcal{Z}}^h_\tau=\frac{\bm{X}^h_\tau}{\sum_{i\in[I]}N_i^0},\] 
which satisfies
\begin{equation}\label{eqn:ODE_Z_h}
    \dot{\bm{\mathcal{Z}}}^h_\tau = \frac{1}{h\sum_{i\in[I]}N_i^0}\beta^h\bigl(h\sum_{i\in[I]}N_i^0\bm{\mathcal{Z}}^h_{\tau},\bm{\xi}^{h}_{\tau}\bigr),
\end{equation}
and $\bm{\mathcal{Z}}^h_0 = \bm{\scrz}_0$. 

\begin{lemma}\label{lemma:convergence-Z}
For $\tau \geq 0$,
\begin{equation}\label{eqn:app:convergence_Z:12}
    \frac{\bm{X}^{\sigma}_\tau}{\sum_{i\in[I]}N_i^0}\sim \bm{\mathcal{Z}}^h_\tau.
\end{equation}   
For any $0<\bar{T}<\infty$ and $\delta>0$, 
\begin{equation}\label{eqn:app:convergence_Z:13}
    \lim_{h\rightarrow \infty}\mathbb{P}\biggl\{\sup_{0\leq \tau\leq \bar{T}}\Bigl\lVert \bm{\mathcal{Z}}^h_{\tau} - \frac{\bar{\bm{x}}_{\tau}}{\sum_{i\in[I]}N_i^0}\Bigr\rVert > \delta\biggr\} = 0,
\end{equation}
where recall that $\bar{\bm{x}}_{\tau}$ is a deterministic trajectory when given $\bar{\bm{x}}_0\in\mathbb{R}^{|\mathscr{K}|}$. 
\end{lemma}

\proof{Proof of Lemma~\ref{lemma:convergence-Z}.}
Let $\sigma = 1/h$, where $h$ is the scaling parameter.
For any $\tau,\bar{T}\in\mathbb{R}_0$,
\begin{equation}\label{eqn:app:convergence_Z:10}
\int_\tau^{\tau+\bar{T}} \beta(\bm{X}^{\sigma}_{\tau},\bm{\xi}^{1}_{\tau/\sigma})d\tau = \sigma\int_{\tau/\sigma}^{(\tau+\bar{T})/\sigma}\beta(\bm{X}^{\sigma}_{\sigma\tau},\bm{\xi}^{1}_\tau)d\tau = \frac{1}{h}\int_{h\tau}^{h(\tau+\bar{T})}\beta(\bm{X}^{\sigma}_{\tau/h},\bm{\xi}^{1}_\tau)d\tau.
\end{equation}

Given $h\in\mathbb{N}_+$ and $\bm{X}^{\sigma}_0=\bar{\bm{x}}_0=\bm{x}_0\in\mathbb{N}_0^{|\mathscr{K}|}$, consider a trajectory $\bm{X}^h_\tau$ which satisfies \eqref{eqn:app:existence_ODE_limit:1}.
Define a set $\mathscr{X}^h \subset \mathbb{R}_0^{|\mathscr{K}|}$ where, for any $\bm{x}=(x_{\kappa}:\kappa\in\mathscr{K})\in\mathscr{X}^h$, each element $x_{\kappa}$ can be rewritten as $\frac{n_{\kappa}}{h}$ with $n_{\kappa}\in\mathbb{N}_0$. 
Based on the definitions of $\bm{\xi}^{h}_\tau$ and $\lim_{a\downarrow 0}Q^{h,\Lipschitza}$, for $t\in[T]_0$, and $\bar{T} = \tau_{t}(m)$ for some $m\in\mathbb{N}_0$, $\bm{X}^{\sigma}_{\bar{T}}$ and $\bm{X}^h_{\bar{T}}$ must take some values in $\mathscr{X}^h$.
In particular, for $\kappa\in\mathscr{K}$, given $\tau_m=\tau_{t^{\kappa}}(m)$ for $m\in\mathbb{N}_0$,
if $\bm{X}^{\sigma}_{\tau_m} = \bm{X}^h_{\tau_m}=\bm{x}$ with  $\bm{x}\in\mathscr{X}^h$ and $x_{\kappa}>0$, then, let $\bar{T}^{\sigma}$ and $\bar{T}^h$ represent the time such that $\lvert \int_{h\tau_m}^{h(\tau_m+\bar{T}^{\sigma})}\beta_{\kappa}(\bm{X}^{\sigma}_{\tau/h},\bm{\xi}^{1}_{\tau})d\tau\rvert= 1$ and $\lvert \int_{\tau_m}^{\tau_m+\bar{T}^h}\beta^h_{\kappa}(h\bm{X}^h_{\tau},\bm{\xi}^{h}_{\tau})d\tau\rvert= 1$, respectively, that are exponentially distributed with rate $h\lceil x_{\kappa}/h\rceil\sum_{\kappa'\in\mathscr{K}:t^{\kappa'}={t^{\kappa}+1}}\scrp_{t^{\kappa}}(\zeta^{\kappa},\zeta^{\kappa'})$, we have $\bar{T}^{\sigma}\sim \bar{T}^h$, where $\sim$ means equivalence in distribution.
Hence, for given $\bar{T}\in\mathbb{R}_0$, if $\bm{X}^{\sigma}_{\tau_m} = \bm{X}^h_{\tau_m}=\bm{x}$, then, for any $\kappa\in\mathscr{K}$,
\begin{equation}\label{eqn:app:convergence_Z:11}
    \frac{1}{h}\int_{h\tau_m}^{h(\tau_m+\bar{T})}\beta_{\kappa}(\bm{X}^{\sigma}_{\tau/h},\bm{\xi}^{1}_{\tau})d\tau \sim \frac{1}{h}\int_{\tau_m}^{\tau_m+\bar{T}}\beta^h_{\kappa}(h\bm{X}^h_{\tau},\bm{\xi}^{h}_{\tau})d \tau,
\end{equation}
leading to $\bm{X}^{\sigma}_\tau \sim \bm{X}^h_\tau$ for all $\tau\in\mathbb{R}_0$.

From \eqref{eqn:app:convergence_Z:11}, we obtain \eqref{eqn:app:convergence_Z:12}.
Together with \eqref{eqn:app:convergence_Z:9}, for any $0<\bar{T}<\infty$ and $\delta>0$, \eqref{eqn:app:convergence_Z:13} holds.

\endproof

Similar to \eqref{eqn:app:convergence_Z:7}, for a time point of the real process $t\in[T]_0$, define
\begin{equation}\label{eqn:app:convergence_Z:15}
\tau^h(t)\coloneqq \inf\Bigl\{\tau\in\mathbb{R}_0~\Bigl|~\sum_{\kappa\in\mathscr{K}:t^{\kappa}<t
}\mathcal{Z}^h_{\kappa,\tau} = 0\Bigr\},
\end{equation}
where $\tau^h(0)=0$.

\begin{lemma}\label{lemma:sim}
For any $\delta>0$, $\bm{\mathcal{Z}}^h_0= \bm{\scrz}_0\coloneqq \bm{x}_0/\sum_{i\in[I]}N_i^0$ that satisfies \eqref{eqn:sim:1},
\begin{equation}
    \lim_{h\to \infty} \bbP\Bigl\{\max_{t\in[T]_0, \zeta\in\calJ}\norm{Z^{\orule,h}_{\zeta}(t)-\mathcal{Z}^h_{\tau^h(t),\kappa(\zeta,t)}} > \delta\Bigr\} = 0,
\end{equation}
where $\kappa(\zeta,t)$ is the $\kappa\in\mathscr{K}$ such that $\zeta^{\kappa} = \zeta$ and $t^{\kappa}=t$.
\end{lemma}
The proof of Lemma~\ref{lemma:sim} is provided in Appendix~\ref{app:lemma:sim}.


\proof{Proof of Theorem~\ref{theorem:convergence-Z}.}
From Lemma~\ref{lemma:convergence-Z}, given an initial point $\bm{\mathcal{Z}}^h_0=\bar{\bm{x}}_0/\sum_{i\in[I]}N^0_i$ that satisfies \eqref{eqn:sim:1}, for any $0\leq \tau \leq \bar{T}$,
\begin{equation}\label{eqn:app:convergence_Z:14}
    \lim_{h\rightarrow \infty}\mathbb{E}\Bigl[\bm{\mathcal{Z}}^h_{\tau}\Bigr]= \frac{\bar{\bm{x}}_{\tau}}{\sum_{i\in[I]}N^0_i}.
\end{equation}
Together with Lemma~\ref{lemma:sim}, for any $t\in[T]_0$ and $\zeta\in\calJ$,
\begin{equation}\label{eqn:app:convergence_Z:15:half}
    \lim_{h\to\infty} \bbE^{\orule}_{\bm{y}^0} Z^{\orule,h}_{\zeta}(t) = \lim_{h\rightarrow \infty}\mathbb{E}\mathcal{Z}^h_{\tau^h(t), \kappa(\zeta,t)}= \frac{\mathbb{E}\bar{x}_{\tau^h(t), \kappa(\zeta,t)}}{\sum_{i\in[I]}N^0_i}.
\end{equation}

From Lemmas~\ref{lemma:convergence-Z} and \ref{lemma:sim},
for $\delta > 0$, 
\begin{multline}
    \lim_{h\to\infty} \bbP\Bigl\{\max_{t\in[T]_0,\zeta\in\calJ}\norm{Z^{\orule,h}_{\zeta}(t) - \frac{\bar{x}_{\tau^h(t),\kappa(\zeta,t)}}{\sum_{i\in[I]}N^0_i}} > \delta\Bigr\}\\
    \leq \lim_{h\to\infty} \bbP\Biggl\{\max_{t\in[T]_0,\zeta\in\calJ}\Bigl\{\norm{Z^{\orule,h}_{\zeta}(t)-\mathcal{Z}^h_{\tau,\kappa(\zeta,t)}}+\norm{\mathcal{Z}^h_{\tau^h(t),\kappa(\zeta,t)} - \frac{\bar{x}_{\tau^h(t),\kappa(\zeta,t)}}{\sum_{i\in[I]}N^0_i}}\Bigr\} > \delta\Biggr\}=0.
\end{multline}
Recall \eqref{eqn:app:convergence_Z:15:half}, it proves the theorem.

\endproof

\section{Continuity of $Q^{h,\Lipschitza}$ defined in \eqref{eqn:app:convergence_Z:3}}\label{app:continuity_Q}

For $\Lipschitza\in(0,1)$ and $u\in[0,1]$, define
\begin{equation}\label{eqn:dirac-delta:1}
y_{\Lipschitza}(u)\coloneqq \begin{cases}
\int_{-\infty}^{\rho(u)}\frac{1}{\Lipschitza\sqrt{\pi}}e^{-(v-\frac{1}{\Lipschitza})^2/\Lipschitza^2} dv,&\text{if }u\in(0,1),\\
0, & \text{if }u=0,\\
1, & \text{if } u=1,
\end{cases}
\end{equation}
where $\rho(u)=-\cot(u\pi)$. The key features we need of $\rho(u)$ are that $\rho(u)\rightarrow -\infty$ for $u\rightarrow 0$, $\rho(u)\rightarrow +\infty$ for $u\rightarrow 1$, and is suitably smooth.
In this context, $y_{\Lipschitza}(u)$ is continuous on $u\in[0,1]$, and, for $u\in(0,1)$, the there exists $K_a <\infty$ such that derivative
\begin{equation}\label{eqn:dirac-delta:2}
    \frac{d y_{\Lipschitza}}{d u} = \frac{1}{\Lipschitza\sqrt{\pi}}e^{-(\cot{u\pi}+\frac{1}{\Lipschitza})^2/\Lipschitza^2}\frac{\pi}{\sin^2{u\pi}}  < K_a < \infty.
\end{equation}

For $u\in(0,1)$,
\begin{equation}\label{eqn:dirac-delta:6}
\frac{d y_{\Lipschitza}(u)}{d\Lipschitza} = \frac{d}{d \Lipschitza} \Phi\Bigl(\frac{\rho(u) - \frac{1}{a}}{\Lipschitza/\sqrt{2}}\Bigr) = \frac{1}{\sqrt{\pi}}\Bigl(\frac{2}{\Lipschitza^3} - \frac{\rho(u)}{\Lipschitza^2}\Bigr)e^{-\bigl(\frac{\rho(u)-\frac{1}{\Lipschitza}}{\Lipschitza}\bigr)^2},
\end{equation}
where $\Phi(x)$ is the cumulative distribution function of the standard normal distribution.
We have
\begin{equation}
    \lim_{\Lipschitza \downarrow 0} \Bigl\lvert \frac{d y_{\Lipschitza}(u)}{d \Lipschitza}\Bigr\rvert = 0,
\end{equation}
and, for any $\Lipschitza\in[0,1)$, there exists $K_{\Lipschitza} < \infty$ such that 
\begin{equation}
    \Bigl\lvert \frac{d y_{\Lipschitza}(u)}{d \Lipschitza}\Bigr\rvert < K_{\Lipschitza}.
\end{equation}


Define 
\begin{equation}\label{eqn:x_t}
    x(t) = \sum_{\kappa''\in\mathscr{K}:t^{\kappa''} \leq t}|x_{\kappa''}|,
\end{equation}
\begin{equation}\label{eqn:xa_t}
    x^{\mu}(t) = \sum_{\kappa''\in\mathscr{K}:t^{\kappa''} = t, \action^{\kappa''} = \dummyAction}|x_{\kappa''}|,
\end{equation}
and $x(t)\equiv x^{\mu}(t) \equiv 0 $ for all $t<0$.



We define, for $h=1$,
\begin{multline}\label{eqn:dirac-delta:3}
f^{h,\Lipschitza}_{\kappa,\kappa'}(\bm{x},\bm{\xi}) = 
\indicator\biggl\{0<x(t^{\kappa}-1)<1 \text{ or } 0<x^{\mu}(t^{\kappa}) <1\biggr\} \Bigl(\sum_{n\in \bigl[\lceil x_{\kappa} \rceil\bigr]}\xi(\kappa,\kappa',n)\\
-\indicator\Bigl\{ i^{\kappa}=i^{\kappa'},s^{\kappa}=s^{\kappa'}, \action^{\kappa} = \dummyAction,t^{\kappa}=t^{\kappa'}\Bigr\}\sum_{\begin{subarray}~\kappa''\in\mathscr{K}:\\i^{\kappa''}=i^{\kappa},\\s^{\kappa''}=s^{\kappa},\\t^{\kappa''}=t^{\kappa}\end{subarray}}x_{\kappa''}\scra^{\orule}_{\zeta^{\kappa'}}(\frac{\bm{x}}{\sum_{i\in[I]}N_i^0},t^{\kappa'})\xi(0)+\Delta^a_{\kappa,\kappa'}(\bm{x})\bm{\xi}\Bigr)\\
\times y_{\Lipschitza}\bigl(1- x(t^{\kappa}-1)\bigr)y_{\Lipschitza}\bigl(1- x^{\mu}(t^{\kappa})\bigr) \\
+\indicator\biggl\{x(t^{\kappa}-1)=0, x^{\mu}(t^{\kappa})=0\biggr\}
\Delta^a_{\kappa,\kappa'}(\bm{x})\bm{\xi}\\
+\indicator\biggl\{0<x(t^{\kappa}-1)<1\biggr\}\indicator\Bigl\{i^{\kappa}=i^{\kappa'},s^{\kappa}=s^{\kappa'},a^{\kappa}=\dummyAction,t^{\kappa}=t^{\kappa'}\Bigr\}\\
\times\sum_{\begin{subarray}~\kappa''\in\mathscr{K}:\\i^{\kappa''}=i^{\kappa},\\s^{\kappa''}=s^{\kappa},\\t^{\kappa''}=t^{\kappa}\end{subarray}}x_{\kappa''}\scra^{\orule}_{\zeta^{\kappa'}}(\frac{\bm{x}}{\sum_{i\in[I]}N_i^0},t^{\kappa'})\xi(0)y_{\Lipschitza}(1-x^{\mu}(t^{\kappa})),
\end{multline}
where $\Delta^a_{\kappa,\kappa'}(\bm{x})\bm{\xi} \coloneqq  
-\xi\bigl(\kappa,\kappa',\lceil x_{\kappa} \rceil\bigr)y_{\Lipschitza}\bigl(\norm{x_{\kappa}}_+\bigr) $
with $\norm{x}_+\coloneqq \lceil x \rceil-x$,  $\bm{x}\in\mathbb{R}_0^{|\mathscr{K}|}$, $\bm{\xi}=(\xi(\kappa,\kappa',n):\kappa,\kappa'\in\mathscr{K},n\in[N_{i^{\kappa}}])\in \mathbb{R}_0^{|\mathscr{K}|\sum_{\kappa\in\mathscr{K}}N_{i^{\kappa}}}$, $\xi(\kappa,\kappa',0)\coloneqq 0$,  and $y_a(u)$ defined in \eqref{eqn:dirac-delta:1}.
Here, based on the definition in \eqref{eqn:x_t}, $x(t^{\kappa}-1) = \sum_{\kappa'\in\mathscr{K}:t^{\kappa'}< t^{\kappa}} |x_{\kappa'}|$.
This  $f^{h,\Lipschitza}_{\kappa,\kappa'}(\bm{x},\bm{\xi})$ is linear in $\bm{\xi}$, continuous in $\Lipschitza\in(0,1)$, and, there exist $K>0$ (independent from $\Lipschitza$) such that
\begin{multline}
\lim_{\Lipschitza\downarrow 0} f^{h,\Lipschitza}_{\kappa,\kappa'}(\bm{x},\bm{\xi})\leq \lim_{\Lipschitza\downarrow 0} K \Bigl(\indicator\bigl\{0<x(t^{\kappa}-1)<1 \text{ or } 0< x^{\mu}(t^{\kappa}) <1\bigr\}y_a(1-x(t^{\kappa}-1))y_a(1-x^{\mu}(t^{\kappa})) \\ + y_a(\norm{ x_{\kappa}}_+)+y_{\Lipschitza}(1-x^{\mu}(t^{\kappa}))\Bigr)  =  0.
\end{multline}
There exist $K_1,K_2,K_3 <\infty$ such that the derivative 
\begin{equation}\label{eqn:derivative_f_to_a}
\lvert\frac{d f^{h,\Lipschitza}_{\kappa,\kappa'}}{d\Lipschitza}\rvert \leq K_1 \bigl\lvert \frac{d y_{\Lipschitza}}{d\Lipschitza} \bigr\rvert + K_2 \bigl\lvert \frac{d \Delta^a_{\kappa,\kappa'}}{d\Lipschitza}\bigr\rvert \leq (K_1+K_3)  \bigl\lvert \frac{d y_{\Lipschitza}}{d\Lipschitza} \bigr\rvert,
\end{equation}
where $K_1,K_2,K_3$ are dependent on $\bm{\xi}$ and $\bm{x}$ but independent from $\Lipschitza\in(0,1)$. 
Based on \eqref{eqn:dirac-delta:6}, $\lim_{\Lipschitza\downarrow 0}\bigl\lvert \frac{d y_{\Lipschitza}}{d\Lipschitza} \bigr\rvert = 0$, leading to $\lim_{\Lipschitza\downarrow 0}\bigl\lvert \frac{d f^{h,\Lipschitza}_{\kappa,\kappa'}}{d\Lipschitza} \bigr\rvert = 0$.

We continue discussing the Lipschitz continuity of $Q^{1,\Lipschitza}$.
In this proof, we consider given $\kappa,\kappa'\in\mathscr{K}$.

In the following, we start with the continuity of $Q^{1,\Lipschitza}(\kappa,\kappa',\bm{x},\bm{\xi})$ in $\bm{x}\in\mathbb{R}_0^{|\mathscr{K}|}$ (right-continuous at the zero point) for given $\kappa,\kappa',\bm{\xi}$. 
For $h=1$, we plug $f^{h,\Lipschitza}_{\kappa,\kappa'}(\bm{x},\bm{\xi})$ defined in \eqref{eqn:dirac-delta:3} into \eqref{eqn:app:convergence_Z:3}, obtaining
\begin{multline}\label{eqn:app:conitinuity_Q:1}
Q^{1,\Lipschitza}(\kappa,\kappa',\bm{x},\bm{\xi}) 
= \indicator\bigl\{0<x(t^{\kappa}-1)<1 \text{ or } 0<x^{\mu}(t^{\kappa})<1\bigr\}\\
\times\Bigl(\sum_{n\in\bigl[\lceil x_{\kappa} \rceil\bigr]}\xi(\kappa,\kappa',n)-\indicator\Bigl\{ i^{\kappa}=i^{\kappa'},s^{\kappa}=s^{\kappa'}, \action^{\kappa} = \dummyAction,t^{\kappa}=t^{\kappa'}\Bigr\}\\
\times\sum_{\begin{subarray}~\kappa''\in\mathscr{K}:\\i^{\kappa''}=i^{\kappa},\\s^{\kappa''}=s^{\kappa},\\t^{\kappa''}=t^{\kappa}\end{subarray}}x_{\kappa''}\scra^{\orule}_{\zeta^{\kappa'}}(\frac{\bm{x}}{\sum_{i\in[I]}N_i^0},t^{\kappa'})\xi(0)+\Delta^a_{\kappa,\kappa'}(\bm{x})\bm{\xi}\Bigr)\\
\times y_{\Lipschitza}\bigl(1- x(t^{\kappa}-1))\bigr)y_{\Lipschitza}\bigl(1- x^{\mu}(t^{\kappa}))\bigr) \\
+\indicator\bigl\{x(t^{\kappa}-1)=0,x^{\mu}(t^{\kappa})=0\bigr\}\\
\times\Bigl(\sum_{n\in\bigl[\lceil x_{\kappa} \rceil\bigr]}\xi(\kappa,\kappa',n)-\indicator\Bigl\{ i^{\kappa}=i^{\kappa'},s^{\kappa}=s^{\kappa'}, \action^{\kappa} = \dummyAction,t^{\kappa}=t^{\kappa'}\Bigr\}\\
\times\sum_{\begin{subarray}~\kappa''\in\mathscr{K}:\\i^{\kappa''}=i^{\kappa},\\s^{\kappa''}=s^{\kappa},\\t^{\kappa''}=t^{\kappa}\end{subarray}}x_{\kappa''}\scra^{\orule}_{\zeta^{\kappa'}}(\frac{\bm{x}}{\sum_{i\in[I]}N_i^0},t^{\kappa'})\xi(0)+\Delta^a_{\kappa,\kappa'}(\bm{x})\bm{\xi}\Bigr)\\
+\indicator\bigl\{0<x(t^{\kappa}-1)<1\bigr\}\indicator\Bigl\{ i^{\kappa}=i^{\kappa'},s^{\kappa}=s^{\kappa'}, \action^{\kappa} = \dummyAction,t^{\kappa}=t^{\kappa'}\Bigr\}\sum_{\begin{subarray}~\kappa''\in\mathscr{K}:\\i^{\kappa''}=i^{\kappa},\\s^{\kappa''}=s^{\kappa},\\t^{\kappa''}=t^{\kappa}\end{subarray}}x_{\kappa''}\scra^{\orule}_{\zeta^{\kappa'}}(\frac{\bm{x}}{\sum_{i\in[I]}N_i^0},t^{\kappa'})\xi(0)\\
+\indicator\bigl\{x(t^{\kappa}-1)=0\bigr\}\indicator\Bigl\{ i^{\kappa}=i^{\kappa'},s^{\kappa}=s^{\kappa'}, \action^{\kappa} = \dummyAction,t^{\kappa}=t^{\kappa'}\Bigr\}\sum_{\begin{subarray}~\kappa''\in\mathscr{K}:\\i^{\kappa''}=i^{\kappa},\\s^{\kappa''}=s^{\kappa},\\t^{\kappa''}=t^{\kappa}\end{subarray}}x_{\kappa''}\scra^{\orule}_{\zeta^{\kappa'}}(\frac{\bm{x}}{\sum_{i\in[I]}N_i^0},t^{\kappa'})\xi(0)y_{\Lipschitza}(1-x(t^{\kappa}-1))
\end{multline}
where $\Delta^a_{\kappa,\kappa'}(\bm{x})\bm{\xi} = 
-\xi\bigl(\kappa,\kappa',\lceil x_{\kappa} \rceil\bigr)y_{\Lipschitza}\bigl(\norm{ x_{\kappa}}_+\bigr) $.

Since $\lim_{x(t^{\kappa}-1)\uparrow 1} Q^{1,\Lipschitza}(\kappa,\kappa',\bm{x},\bm{\xi}) = 0$,
and for all $x(t^{\kappa}-1)\geq 1$, $Q^{1,\Lipschitza}(\kappa,\kappa',\bm{x},\bm{\xi}) = 0$, $Q^{1,\Lipschitza}(\kappa,\kappa',\bm{x},\bm{\xi})$ is continuous for all $x(t^{\kappa}-1) \geq 1$.
Similarly, $Q^{1,\Lipschitza}(\kappa,\kappa',\bm{x},\bm{\xi})=\lim_{x^{\mu}(t^{\kappa})\uparrow 1} Q^{1,\Lipschitza}(\kappa,\kappa',\bm{x},\bm{\xi}) $, and thus $Q^{1,\Lipschitza}$ is continuous for all $x^{\mu}(t^{\kappa})\geq 1$. 
We then consider the case with $0\leq x(t^{\kappa}-1) < 1$ or $0\leq x^{\mu}(t^{\kappa}) < 1$.
In this case, \eqref{eqn:app:conitinuity_Q:1} becomes
\begin{multline}\label{eqn:app:conitinuity_Q:3}
    Q^{1,\Lipschitza}(\kappa,\kappa',\bm{x},\bm{\xi}) 
= \\\Bigl(\sum_{n\in\bigl[\lceil x_{\kappa} \rceil\bigr]}\xi(\kappa,\kappa',n)-\indicator\Bigl\{ i^{\kappa}=i^{\kappa'},s^{\kappa}=s^{\kappa'}, \action^{\kappa} = \dummyAction,t^{\kappa}=t^{\kappa'}\Bigr\}\sum_{\begin{subarray}~\kappa''\in\mathscr{K}:\\i^{\kappa''}=i^{\kappa},\\s^{\kappa''}=s^{\kappa},\\t^{\kappa''}=t^{\kappa}\end{subarray}}x_{\kappa''}\scra^{\orule}_{\zeta^{\kappa'}}(\frac{\bm{x}}{\sum_{i\in[I]}N_i^0},t^{\kappa'})\xi(0)+\Delta^a_{\kappa,\kappa'}(\bm{x})\bm{\xi}\Bigr)\\
\times \Biggl(\indicator\bigl\{x(t^{\kappa}-1)=0,x^{\mu}(t^{\kappa})=0\bigr\} + \Bigl(1-\indicator\bigl\{x(t^{\kappa}-1)=0,x^{\mu}(t^{\kappa})=0\bigr\}\Bigr)y_{\Lipschitza}(1-x(t^{\kappa}-1))y_{\Lipschitza}(1-x^{\mu}(t^{\kappa}))\Biggr)\\
+\indicator\Bigl\{ i^{\kappa}=i^{\kappa'},s^{\kappa}=s^{\kappa'}, \action^{\kappa} = \dummyAction,t^{\kappa}=t^{\kappa'}\Bigr\}\sum_{\begin{subarray}~\kappa''\in\mathscr{K}:\\i^{\kappa''}=i^{\kappa},\\s^{\kappa''}=s^{\kappa},\\t^{\kappa''}=t^{\kappa}\end{subarray}}x_{\kappa''}\scra^{\orule}_{\zeta^{\kappa'}}(\frac{\bm{x}}{\sum_{i\in[I]}N_i^0},t^{\kappa'})\xi(0)\\
\times \Biggl(\indicator\bigl\{x(t^{\kappa}-1) = 0\bigr\}+\Bigl(1-\indicator\Bigl\{x(t^{\kappa}-1) = 0\Bigr\}\Bigr)y_{\Lipschitza}(1-x(t^{\kappa}-1))\Biggr).
\end{multline}
For the first term in the right hand side of \eqref{eqn:app:conitinuity_Q:3}, we define
\begin{multline}\label{eqn:app:conitinuity_Q:4:1}
G^1_a(\bm{x},\bm{\xi})\coloneqq y_{\Lipschitza}\bigl(1- x(t^{\kappa}-1)\bigr)y_{\Lipschitza}\bigl(1- x^{\mu}(t^{\kappa})\bigr)\Bigl(1-\indicator\bigl\{x(t^{\kappa}-1)=0, x^{\mu}(t^{\kappa})=0\bigr\}\Bigr)+\indicator\bigl\{x(t^{\kappa}-1)=0,x^{\mu}(t^{\kappa})=0\bigr\}\\
=\begin{cases}
1, &\text{if } x(t^{\kappa}-1) = 0, x^{\mu}(t^{\kappa})=0,\\
y_{\Lipschitza}\bigl(1- x(t^{\kappa}-1)\bigr)y_{\Lipschitza}\bigl(1- x^{\mu}(t^{\kappa})\bigr), &\text{otherwise}.
\end{cases}
\end{multline}
Together with $\lim_{\begin{subarray}~x(t^{\kappa}-1)\downarrow 0,\\x^{\mu}(t^{\kappa})\downarrow 0\end{subarray}}y_{\Lipschitza} \bigl(1- x(t^{\kappa}-1)\bigr) y_{\Lipschitza} \bigl(1- x^{\mu}(t^{\kappa})\bigr)= 1$, $G^1_a(\bm{x},\bm{\xi})$ is continuous in $x(t^{\kappa}-1)= 0$ and $x^{\mu}(t^{\kappa})= 0$ and continuous for all $x(t^{\kappa}-1) \in (0,1)$ or $x^{\mu}(t^{\kappa}) \in (0,1)$.
Also, since there exists $K<\infty$ such that $\lvert d y_a(u)/d u\rvert < K$ (based on \eqref{eqn:dirac-delta:2}), $G^1_a(\bm{x},\bm{\xi})$ is Lipschitz continuous for all $x(t^{\kappa}-1) \in [0,1)$ or $x^{\mu}(t^{\kappa}) \in [0,1)$.

For the second term in the right hand side of \eqref{eqn:app:conitinuity_Q:3}, consider
\begin{multline}\label{eqn:app:conitinuity_Q:4:2}
G^2_a(\bm{x},\bm{\xi})\coloneqq y_{\Lipschitza}\bigl(1- x(t^{\kappa}-1)\bigr)\Bigl(1-\indicator\bigl\{x(t^{\kappa}-1)=0\bigr\}\Bigr)+\indicator\bigl\{x(t^{\kappa}-1)=0\bigr\}\\
=\begin{cases}
1, &\text{if } x(t^{\kappa}-1) = 0,\\
y_{\Lipschitza}\bigl(1- x(t^{\kappa}-1)\bigr), &\text{otherwise}.
\end{cases}
\end{multline}
Similarly, $\lim_{x(t^{\kappa}-1)\downarrow 0}y_{\Lipschitza} \bigl(1- x(t^{\kappa}-1)\bigr) = 1$, $G^2_a(\bm{x},\bm{\xi})$ is right-continuous in $x(t^{\kappa}-1)= 0$ and continuous for all $x(t^{\kappa}-1) \in (0,1)$.
Together with \eqref{eqn:dirac-delta:2}, $G^2_a(\bm{x},\bm{\xi})$ is Lipschitz continuous in $x(t^{\kappa}-1) \in [0,1)$.

We then discuss the continuity of $\sum_{n\in\bigl[\lceil x_{\kappa}\rceil\bigr]}\xi(\kappa,\kappa',n) + \Delta^a_{\kappa,\kappa'}(\bm{x})\bm{\xi}$.
We expand it as 
\begin{equation}\label{eqn:app:conitinuity_Q:5}
 Y_a(\bm{x},\bm{\xi})\coloneqq  \sum_{n\in\bigl[\lceil x_{\kappa}\rceil\bigr]}\xi(\kappa,\kappa',n) + \Delta^a_{\kappa,\kappa'}(\bm{x})\bm{\xi}\\ = \sum_{n\in\bigl[\lceil x_{\kappa}\rceil\bigr]}\xi(\kappa,\kappa',n)
-\xi\bigl(\kappa,\kappa',\lceil x_{\kappa} \rceil\bigr)y_{\Lipschitza}\bigl(\norm{ x_{\kappa}}_+\bigr),
\end{equation}
for which the problematic points are those with $x_{\kappa} = \lceil x_{\kappa}\rceil$.


Consider the limit where $x_{\kappa}\downarrow n_1$ for $ n_1\in\mathbb{N}_0$. 

From \eqref{eqn:app:conitinuity_Q:5}, 
\begin{multline}\label{eqn:app:conitinuity_Q:6}
  \lim_{x_{\kappa}\downarrow n_1} Y_a(\bm{x},\bm{\xi})
  =\lim_{x_{\kappa}\downarrow n_1}\Biggl(\sum_{n\in[n_1] }\xi(\kappa,\kappa',n) + \xi(\kappa,\kappa',n_1+1)
  -\xi(\kappa,\kappa',n_1+1)y_\Lipschitza(\norm{ x_{\kappa}}_+)\Biggr) \\
  = \lim_{x_{\kappa}\downarrow n_1}\sum_{n\in[n_1] }\xi(\kappa,\kappa',n) +\xi_{\kappa,\kappa'}(n_1+1) - \xi(\kappa,\kappa',n_1+1)\lim_{u\uparrow 1}y_\Lipschitza(u) = \sum_{n\in[n_1] }\xi(\kappa,\kappa',n)\\
  = Y_a(\bm{x},\bm{\xi})\Bigl|_{x_{\kappa} = n_1}.
\end{multline}

Along similar lines, for $x_{\kappa} \uparrow n_1\in\mathbb{N}_+$, we obtain
\begin{multline}\label{eqn:app:conitinuity_Q:12}
    \lim_{x_{\kappa}\uparrow n_1} Y_a(\bm{x},\bm{\xi}) = 
    \lim_{x_{\kappa}\uparrow n_1}\biggl(\sum_{n\in[n_1]}\xi(\kappa,\kappa',n) - \xi(\kappa,\kappa',n_1)y_\Lipschitza(\norm{ x_{\kappa}}_+)\biggr)\\
    =\lim_{x_{\kappa}\uparrow n_1}\sum_{n\in[n_1]}\xi(\kappa,\kappa',n) - \xi(\kappa,\kappa',n_1)\lim_{u\downarrow 0}y_\Lipschitza(u)
    = Y_a(\bm{x},\bm{\xi})\Bigl|_{x_{\kappa}=n_1}.
\end{multline}

That is, $Y_a(\bm{x},\bm{\xi})$ is continuous in all $x_{\kappa} \geq 0$.
Recall that, based on the definition of $Y_a(\bm{x},\bm{\xi})$ in \eqref{eqn:app:conitinuity_Q:5}, it is dependent on $\bm{x}$ through only $x_{\kappa}$. $Y_a(\bm{x},\bm{\xi})$ is continuous in all $\bm{x}\in\mathbb{R}_0^{|\mathscr{K}|}$.

Based on the continuity of $Y_a(\bm{x},\bm{\xi})$ in $\bm{x}$ and continuity of $G^1_a(\bm{x},\bm{\xi})$ and $G^2_a(\bm{x},\bm{\xi})$ (defined in \eqref{eqn:app:conitinuity_Q:4:1} and \eqref{eqn:app:conitinuity_Q:4:2}) in $x(t^{\kappa}-1)$ and $x^{\mu}(t^{\kappa})$, we plug $Y_a(\bm{x},\bm{\xi})$, $G^1_a(\bm{x},\bm{\xi})$, and $G^2_a(\bm{x},\bm{\xi})$ in \eqref{eqn:app:conitinuity_Q:3},
\begin{multline}\label{eqn:app:conitinuity_Q:13}
        Q^{1,\Lipschitza}(\kappa,\kappa',\bm{x},\bm{\xi})\\
= \Bigl(Y_a(\bm{x},\bm{\xi})- \indicator\Bigl\{ i^{\kappa}=i^{\kappa'},s^{\kappa}=s^{\kappa'}, \action^{\kappa} = \dummyAction,t^{\kappa}=t^{\kappa'}\Bigr\}\sum_{\begin{subarray}~\kappa''\in\mathscr{K}:\\i^{\kappa''}=i^{\kappa},\\s^{\kappa''}=s^{\kappa},\\t^{\kappa''}=t^{\kappa}\end{subarray}}x_{\kappa''}\scra^{\orule}_{\zeta^{\kappa'}}(\frac{\bm{x}}{\sum_{i\in[I]}N_i^0},t^{\kappa'})\xi(0)\Bigr)G^1_a(\bm{x},\bm{\xi})\\
+ \indicator\Bigl\{ i^{\kappa}=i^{\kappa'},s^{\kappa}=s^{\kappa'}, \action^{\kappa} = \dummyAction,t^{\kappa}=t^{\kappa'}\Bigr\}\sum_{\begin{subarray}~\kappa''\in\mathscr{K}:\\i^{\kappa''}=i^{\kappa},\\s^{\kappa''}=s^{\kappa},\\t^{\kappa''}=t^{\kappa}\end{subarray}}x_{\kappa''}\scra^{\orule}_{\zeta^{\kappa'}}(\frac{\bm{x}}{\sum_{i\in[I]}N_i^0},t^{\kappa'})\xi(0) G^2_a(\bm{x},\bm{\xi}),
\end{multline}
which is continuous in $\bm{x}$.

It remains to show the Lipschitz continuity of $Q^{1,\Lipschitza}$, separately, in $\bm{x}$ and $\bm{\xi}$.
For given $\bm{x}$, $Q^{1,\Lipschitza}<\infty$  and is linear in $\bm{\xi}$, which leads to its Lipschitz continuity in $\bm{\xi}$.

Recall that $G^1_a(\bm{x},\bm{\xi})$ and $G^2_a(\bm{x},\bm{\xi})$ are Lipschiz continuous in $\bm{x}$ (discussed after \eqref{eqn:app:conitinuity_Q:4:1} and \eqref{eqn:app:conitinuity_Q:4:2}).
Observing $Y_a$ defined in \eqref{eqn:app:conitinuity_Q:5}, it is piece-wise differentiable on $\bm{x}$. 
That is, when $x_{\kappa} \neq \lceil x_{\kappa} \rceil$, $Y_a(\bm{x},\bm{\xi})$ is differentiable in $\bm{x}$.
Also, since it is dependent on $\bm{x}$ through only $x_{\kappa}$, for $\kappa'' \neq \kappa$, $dY_a(\bm{x},\bm{\xi})/d x_{\kappa''} = 0$.

For $x_{\kappa}\in (n_1,n_1+1)$ with $n_1\in\mathbb{N}_0$, 
\begin{equation}
    \Bigl\lvert \frac{d Y_a(\bm{x},\bm{\xi})}{d x_{\kappa}}\Bigr\rvert  \leq \xi(\kappa,\kappa',n_1+1)    
    \Bigl\lvert \frac{d y_\Lipschitza (\norm{ x_{\kappa}}_+)}{d x_{\kappa}} \Bigr\rvert
    \leq K \max_{u\in[0,1]}\Bigl\lvert \frac{d y_a (u)}{d u} \Bigr\rvert,
\end{equation}
where $K<\infty$ dependent on $\bm{\xi}$ but independent to $\bm{x}$.

Based on \eqref{eqn:dirac-delta:2}, there exists $K<\infty$ (independent to $u$) such that $\lvert d y_a(u)/d u\rvert < K$.
Hence, there exists $K<\infty$ such that $\Bigl\lvert \frac{d Y_a(\bm{x},\bm{\xi})}{d x_{\kappa}}\Bigr\rvert  < K$ for all the points $\bm{x}$ with $x_{\kappa}\in(n_1,n_1+1)$. 
Since $ Y_a(\bm{x},\bm{\xi})$  is piece-wise differentiable with $\Bigl\lvert\frac{d Y_a(\bm{x},\bm{\xi})}{d x_{\kappa}}\Bigr\rvert  < K <\infty$ and is  continuous in $\bm{x}\in\mathbb{R}_0^{|\mathscr{K}|}$, there exists $K<\infty$ such that, for any $\bm{x},\bm{x}'\in\mathbb{R}_0^{|\mathscr{K}|}$, 
\[
\bigl\lvert Y_a(\bm{x},\bm{\xi}) - Y_a(\bm{x}',\bm{\xi})\bigr\rvert\leq K\bigl\lVert \bm{x}-\bm{x}'\bigr\rVert. 
\]
It is Lipchitz continuous in $\bm{x}\in\mathbb{R}_0^{|\mathscr{K}|}$.

Recall \eqref{eqn:app:conitinuity_Q:13}, since $Y_a$, $G^1_a$, and $G^2_a$ are Lipchitz continuous in $\bm{x}\in\mathbb{R}_0^{|\mathscr{K}|}$, so does $Q^{1,\Lipschitza}$.

\section{Proof of Theorem~\ref{theorem:convergence_Z_exp}}
\label{app:theorem:convergence_Z_exp}
Recall that the following proof is based on \cite[Theorem 4.1 in Chapter 7 and Theorem 3.3 in Chapter 3]{freidlin2012random}.
It follows similar lines as the proof for \cite[Theorem 3]{fu2024patrolling} but considers a more general case and
achieves Theorem~\ref{theorem:convergence_Z_exp} that is applicable to a broader range of problems than \cite[Theorem 3]{fu2024patrolling}.

For $U\in\mathbb{R}_0$, $\bm{x},\bm{\omega}\in\mathbb{R}^{|\mathscr{K}|}$, $\phi\in\lPhih$, 
we define
\begin{equation}\label{eqn:large_deviation:define_H}
\mathcal{G}^{\Lipschitza,U}(\bm{x},\bm{\omega}) \coloneqq \lim_{\bar{T}\rightarrow \infty}\frac{1}{\bar{T}}\ln \mathbb{E}\exp\Bigl\{\int_0^{\bar{T}}\innerproduct{\bm{\omega}}{b^{\Lipschitza,U}(\bm{x},\bm{\xi}^{1}_{\tau})}d\tau\Bigr\},
\end{equation}
where $\innerproduct{\cdot}{\cdot}$ is the dot production. 
For given $U<\infty$ and $\Lipschitza\in(0,1)$, $\mathcal{G}^{\Lipschitza,U}$ is bounded and Lipschitz continuous in both arguments.
Based on \cite[Lemma 4.1 in Chapter 7]{freidlin2012random}, such $\mathcal{G}^{\Lipschitza,U}$ is convex in the second argument $\bm{\omega}$.

\begin{lemma}\label{lemma:large_deviation:integral_H}
For $U\in\mathbb{R}_0\cup\{\infty\}$, $\phi\in\lPhih$,  
any compact sets $\mathscr{X}^c,\mathscr{W}^c\subset\mathbb{R}^{|\mathscr{K}|}$, and any $\bm{x}\in\mathscr{X}^c$ and $\bm{\omega}\in\mathscr{W}^c$, $\mathcal{G}^{\Lipschitza,U}(\bm{x},\bm{\omega})$ is Riemann integrable.
\end{lemma}
\proof{Proof of Lemma~\ref{lemma:large_deviation:integral_H}.}
For  any $\bm{\mu}\in\mathbb{R}^{|\mathscr{K}|\sum_{\kappa\in\mathscr{K}}N^0_{i^{\kappa}}+1}$, define
\begin{equation}
\mathcal{G}_{\xi}(\bm{\mu})\coloneqq \lim_{\bar{T}\rightarrow\infty} \frac{1}{\bar{T}}\ln \mathbb{E} \exp\Bigl\{\int_0^{\bar{T}}\innerproduct{\bm{\mu}}{\bm{\xi}^{1}_{\tau}}d\tau\Bigr\}.
\end{equation}

Based on the definition in Appendix~\ref{app:theorem:convergence-Z}, 
$$\mathbb{E} \exp\Bigl\{\mu(0)\Bigl\lfloor\int_0^{\bar{T}}\xi^{1}_{\tau}(0)d\tau\Bigr\rfloor\Bigr\} = \exp\Bigl\{\mu(0)\lfloor \frac{\bar{T}+1}{2}\rfloor\Bigr\},$$
and for $\kappa,\kappa'\in\mathscr{K}$,
since $\Bigl\lfloor\int_0^{\bar{T}}\xi^{1}_{\tau}(\kappa,\kappa',n)d\tau\Bigr\rfloor = \Bigl\lfloor\int_{\begin{subarray}~\tau\in[0,\bar{T})\\\tau\notin \scrR\end{subarray}}\xi^{1}_{\tau}(\kappa,\kappa',n)d\tau\Bigr\rfloor $ is Poisson distributed,
\begin{multline}\label{eqn:lemma:integral_H:1}
\mathbb{E}\exp\Bigl\{\int_0^{\bar{T}}\mu_{\kappa,\kappa'}(n)\xi^{1}_{\tau}(\kappa,\kappa',n)d\tau\Bigr\}
\leq \mathbb{E}\exp\Bigl\{\mu_{\kappa,\kappa'}(n)\bigl\lfloor\int_0^{\bar{T}}\xi^{1}_{\tau}(\kappa,\kappa',n)d\tau\bigr\rfloor+|\mu_{\kappa,\kappa'}(n)| \Bigr\}\\
=\exp\Bigl\{\scrp_{t^{\kappa}}(\zeta^{\kappa},\zeta^{\kappa'})\bigl\lvert[0,\bar{T})\backslash \scrR\bigr\rvert\bigl(e^{\mu_{\kappa,\kappa'}(n)}-1\bigr)+|\mu_{\kappa,\kappa'}(n)|\Bigr\},
\end{multline}
where recall the parameter of the Poisson distribution $\scrp_{t^{\kappa}}(\zeta^{\kappa},\zeta^{\kappa'})= \frac{1}{|[0,\bar{T})\backslash \scrR|}\mathbb{E}\Bigl[\bigl\lfloor\int_{\begin{subarray}~\tau\in[0,\bar{T})\\\tau\notin \scrR\end{subarray}}\xi^1_{\tau}(\kappa,\kappa',n)d\tau\bigr\rfloor\Bigr]$.

Hence,
\begin{equation}\label{eqn:lemma:integral_H:2}
    \mathcal{G}_{\xi}(\bm{\mu}) \leq \frac{\mu(0)}{2} + \frac{1}{2}\sum_{\begin{subarray}~\kappa,\kappa'\in\mathscr{K},\\n\in[N_{i^{\kappa}}^0]\end{subarray}} \scrp_{t^{\kappa}}(\zeta^{\kappa},\zeta^{\kappa'})\bigl(e^{\mu_{\kappa,\kappa'}(n)}-1\bigr).
\end{equation}

Similarly, replacing $|\mu_{\kappa,\kappa'}(n)|$ with $-|\mu_{\kappa,\kappa'}(n)|$ in \eqref{eqn:lemma:integral_H:1}, we can obtain
\begin{equation}\label{eqn:lemma:integral_H:3}
    \mathcal{G}_{\xi}(\bm{\mu}) \geq \frac{\mu(0)}{2} + \frac{1}{2}\sum_{\begin{subarray}~\kappa,\kappa'\in\mathscr{K},\\n\in[N_{i^{\kappa}}^0]\end{subarray}} \scrp_{t^{\kappa}}(\zeta^{\kappa},\zeta^{\kappa'})\bigl(e^{\mu_{\kappa,\kappa'}(n)}-1\bigr).
\end{equation}
Together with \eqref{eqn:lemma:integral_H:2},  
\begin{equation}\label{eqn:lemma:integral_H:4}
 \mathcal{G}_{\xi}(\bm{\mu}) = \frac{\mu(0)}{2} +\frac{1}{2}\sum_{\begin{subarray}~\kappa,\kappa'\in\mathscr{K},\\n\in[N_{i^{\kappa}}^0]\end{subarray}} \scrp_{t^{\kappa}}(\zeta^{\kappa},\zeta^{\kappa'})\bigl(e^{\mu_{\kappa,\kappa'}(n)}-1\bigr).
\end{equation}

For any $\bm{v}\in\mathbb{R}^{|\mathscr{K}|\sum_{\kappa\in\mathscr{K}}N^0_{i^{\kappa}}+1}$, let $\min\{\bm{v},U\} \coloneqq (\min\{v_n,U\}:n\in[N])$.
Recall $b^{\Lipschitza,U}(\bm{x},\bm{\xi}) = \tilde{\mathcal{Q}}^{1,\Lipschitza}(\bm{x})\min\{\bm{\xi},U\}$.
That is, 
\begin{multline}\label{eqn:lemma:integral_H:5}
    \mathcal{G}^{\Lipschitza,U}(\bm{x},\bm{\omega}) = \lim_{\bar{T}\rightarrow\infty}\frac{1}{\bar{T}}\ln\mathbb{E}\exp \Bigl\{\innerproduct{\bm{\omega}^T\tilde{\mathcal{Q}}^{1,\Lipschitza}(\bm{x})}{\int_0^{\bar{T}}\min\{\bm{\xi}^{1}_{\tau},U\}d\tau}\Bigr\}
    \\
    \leq \lim_{\bar{T}\rightarrow\infty}\frac{1}{\bar{T}}\ln\mathbb{E}\exp \Bigl\{\innerproduct{\bigl(\bm{\omega}^T\tilde{\mathcal{Q}}^{1,\Lipschitza}(\bm{x})\bigr)^+}{\int_0^{\bar{T}}\bm{\xi}^{1}_{\tau}d\tau}\Bigr\}
     = \mathcal{G}_{\xi}\bigl((\bm{\omega}^T\tilde{\mathcal{Q}}^{1,\Lipschitza}(\bm{x}))^+\bigr),
\end{multline}
where $(\bm{v})^+\coloneqq \Bigl(\max\{v_n,0\}:n\in\bigl[|\mathscr{K}|\sum_{\kappa\in\mathscr{K}}N^0_{i^{\kappa}}+1\bigr]\Bigr)$ for any vector $\bm{v}\in \mathbb{R}^{|\mathscr{K}|\sum_{\kappa\in\mathscr{K}}N^0_{i^{\kappa}}+1}$.
For any compact sets $\mathscr{W}^c,\mathscr{X}^c\subset \mathbb{R}^{|\mathscr{K}|}$, $\bm{\omega}\in\mathscr{W}^c$, and $\bm{x}\in\mathscr{X}^c$, $\mathcal{G}^{\Lipschitza,U}(\bm{x},\bm{\omega})$ is bounded and is jointly continuous in both arguments.
Hence, it is Riemann integrable. The lemma is proved.

\endproof

From Lemma~\ref{lemma:large_deviation:integral_H}, there exists $\mathcal{G}^{\Lipschitza,U}$ defined in \eqref{eqn:large_deviation:define_H} that satisfies
\begin{equation}\label{eqn:theorem:convergence_Z_exp:condition}
    \int_0^{\bar{T}}\mathcal{G}^{\Lipschitza,U}(\bm{x}_\tau,\bm{\omega}_\tau)d \tau = \lim_{\epsilon\downarrow 0}\ln \mathbb{E}\exp\Bigl\{\frac{1}{\epsilon}\int_0^{\bar{T}}\innerproduct{\bm{\omega}_\tau}{b^{\Lipschitza,U}(\bm{x}_\tau,\bm{\xi}^{1}_{\tau/\epsilon})}d\tau\Bigr\}.
\end{equation}

For $\bm{x},\bm{\beta}\in\mathbb{R}^{|\mathscr{K}|}$, consider the Legendre transform of $\mathcal{G}^{\Lipschitza,U}(\bm{x},\bm{\omega})$,
\begin{equation}\label{eqn:large_deviation:legendre_transfor}
\mathcal{L}^{\Lipschitza,U}(\bm{x},\bm{\beta}) \coloneq \sup_{\bm{\omega}\in\mathbb{R}^{|\mathscr{K}|}}\Bigl[\innerproduct{\bm{\omega}}{\bm{\beta}} - \mathcal{G}^{\Lipschitza,U}(\bm{x},\bm{\omega})\Bigr],
\end{equation}
where recall $\mathcal{G}^{\Lipschitza,U}(\bm{x},\bm{\omega})$ is convex in $\bm{\omega}$. 
Since $\innerproduct{\bm{0}}{\bm{\beta}} - \mathcal{G}^{\Lipschitza,U}(\bm{x},\bm{0}) = 0$, $\mathcal{L}^{\Lipschitza,U}(\bm{x},\bm{\beta})$ is always non-negative.

\begin{lemma}\label{lemma:large_deviation:derivative_H}
Given $\bm{x},\bm{\omega}\in\mathbb{R}^{|\mathscr{K}|}$, 
\begin{equation}\label{eqn:lemma:derivative_H}
\lim_{U\rightarrow \infty}\frac{\partial \mathcal{G}^{\Lipschitza,U}}{\partial \bm{\omega}} = \frac{\partial \mathcal{G}_{\xi}(\bm{\omega}^T\tilde{\mathcal{Q}}^{1,\Lipschitza}(\bm{x}))}{\partial \bm{\omega}}.
\end{equation}
\end{lemma}
\proof{Proof of Lemma~\ref{lemma:large_deviation:derivative_H}.}
For $U\in\mathbb{R}_0$, $\phi\in\lPhih$,  
and any $\bm{\mu}\in\mathbb{R}^{|\mathscr{K}|\sum_{\kappa\in\mathscr{K}}N^0_{i^{\kappa}}+1}$, define
\begin{equation}
\mathcal{G}_{\xi}^U(\bm{\mu})\coloneqq \lim_{\bar{T}\rightarrow\infty} \frac{1}{\bar{T}}\ln \mathbb{E} \exp\Bigl\{\int_0^{\bar{T}}\innerproduct{\bm{\mu}}{\min\{\bm{\xi}^{1}_{\tau},U\}}d\tau\Bigr\}.
\end{equation}
Similar to the analysis in \eqref{eqn:lemma:integral_H:1}-\eqref{eqn:lemma:integral_H:3}, we have
\begin{multline}\label{eqn:large_deviation:derivative_H:1}
\mathcal{G}_{\xi}^U(\bm{\mu})=\lim_{\bar{T}\rightarrow\infty} \frac{1}{\bar{T}}\ln \mathbb{E} \exp\Bigl\{\innerproduct{\bm{\mu}}{\int_0^{\bar{T}}\min\{\bm{\xi}^{1}_{\tau},U\}d\tau}\Bigr\}\\ 
\leq \lim_{\bar{T}\rightarrow\infty} \frac{1}{\bar{T}}\ln \mathbb{E} \exp\Bigl\{\innerproduct{\bm{\mu}}{\bigl\lfloor\int_0^{\bar{T}}\min\{\bm{\xi}^{1}_{\tau},U\}d\tau\bigr\rfloor + \bm{1}(\bm{\mu})}\Bigr\} \\
= \lim_{\bar{T}\rightarrow\infty} \frac{1}{\bar{T}}\ln \mathbb{E} \exp\Bigl\{\innerproduct{\bm{\mu}}{\bigl\lfloor\int_0^{\bar{T}}\min\{\bm{\xi}^{1}_{\tau},U\}d\tau\bigr\rfloor}\Bigr\},
\end{multline}
and
\begin{multline}\label{eqn:large_deviation:derivative_H:2}
\mathcal{G}_{\xi}^U(\bm{\mu})\geq \lim_{\bar{T}\rightarrow\infty} \frac{1}{\bar{T}}\ln \mathbb{E} \exp\Bigl\{\innerproduct{\bm{\mu}}{\bigl\lfloor\int_0^{\bar{T}}\min\{\bm{\xi}^{1}_{\tau},U\}d\tau\bigr\rfloor - \bm{1}(\bm{\mu})}\Bigr\} \\
= \lim_{\bar{T}\rightarrow\infty} \frac{1}{\bar{T}}\ln \mathbb{E} \exp\Bigl\{\innerproduct{\bm{\mu}}{\bigl\lfloor\int_0^{\bar{T}}\min\{\bm{\xi}^{1}_{\tau},U\}d\tau\bigr\rfloor}\Bigr\}.
\end{multline}
That is,
\begin{equation}\label{eqn:large_deviation:derivative_H:4}
\mathcal{G}_{\xi}^U(\bm{\mu})
= \lim_{\bar{T}\rightarrow\infty} \frac{1}{\bar{T}}\ln \mathbb{E} \exp\Bigl\{\innerproduct{\bm{\mu}}{\bigl\lfloor\int_0^{\bar{T}}\min\{\bm{\xi}^{1}_{\tau},U\}d\tau\bigr\rfloor}\Bigr\}.
\end{equation}
Recall that $\bigl\lfloor\int_{\tau\in[0,\bar{T})\backslash \scrR}\xi_{\tau}(\kappa,\kappa',n)d\tau\bigr\rfloor $ is Poisson distributed with parameter $\scrp_{t^{\kappa}}(\zeta^{\kappa},\zeta^{\kappa'})$, and given a trajectory $\bigl\{\xi_{\tau}(\kappa,\kappa',n), 0\leq \tau\leq \bar{T}\bigr\}$, there exist a sequence of points $0=\tau_0<\tau_1<\ldots<\tau_M\leq \bar{T}$ such that, for $m\in[M]$,
\[ \int_{\tau_{m-1}}^{\tau_m}\xi_{\tau}(\kappa,\kappa',n)d\tau = 1.\]
With the imposed upper bound $U$, we have
\[
\Bigl\lfloor\int_0^{\bar{T}}\min\{\xi_{\tau}(\kappa,\kappa',n),U\}d\tau\Bigr\rfloor = \sum_{m\in[M]}\theta_m\int_{\tau_{m-1}}^{\tau_m}\xi_{\tau}(\kappa,\kappa',n)d\tau = \sum_{m\in[M]}\theta_m,
\]
where $M$ is Poisson distributed, $\theta_m \coloneqq \max\Bigl\{1,\frac{U}{\xi_{\tau_{m-1}}(\kappa,\kappa',n)}\Bigr\}$, and all such $\theta_m$ ($m\in[M]$ are independently and identically distributed.
It follows that $\Bigl\lfloor\int_0^{\bar{T}}\min\{\xi_{\tau}(\kappa,\kappa',n),U\}d\tau\Bigr\rfloor $ follows compound Poisson distribution, and hence
\begin{multline}
    \mathbb{E}\exp\Bigl\{\innerproduct{\bm{\mu}}{\bigl\lfloor\int_0^{\bar{T}}\min\{\bm{\xi}^{1}_{\tau},U\}d\tau\bigr\rfloor}\Bigr\}\\
=  \exp\Bigl\{\frac{1}{2}\sum_{\kappa,\kappa'\in\mathscr{K},n\in[N^0_{i^{\kappa}}]}\bigl(\Theta_{\kappa,\kappa'}(n)-1\bigr)\scrp_{t^{\kappa}}(\zeta^{\kappa},\zeta^{\kappa'})(\bar{T}+o(\bar{T})) + \mu(0)\frac{\lfloor \bar{T}+1\rfloor}{2}\Bigr\},
\end{multline}
where 
\[\Theta_{\kappa,\kappa'}(n)\coloneqq \int_0^1 \exp\bigl\{\mu_{\kappa,\kappa'}(n)\theta\bigr\}\bbP\{\theta\}d\theta \leq \exp\{\mu_{\kappa,\kappa'}(n)\}.\]
In particular, $\lim_{U\rightarrow \infty}\bbP\{\theta\} = \delta_1(\theta)$ where $\delta_1(\theta)$ is a Dirac delta function satisfying 
\[ \delta_1(\theta) = \begin{cases}
    0,& \text{if }\theta \neq 1,\\
    \infty, &\text{otherwise},
\end{cases}\]
and $\int_{-\infty}^{\infty}\delta_1(\theta)d\theta = 1$.
Let $\bm{\Theta}\coloneqq (\Theta_{\kappa,\kappa'}(n):\kappa,\kappa'\in\mathscr{K},n\in[N^0_{i^{\kappa}}])$.
Obviously, $\lim_{U\rightarrow\infty} \bm{\Theta} = (e^{\mu_{\kappa,\kappa'}(n)}:\kappa,\kappa'\in\mathscr{K},n\in[N^0_{i^{\kappa}}])$.

We then obtain
\begin{equation}
    \mathcal{G}^U_{\xi}(\bm{\mu}) = \frac{\mu(0)}{2}+\frac{1}{2}\sum_{\kappa,\kappa'\in\mathscr{K},n\in[N^0_{i^{\kappa}}]}\bigl(\Theta_{\kappa,\kappa'}(n)-1\bigr)\scrp_{t^{\kappa}}(\zeta^{\kappa},\zeta^{\kappa'}),
\end{equation}
and, for $\kappa\in\mathscr{K}$, 
\begin{equation}
  \frac{\partial}{\omega_{\kappa}} \mathcal{G}^{\Lipschitza,U}(\bm{x},\bm{\omega}) = \frac{\partial}{\omega_{\kappa}} \mathcal{G}^U_{\xi}\bigl(\bm{\omega}^T\tilde{\mathcal{Q}}^{1,\Lipschitza}(\bm{x})\bigr) = \innerproduct{\bm{\rho}}{\bm{q}_{\kappa}},
\end{equation}
where $\bm{q}_{\kappa}$ is the $\kappa$th row vector of $\tilde{\mathcal{Q}}^{1,\Lipschitza}(\bm{x})$, and 
\[\bm{\rho} \coloneqq \biggl(\frac{1}{2};\frac{1}{2}\scrp_{t^{\kappa_1}}(\zeta^{\kappa_1},\zeta^{\kappa_2})\int_0^1\theta e^{\mu_{\kappa_1,\kappa_2}(n)\theta}\bbP\{\theta\}d\theta:~\kappa_1,\kappa_2\in\mathscr{K},n\in[N^0_{i^{\kappa_1}}]\biggr),\]
with $\mu_{\kappa_1,\kappa_2}(n) = \bigl(\bm{\omega}^T\tilde{\mathcal{Q}}^{1,\Lipschitza}(\bm{x})\bigr)_{\kappa_1,\kappa_2}(n)$.
Together with $\mathcal{G}_{\xi}$ given by \eqref{eqn:lemma:integral_H:4}, we obtain \eqref{eqn:lemma:derivative_H}.
It proves the lemma.

\endproof

\begin{lemma}\label{lemma:large_deviation:unique_zero}
For $U\in\mathbb{R}_0\cup\{\infty\}$, $\phi\in\lPhih$,  
$\bm{x},\bm{\beta}\in\mathbb{R}^{|\mathscr{K}|}$, and any $\delta > 0$, there exists $U_0<\infty$ such that, for all $U > U_0$, if $\lVert \bm{\beta} - \mathbb{E}b^{\Lipschitza,U}(\bm{x},\bm{\xi}^{1}_\tau)\rVert \geq \delta$, then $\mathcal{L}^{\Lipschitza,U}(\bm{x},\bm{\beta}) > 0$.
\end{lemma}
\proof{Proof of Lemma~\ref{lemma:large_deviation:unique_zero}.}
Based on \cite[Chapter 7, Section 4]{freidlin2012random}, for any $U\in\mathbb{R}_0$, if $\bm{\beta}=\bar{b}^U(\bm{x})$, then $\mathcal{L}^U(\bm{x},\bm{\beta}) = 0$.

Let $\bm{\omega}^{\Lipschitza,U}(\bm{x},\bm{\beta})$ represent the extreme point satisfying $\frac{\partial \mathcal{G}^{\Lipschitza,U}}{\partial \bm{\omega}}\Bigr|_{\bm{\omega}=\bm{\omega}^{\Lipschitza,U}(\bm{x},\bm{\beta})}=\bm{\beta}$. Such an extreme point may not be unique.
Based on Lemma~\ref{lemma:large_deviation:derivative_H}, when $U\rightarrow \infty$, it becomes
\begin{equation}\label{eqn:lemma:unique_zero:1}
    \bm{\beta} = \lim_{U\rightarrow \infty}\frac{\partial \mathcal{G}^{\Lipschitza,U}(\bm{x},\bm{\omega})}{\partial \bm{\omega}}\Bigr|_{\bm{\omega}=\bm{\omega}^{\Lipschitza,U}(\bm{x},\bm{\beta})} = \frac{\partial \mathcal{G}_{\xi}\bigl(\bm{\omega}\tilde{\mathcal{Q}}^{1,\Lipschitza}(\bm{x})\bigr)}{\partial \bm{\omega}}\Bigr|_{\bm{\omega}=\bm{\omega}^{\Lipschitza,\infty}(\bm{x},\bm{\beta})},
\end{equation}
where $\bm{\omega}^{\Lipschitza,\infty}(\bm{x},\bm{\beta}) \coloneqq \lim_{U\rightarrow \infty} \bm{\omega}^{\Lipschitza,U}(\bm{x},\bm{\beta})$, and $\mathcal{G}_{\xi}$ is given in \eqref{eqn:lemma:integral_H:4}.
Substituting \eqref{eqn:lemma:integral_H:4} in \eqref{eqn:lemma:unique_zero:1}, it becomes
\begin{equation}\label{eqn:lemma:unique_zero:2}
\bm{\beta} = \tilde{\mathcal{Q}}^{1,\Lipschitza}(\bm{x})\Bigl(\frac{1}{2};\frac{1}{2}\scrp_{t^{\kappa}}(\zeta^{\kappa},\zeta^{\kappa'})\exp\bigl\{<\bm{\omega}^{\Lipschitza,\infty}(\bm{x},\bm{\beta}),\tilde{\bm{q}}_{\kappa,\kappa'}(n,\bm{x})>\bigr\}: \kappa,\kappa'\in\mathscr{K},n\in[N^0_{i^{\kappa}}]\Bigr),
\end{equation}
where $\tilde{\bm{q}}_{\kappa,\kappa'}(n,\bm{x})$ are the column vectors of matrix $\tilde{\mathcal{Q}}^{1,\Lipschitza}(\bm{x})$.
Moreover, if 
\begin{equation}\label{eqn:lemma:unique_zero:3}
    \lim_{U\rightarrow\infty} \mathcal{L}^{\Lipschitza,U}(\bm{x},\bm{\beta}) = <\bm{\omega}^{\Lipschitza,\infty}(\bm{x},\bm{\beta}),\bm{\beta}>-\mathcal{G}_{\xi}\bigl((\bm{\omega}^{\Lipschitza,\infty}(\bm{x},\bm{\beta}))^T\tilde{\mathcal{Q}}^{1,\Lipschitza}(\bm{x})\bigr)=0,
\end{equation}
then, substituting \eqref{eqn:lemma:unique_zero:2} and \eqref{eqn:lemma:integral_H:4} in \eqref{eqn:lemma:unique_zero:3}, we obtain
\begin{multline}\label{eqn:lemma:unique_zero:4}
<\bm{\omega}^{\Lipschitza,\infty}(\bm{x},\bm{\beta}),\bm{\beta}> - \mathcal{G}_{\xi}\bigl((\bm{\omega}^{\Lipschitza,\infty}(\bm{x},\bm{\beta}))^T\tilde{\mathcal{Q}}^{1,\Lipschitza}(\bm{x})\bigr)=0\\
=\bigl(\bm{\omega}^{\Lipschitza,\infty}(\bm{x},\bm{\beta})\bigr)^T\tilde{\mathcal{Q}}^{1,\Lipschitza}(\bm{x})\Bigl(\frac{1}{2};\frac{1}{2}\scrp_{t^{\kappa}}(\zeta^{\kappa},\zeta^{\kappa'})\exp\bigl\{<\bm{\omega}^{\Lipschitza,\infty}(\bm{x},\bm{\beta}),\tilde{\bm{q}}_{\kappa,\kappa'}(n,\bm{x})>\bigr\}: \kappa,\kappa'\in\mathscr{K},n\in[N^0_{i^{\kappa}}]\Bigr)\\
-\frac{<\bm{\omega}^{\Lipschitza,\infty}(\bm{x},\bm{\beta}), \tilde{q}(0)>}{2}-\frac{1}{2} \sum_{\begin{subarray}~\kappa,\kappa'\in\mathscr{K},\\n\in[N^0_{i^{\kappa}}]\end{subarray}}\scrp_{t^{\kappa}}(\zeta^{\kappa},\zeta^{\kappa'})\Bigl(\exp\bigl\{<\bm{\omega}^{\Lipschitza,\infty}(\bm{x},\bm{\beta}),\tilde{\bm{q}}_{\kappa,\kappa'}(n,\bm{x})>\bigr\}-1\Bigr)\\
=\frac{1}{2}\sum_{\begin{subarray}~\kappa,\kappa'\in\mathscr{K},\\n\in[N^0_{i^{\kappa}}]\end{subarray}}\scrp_{t^{\kappa}}(\zeta^{\kappa},\zeta^{\kappa'})\Bigl(\exp\bigl\{<\bm{\omega}^{\Lipschitza,\infty}(\bm{x},\bm{\beta}),\tilde{\bm{q}}_{\kappa,\kappa'}(n,\bm{x})>\bigr\}\Bigl(<\bm{\omega}^{\Lipschitza,\infty}(\bm{x},\bm{\beta}),\tilde{\bm{q}}_{\kappa,\kappa'}(n,\bm{x})>-1\Bigr)+1\Bigr),
\end{multline}
where $\tilde{q}(0)$ is the first column vector of $\tilde{\mathcal{Q}}^{1,\Lipschitza}$.
For \eqref{eqn:lemma:unique_zero:4}, since $\exp\bigl\{<\bm{\omega},\bm{q}>\bigr\}\bigl(<\bm{\omega},\bm{q}>-1\bigr)+1 \geq 0$ for any $\bm{\omega},\bm{q} \in \mathbb{R}^{|\mathscr{K}|}$ and the equality holds if and only if $<\bm{\omega},\bm{q}>=0$, 
\begin{equation}
\exp\bigl\{<\bm{\omega}^{\infty}(\bm{x},\bm{\beta}),\tilde{\bm{q}}_{\kappa,\kappa'}(n,\bm{x})>\bigr\}\Bigl(<\bm{\omega}^{\infty}(\bm{x},\bm{\beta}),\tilde{\bm{q}}_{\kappa,\kappa'}(n,\bm{x})>-1\Bigr)+1 = 0,    
\end{equation}
or, equivalently,
\begin{equation}\label{eqn:lemma:unique_zero:5}
  <\bm{\omega}^{\infty}(\bm{x},\bm{\beta}),\tilde{\bm{q}}_{\kappa,\kappa'}(n,\bm{x})> = 0,
\end{equation}
for all $\kappa,\kappa'\in\mathscr{K}$ and $n\in[N^0_{i^{\kappa}}]$.

Based on \eqref{eqn:lemma:unique_zero:5} and \eqref{eqn:lemma:unique_zero:2}, given $\bm{x}\in\mathbb{R}^{|\mathscr{K}|}$,
if $\lim_{U\rightarrow \infty}\mathcal{L}^{\Lipschitza,U}(\bm{x},\bm{\beta}) = 0$, then $$\bm{\beta}= \tilde{\mathcal{Q}}^{1,\Lipschitza}(\bm{x})(\frac{1}{2};\frac{1}{2}\scrp_{t^{\kappa}}(\zeta^{\kappa},i^{\kappa'}):\kappa,\kappa'\in\mathscr{K},n\in[N_{i^{\kappa}}^0])$$ which is the unique solution satisfying $\lim_{U\rightarrow \infty}\mathcal{L}^{\Lipschitza,U}(\bm{x},\bm{\beta}) = 0$.

Recall that, for all $U\in\mathbb{R}_0$, if $\bm{\beta}=\mathbb{E}b^{\Lipschitza,U}(\bm{x},\bm{\xi}^{1}_{\tau})$, then $\mathcal{L}^{\Lipschitza,U}(\bm{x},\bm{\beta}) = 0$.
Since $\mathcal{L}^{\Lipschitza,U}(\bm{x},\bm{\beta})$ is continuous in $U$, together with the unique $\bm{\beta}$ for $\lim_{U\rightarrow \infty}\mathcal{L}^{\Lipschitza,U}(\bm{x},\bm{\beta}) = 0$, for any given $\bm{x}$ and $\delta > 0$, if $\lVert \bm{\beta} - \mathbb{E}b^{\Lipschitza,U}(\bm{x},\bm{\xi}^{1}_{\tau})\rVert \geq \delta$, then there exists $U_0 <\infty$ such that, for all $U > U_0$, $\mathcal{L}^{\Lipschitza,U}(\bm{x},\bm{\beta}) > 0$. It proves the lemma.

\endproof

\proof{Proof of Theorem~\ref{theorem:convergence_Z_exp}.}
For $\bm{\chi}_\tau: \mathbb{R}_0\mapsto \mathbb{R}^{|\mathscr{K}|}$ and a trajectory $\varpi\coloneqq\{\chi_\tau, 0\leq\tau\leq\bar{T}\}$, define $\Lambda^{\Lipschitza  ,U}_{0,\bar{T}}(\varpi)\coloneqq \int_0^{\bar{T}}\mathcal{L}^{\Lipschitza,U}(\bm{\chi}_{\tau},\dot{\bm{\chi}}_{\tau})d\tau$.
Define $\Pi^{\Lipschitza,U}_{0,\bar{T}}$ as the compact set of all such trajectories with given $\bm{\chi}_0=\sum_{i\in[I]}N_i^0\bm{\scrz}_0\in\mathbb{R}^{|\mathscr{K}|}$, where recall that $\sum_{i\in[I]}N_i^0\bm{\scrz}_0\in\mathbb{R}^{|\mathscr{K}|}$ is also the given initial state of $\{\bm{X}^{\sigma,\Lipschitza,U}_\tau,0\leq \tau\leq\bar{T}\}$ and $\{\bar{\bm{x}}^{\Lipschitza,U}_\tau,0\leq \tau\leq\bar{T}\}$ that are defined in Appendix~\ref{app:theorem:convergence-Z} and are the solutions of $\dot{\bm{X}}^{\sigma,\Lipschitza,U}_{\tau}=b^{\Lipschitza,U}(\bm{X}^{\sigma}_{\tau},\bm{\xi}^{1}_{\tau/\sigma})$ and $\dot{\bar{\bm{x}}}^{\Lipschitza,U}_{\tau} = \mathbb{E}b^{\Lipschitza,U}(\bar{\bm{x}}^{\Lipschitza,U}_{\tau},\bm{\xi}^{1}_{\tau/\sigma})$, respectively.
For $U\in\mathbb{R}_0$, $\Lipschitza\in(0,1)$, and $\delta > 0$, define a closed set $\mathscr{C}^{\Lipschitza,U}(\delta)\coloneqq \{\varpi\in\Pi^{\Lipschitza,U}_{0,\bar{T}}~|~\sup_{0\leq\tau\leq\bar{T}}\lVert \bm{\chi}_{\tau}-\bar{\bm{x}}^{\Lipschitza,U}_\tau\rVert \geq \delta\}$.

Based on \cite[Theorem 4.1 in Chapter 7 and Theorem 3.3 in Chapter 3]{freidlin2012random}, for sufficiently large $U\in\mathbb{R}_0$, any $\Lipschitza\in(0,1)$, and any $\delta>0$, as there exists $\mathcal{G}^{\Lipschitza,U}$ satisfying \eqref{eqn:theorem:convergence_Z_exp:condition},
\begin{equation}\label{eqn:theorem:convergence_Z_exp:5}
    \limsup_{\sigma\downarrow 0} \sigma \ln \mathbb{P}\Bigl\{\sup_{0\leq \tau\leq \bar{T}}\lVert \bm{X}^{\sigma,\Lipschitza,U}_\tau - \bar{\bm{x}}^{\Lipschitza,U}_{\tau}\rVert > \delta\Bigr\}\leq -\inf_{\varpi\in\mathscr{C}^{\Lipschitza,U}(\delta)}\Lambda^{\Lipschitza,U}_{0,\bar{T}}(\varpi).
\end{equation}
Based on Lemma~\ref{lemma:large_deviation:unique_zero} and the continuity of $\bar{\bm{x}}^{\Lipschitza,U}_\tau$ and $\Lambda^{\Lipschitza,U}_{0,\bar{T}}(\varpi)$ in $U\in\mathbb{R}_+\cup\{\infty\}$, for any $\delta >0$, there exist $\sigma>0$ and $U_0<\infty$ such that, for all $U>U_0$, $\inf_{\varpi\in\mathscr{C}^{\Lipschitza,U}(\delta)}\Lambda^{\Lipschitza,U}_{0,\bar{T}}(\varpi) \geq \sigma > 0$.
From \eqref{eqn:theorem:convergence_Z_exp:5}, for any $\delta > 0$, sufficiently large $U$, and given initial state $\bm{X}^{\sigma}_0=\bar{\bm{x}}_0=\sum_{i\in[I]}N^0_i \bm{\scrz}_0$, there exist $\sigma_0,C > 0$ such that, for all $\sigma < \sigma_0$,
\begin{equation}\label{eqn:theorem:convergence_Z_exp:6}
\mathbb{P}\Bigl\{\sup_{0\leq \tau\leq \bar{T}}\lVert \bm{X}^{\sigma,\Lipschitza,U}_\tau - \bar{\bm{x}}^{\Lipschitza,U}_{\tau}\rVert >  \delta\Bigr\} \leq e^{-C/\sigma}.
\end{equation}

Similar to the analysis of \eqref{eqn:app:convergence_Z:9}, 
\begin{equation}\label{eqn:theorem:convergence_Z_exp:7}
    \mathbb{P}\Bigl\{\sup_{0\leq \tau\leq \bar{T}}\lVert \bm{X}^{\sigma}_\tau - \bar{\bm{x}}_{\tau}\rVert > \delta\Bigr\} 
    \leq \mathbb{P}\Bigl\{\sup_{0\leq \tau\leq \bar{T}}\norm{ \bm{X}^{\sigma}_\tau - \bm{X}^{\sigma,\Lipschitza_0,U_0}}_\tau+\norm{\bm{X}^{\sigma,\Lipschitza_0,U_0}_\tau - \bar{\bm{x}}^{\Lipschitza_0,U_0}_\tau}
    +\norm{\bar{\bm{x}}^{\Lipschitza_a,U_0}_{\tau} - \bar{\bm{x}}_{\tau}}>\delta
    \Bigr\}.
\end{equation}
Based on Lemma~\ref{lemma:continuity_trajectory_a:stochastic}  and \eqref{eqn:theorem:convergence_Z_exp:7}, for any $\delta>0$,
there exist  $0<\epsilon < \delta/2$, 
$\Lipschitza_0>0$, and $U_0 <\infty$ such that
\begin{multline}\label{eqn:theorem:convergence_Z_exp:8}
    \mathbb{P}\Bigl\{\sup_{0\leq \tau\leq \bar{T}}\lVert \bm{X}^{\sigma}_\tau - \bar{\bm{x}}_{\tau}\rVert > \delta\Bigr\} 
    \leq \mathbb{P}\Bigl\{\sup_{0\leq \tau\leq \bar{T}}\norm{ \bm{X}^{\sigma}_\tau - \bm{X}^{\sigma,\Lipschitza_0,U_0}}_\tau+\norm{\bm{X}^{\sigma,\Lipschitza_0,U_0}_\tau - \bar{\bm{x}}^{\Lipschitza_0,U_0}_\tau}
    +\norm{\bar{\bm{x}}^{\Lipschitza_a,U_0}_{\tau} - \bar{\bm{x}}_{\tau}}>\delta
    \Bigr\}\\
    \leq \mathbb{P}\Bigl\{\sup_{0\leq \tau\leq \bar{T}}\norm{\bm{X}^{\sigma,\Lipschitza_0,U_0}_\tau - \bar{\bm{x}}^{\Lipschitza_0,U_0}_\tau}
   >\delta - 2\epsilon
    \Bigr\}.
\end{multline}
Based on \eqref{eqn:theorem:convergence_Z_exp:6} and \eqref{eqn:theorem:convergence_Z_exp:8}, for any $\delta >0$,  there exist $\sigma_0,C > 0$ such that, for all $\sigma < \sigma_0$,
\begin{equation}
\mathbb{P}\Bigl\{\sup_{0\leq \tau\leq \bar{T}}\lVert \bm{X}^{\sigma}_\tau - \bar{\bm{x}}_{\tau}\rVert >  \delta\Bigr\} \leq e^{-C/\sigma}.
\end{equation}

Along the same lines as the proof of Theorem~\ref{theorem:convergence-Z}, for the solution, $\bm{\mathcal{Z}}^h_\tau$, of \eqref{eqn:ODE_Z_h} and $\bm{\mathcal{Z}}^h_0 = \bm{\scrz}_0$ and $\sigma = \frac{1}{h}$, we have \eqref{eqn:app:convergence_Z:12}.
Hence, for any $\bar{T} <\infty$ and $\delta>0$, there exist $H <\infty$ and $C>0$ such that, for all $h>H$,
\begin{equation}
  \mathbb{P}\Bigl\{\sup_{0\leq \tau\leq \bar{T}}\lVert \bm{\mathcal{Z}}^{h}_\tau - \lim_{h\rightarrow \infty} \mathbb{E}\bm{\mathcal{Z}}^h_\tau\rVert >  \delta\Bigr\} \leq e^{-Ch}.
\end{equation}
Together with Lemma~\ref{lemma:sim}, 
we prove \eqref{eqn:theorem:convergence_Z_exp} for any given $\orule\in\PsiZ$.
It proves the theorem.

\endproof

\section{Proof of Proposition~\ref{prop:asym_opt:LP}}\label{app:asym_opt:LP}

\proof{Proof of Proposition~\ref{prop:asym_opt:LP}.}
For any $\orule\in\lPsi$ and $t\in[T]_0$, define an action matrix $\mathcal{A}^{\orule,h}(t) \coloneqq \bigl[A^{\varphi,h}_{i,s,i',s',\action'}(t)\bigr]_{(\sum_{i\in[I]}|\bS_i|)\times|\mathcal{J}|}$,  where
\begin{equation}
    A^{\orule,h}_{i,s,i',s',\action'}(t) \coloneqq \begin{cases}
        \alpha^{\orule,h}_{i',s',\action'}(\bm{Y}^{\varphi,h}(t),t),&\text{if }i=i',s=s',\\
        0,&\text{otherwise},
    \end{cases}
\end{equation}
where $\alpha^{\orule,h}_{i,s,\action}$ is defined in \eqref{eqn:define:alpha}.

We start by analyzing \eqref{eqn:prop:asym_opt:LP:1}.
By Theorem~\ref{theorem:convergence_Z_exp}, for any $\epsilon>\epsilon'>0$ and $\orule\in\PsiZ$, there exist $H<\infty$ and $C>0$ such that for all $h>H$,
\begin{multline}\label{eqn:prop:asym_opt:LP:3-1}
    \bbP\Bigl\{\norm{\bZ^{\orule,h}(t)-\bz^{\dborule(\bm{x})}(t)}>\epsilon\Bigr\}
    \bbP \Bigl\{\norm{\bZ^{\orule,h}(t)-\bz^{\dborule(\bm{x})}(t)}>\epsilon~\Bigl|~\norm{\bZ^{\orule,h}(t)-\bz^{\orule,h}(t)}\leq \epsilon'\Bigr\} + e^{-Ch}\\
    \leq \bbP\Bigl\{\norm{\bz^{\orule,h}(t)-\bz^{\dborule(\bm{x})}(t)}>\epsilon-\epsilon'~\Bigl|~\norm{\bZ^{\orule,h}(t)-\bz^{\orule,h}(t)}\leq \epsilon'\Bigr\} + e^{-Ch}.
\end{multline}
For $\orule\in\PsiZ$ and $\bm{x}\in[0,1]^{|\mathcal{J}|(T+1)}$ satisfying \eqref{eqn:constraint:linear programming:1}-\eqref{eqn:constraint:linear programming:3}, 
we have $\bm{Y}^{\orule,h}(0)=\bm{Y}^{\dborule(\bm{x}),h}(0)=\bm{y}^0$.
For the special case $t=0$,
\begin{equation}
    \lim_{h\to\infty}\bbP\Bigl\{\norm{\bz^{\orule,h}(0)-\bz^{\dborule(\bm{x})}(0)}>\epsilon \Bigr\}
    \leq\lim_{h\to\infty}\bbP\Bigl\{\norm{\bigl(\bm{Y}^{\orule,h}(0)\bigr)^T(\mathcal{A}^{\orule,h}(0)-\mathcal{A}^{\dborule(\bm{x}),h}(0))}>\epsilon\Bigr\}
    =0,
\end{equation}
where the last equality is based on \eqref{eqn:assumption:action}.
Together with \eqref{eqn:prop:asym_opt:LP:3-1}, for $t=0$ and any $\epsilon>0$, there exist $H<\infty$ and $C>0$ such that 
\begin{equation}\label{eqn:prop:asym_opt:LP:3-3}
    \bbP\Bigl\{\norm{\bZ^{\orule,h}(t)-\bz^{\dborule(\bm{x})}(t)}>\epsilon\Bigr\} \leq e^{-Ch}.
\end{equation}

Now we assume that, for all $t=0,1,\ldots, T-1$, \eqref{eqn:prop:asym_opt:LP:3-3} holds.
Then, for $t=T$ and any $\epsilon,\epsilon'>0$, there exist $H<\infty$ and $C,C'>0$ such that 
\begin{multline}\label{eqn:prop:asym_opt:LP:3-4}
  \bbP\Bigl\{\norm{\bm{y}^{\orule,h}(t)-\bm{y}^{\dborule(\bm{x})}(t)}>\epsilon\Bigr\}
  \leq \bbP\Bigl\{\norm{\bbE^{\orule}_{\bm{y}^0}\bigl(\bZ^{\orule,h}(t-1)\bigr)^T-\bigl(\bz^{\dborule(\bm{x})}(t-1)\bigr)^T}>\epsilon\Bigr\}\\
  \leq  \bbP\Bigl\{\norm{\bz^{\orule,h}(t-1)-\bz^{\dborule(\bm{x})}(t-1)}>\epsilon~\Bigl|~\norm{\bZ^{\orule,h}(t-1)-\bz^{\orule,h}(t-1)}\leq \epsilon'\Bigr\} + e^{-Ch}\\
  \leq  \bbP\Bigl\{\norm{\bZ^{\orule,h}(t-1)-\bz^{\dborule(\bm{x})}(t-1)}>\epsilon-\epsilon'~\Bigl|~\norm{\bZ^{\orule,h}(t-1)-\bz^{\orule,h}(t-1)}\leq \epsilon'\Bigr\} + e^{-Ch}\\
  \leq e^{-C' h},
\end{multline}
where the third inequality is from Theorem~\ref{theorem:convergence_Z_exp}, and the last inequality is based on the above assumption that \eqref{eqn:prop:asym_opt:LP:3-3} for all $t=0,1,\ldots,T-1$.
For $t=T$,
\begin{multline}
\lim_{h\to\infty}\bbP\Bigl\{\norm{\bz^{\orule,h}(t)-\bz^{\dborule(\bm{x})}(t)}>\epsilon \Bigr\}\\
\leq \lim_{h\to\infty}\bbP\Bigl\{\norm{\bigl(\bm{y}^{\orule,h}(t)\bigr)^T\mathcal{A}^{\orule,h}(t)-\bigl(\bm{y}^{\dborule(\bm{x})}(t)\bigr)^T\mathcal{A}^{\dborule(\bm{x}),h}(0))}>\epsilon~\Bigl|~\norm{\bm{y}^{\orule,h}(t)-\bm{y}^{\dborule(\bm{x})}(t)}\leq \epsilon'\Bigr\} \\
\leq \lim_{h\to\infty}\bbP\Bigl\{\norm{\bigl(\bm{y}^{\orule,h}(t)\bigr)^T(\mathcal{A}^{\orule,h}(t)-\mathcal{A}^{\dborule(\bm{x}),h})}>\epsilon-\epsilon'~\Bigl|~\norm{\bm{y}^{\orule,h}(t)-\bm{y}^{\dborule(\bm{x})}(t)}\leq \epsilon'\Bigr\}\\
\leq \lim_{h\to\infty}\bbP\Bigl\{\norm{\bigl(\bm{Y}^{\orule,h}(t)\bigr)^T(\mathcal{A}^{\orule,h}(t)-\mathcal{A}^{\dborule(\bm{x}),h})}>\epsilon-2\epsilon'~\Bigl|~\norm{\bm{y}^{\orule,h}(t)-\bm{y}^{\dborule(\bm{x})}(t)}\leq \epsilon',\norm{\bm{Y}^{\orule,h}(t)-\bm{y}^{\dborule(\bm{x})}(t)}\leq \epsilon'\Bigr\}
\\=0,
\end{multline}
where the first inequality is based on \eqref{eqn:prop:asym_opt:LP:3-4}, the third inequality comes from Theorem~\ref{theorem:convergence_Z_exp}, and the last equality is from \eqref{eqn:assumption:action}.
Recall \eqref{eqn:prop:asym_opt:LP:3-1}, we obtain that \eqref{eqn:prop:asym_opt:LP:3-3} holds for $t=T$.
That is, for any $\epsilon>0$, there exists $H<\infty$ and $C>0$ such that, for all $h>H$, \eqref{eqn:prop:asym_opt:LP:1} holds.

For any $\epsilon>0$ and $t\in[T]_0$, there exist $H<\infty$ and $C,C_1,C_2>0$ such that, for all $h>H$,
\begin{multline}\label{eqn:prop:asym_opt:LP:3-5}
    \norm{\bz^{\orule,h}(t)-\bz^{\dborule(\bm{x})}(t)}\\
    \leq \bbE\Bigl\{\norm{\bZ^{\orule,h}(t)-\bz^{\dborule(\bm{x})}(t)}~\Bigl|~\norm{\bZ^{\orule,h}(t)-\bz^{\orule,h}(t)}\leq \epsilon\Bigr\} + e^{-Ch}\\
    \leq C_1 e^{-C_2 h} + \epsilon,
\end{multline}
where the inequalities are based on Theorem~\ref{theorem:convergence_Z_exp} and \eqref{eqn:prop:asym_opt:LP:1}, respectively.

It remains to show the necessity of \eqref{eqn:assumption:action}.
Based on Theorem~\ref{theorem:convergence_Z_exp} and \eqref{eqn:prop:asym_opt:LP:1},
\begin{equation}
    \lim_{h\to \infty} \max_{t\in [T]_0} \norm{\bz^{\orule,h}(t)-\bz^{\orule(\bm{x})}(t)} = 0.
\end{equation}
Hence, for any $t\in[T]_0$,
\begin{equation}\label{eqn:prop:asym_opt:LP:3-6}
    \lim_{h\to\infty} \bm{y}^{\orule,h}(t) = \lim_{h\to\infty} \bm{y}^{\orule(\bm{x})}(t),
\end{equation}
and
\begin{equation}\label{eqn:prop:asym_opt:LP:3-7}
    \lim_{h\to\infty} \bbE\bigl[\bZ^{\orule,h}(t)\bigl|\bm{Y}^{\orule,h}(t)\bigr]=\lim_{h\to\infty} \bbE\bigl[\bZ^{\orule,h}(t)\bigl|\bm{Y}^{\orule,h}(t)\bigr].
\end{equation}

For any $\epsilon>0$, $(i,s,\action)\in\calJ$ and $t\in[T]_0$,
\begin{multline}
    \lim_{h\to\infty}\bbP\Bigl\{\Bigl\lvert \bigl(\alpha^{\orule,h}_{i,s,\action}(\bm{Y}^{\orule,h}(t),t)-\oalphaLP_{i,s,\action}(\bm{x},t)\bigr)Y^{\orule,h}_{i,s}(t)\Bigr\rvert>\epsilon\Bigr\}\\
    = \lim_{h\to\infty}\bbP\Bigl\{\Bigl\lvert \bbE\bigl[Z^{\orule,h}_{i,s,\action}(t)\bigl|\bm{Y}^{\orule,h}(t)\bigr]-\oalphaLP_{i,s,\action}(\bm{x},t)Y^{\orule,h}_{i,s}(t)\Bigr\rvert>\epsilon\Bigr\}\\
    \leq \lim_{h\to\infty}\bbP\Bigl\{\Bigl\lvert \bbE\bigl[Z^{\orule,h}_{i,s,\action}(t)\bigl|\bm{Y}^{\orule,h}(t)\bigr]-\oalphaLP_{i,s,\action}(\bm{x},t)Y^{\orule(\bm{x})}_{i,s}(t)\Bigr\rvert
    +\Bigl\lvert \oalphaLP_{i,s,\action}(\bm{x},t)\bigl(Y^{\orule(\bm{x})}_{i,s}(t)-Y^{\orule,h}_{i,s}(t)\bigr)\Bigr\rvert  
    >\epsilon\Bigr\}\\
    =   \lim_{h\to\infty}\bbP\Bigl\{\Bigl\lvert \bbE\bigl[Z^{\orule,h}_{i,s,\action}(t)\bigl|\bm{Y}^{\orule,h}(t)\bigr]-\bbE\bigl[Z^{\orule(\bm{x})}_{i,s,\action}(t)\bigl|\bm{Y}^{\orule(\bm{x})}(t)\bigr]\Bigr\rvert
    +\Bigl\lvert \oalphaLP_{i,s,\action}(\bm{x},t)\bigl(Y^{\orule(\bm{x})}_{i,s}(t)-Y^{\orule,h}_{i,s}(t)\bigr)\Bigr\rvert  
    >\epsilon\Bigr\}
    \\=0,
\end{multline}
where the first equality is based on the definition of $\balpha^{\orule,h}$ in \eqref{eqn:define:alpha}, the last equality comes from \eqref{eqn:prop:asym_opt:LP:3-6}, \eqref{eqn:prop:asym_opt:LP:3-7}, and Theorem~\ref{theorem:convergence_Z_exp}.
That is, \eqref{eqn:assumption:action} holds.
It proves the proposition.


\endproof

\section{Simulation Settings}\label{app:simulation:settings}
The simulation was coded through g++ 11.4.0 (Linux) and implemented on the high performance computing platform, Spartan, offered by The University of Melbourne.

We consider a WCG system with $I=5$, $T=30$, and $L=6$.
\begin{enumerate}[label=(\alph*)]
    \item The state and action spaces are $\bS_1 = \{0,1\}$, $\bA_1=\{0,1\}$; 
    $\bS_2 =\bS_3= \{0,1,2\}$, $\bA_2=\bA_3=\{0,1,2\}$;
    $\bS_4 = \{0,1,2,3\}$, $\bA_4=\{0,1,2,3\}$; and
    $\bS_5 = \{0,1,\ldots,30\}\times [T]_0$, $\bA_5=\{0\}$.
    \item The transition matrices
    $\mathcal{P}_i(a)=\bigl[p_i(s,a,s')\bigr]_{|\bS_i|\times|\bS_i|}$ ($i\in[I]$) are 
    \[\mathcal{P}_1(0) = \left[\begin{array}{ll}
    1&0\\0.3 & 0.7
    \end{array}\right], 
    \mathcal{P}_1(1) = \left[\begin{array}{ll}
    0&1\\0.3 & 0.7
    \end{array}\right];
    \]
    \[\mathcal{P}_2(0)=\mathcal{P}_3(0)  = \left[\begin{array}{lll}
    1&0&0\\0.3 & 0.7 & 0\\0.09& 0.42 &0.49
    \end{array}\right], 
    \mathcal{P}_2(1) =\mathcal{P}_3(1) = \left[\begin{array}{lll}
    0&1&0\\0&0.3 & 0.7 \\0.09& 0.42& 0.49
    \end{array}\right],\]
    \[\mathcal{P}_2(2) = \mathcal{P}_3(2)= \left[\begin{array}{lll}
    0&0&1\\0&0.3 & 0.7 \\0.09& 0.42 &0.49
    \end{array}\right];
    \]
    \[\mathcal{P}_4(0)=\left[\begin{array}{llll}
    1&0&0&0 \\ 0.2&0.8 & 0& 0 \\0.04& 0.32 &0.64 & 0\\ 0.008 & 0.096 & 0.384 & 0.512
    \end{array}\right],
    \mathcal{P}_4(1)=\left[\begin{array}{llll}
    0&1&0&0 \\ 0&0.2&0.8 & 0 \\0&0.04& 0.32 &0.64 \\ 0.008 & 0.096 & 0.384 & 0.512
    \end{array}\right],\]
    \[
    \mathcal{P}_4(2)=\left[\begin{array}{llll}
    0&0&1&0 \\ 0&0&0.2&0.8 \\0&0.04& 0.32 &0.64 \\ 0.008 & 0.096 & 0.384 & 0.512
    \end{array}\right],
    \mathcal{P}_4(3)=\left[\begin{array}{llll}
    0&0&0&1 \\ 0&0&0.2&0.8 \\0&0.04& 0.32 &0.64 \\ 0.008 & 0.096 & 0.384 & 0.512
    \end{array}\right].
    \]
    For $s\in\bS_5 = \{0,1,\ldots,30\}\times[T]_0$, let $s\coloneqq(n^s,\tau^s)$. For $s\in\bS_i$ and $s'\in\bS_i\backslash \bigl(\{30\}\times[T]_0\bigr)$, the transition probabilities
    \[p_5(s,0,s') = \begin{cases}
        \frac{\lambda(\tau^s)^{s'}e^{-\lambda(\tau^s)}}{s'!},&\text{if } \tau^{s'} = \tau^s+1,\\
        0, & \text{otherwise},
    \end{cases}
    \] 
    where $(\lambda(\tau):\tau=0,1,\ldots,30) = (10,9.6,9,5.8,6.8,3.2,6.2,5.2,5.5,8.8,5,4.6,4,2.8,$ $1.8,3.2,3.2,2.2,2.5,0,6,3.6,3,2.8,3.8,3.2,5.2,5.2,2.5,4.4,0)$.
    In the case with $\tau^{s'} = \tau^s+1$, $p_5(s,0,s')$ is the probability of Poisson distribution with parameter $\lambda(\tau^s)$.
    For $s'=(30,\tau')$, $p_5(s,0,s') = 1- \sum_{s''\in\bS_5\backslash\bigl(\{30\}\times[T]_0\bigr)}p_5(s,0,s'')$ so that $\sum_{s''\in\bS_5}p_5(s,0,s'') = 1$.
\end{enumerate}
Here, for gangs $i=1,2,3,4$, $\sum_{n\in[N_i]}s^{\phi,h}_{i,n}(t)$ represent the number of customers that are being served by service $i$ at time $t$, and $\sum_{n\in[N_i]}\action^{\phi,h}_{i,n}(t)$ is the potential number of new customers that are allocated to service $i$ during time slot $t$.
If $s^{\phi,h}_{i,n}(t)+\action^{\phi,h}_{i,n}(t)$ exceeds the maximal possible number of customers on service $i$, $|\bS_i|-1$, (which is determined by resource capacities discussed in later paragraphs), then in reality only $\min\bigl\{|\bS_i|-1-s^{\phi,h}_{i,n}(t),\action^{\phi,h}_{i,n}(t)\bigr\}$ are allocated.

The transition probability $p_i(s,a,s')$ for $i=1,2,3,4$ is determined as follows.
If $a=1$, then a new arrival is added to service $i$; otherwise, there is no new arrival.
Meanwhile, each customer being served by service $i$ departures identically and independently with probability $d_i = 0.3,0.3,0.3,0.2$ for $i=1,2,3,4$, respectively. 
Let $D_i(t)$ represent the number of departures for service $i$ during time slot $t$.
The transition probability of gang $i=1,2,3,4$ is given by 
\[
p_i(s,a,s') = \mathbb{P}\Bigl\{s'= \min\{s+a-D_i(t), |\bS_i|-1\}~\Bigl|~s^{\phi,h}_{i,n}(t) = s, \action^{\phi,h}_{i,n}(t) = a\Bigr\}.
\]
Note that the value of $\action^{\phi,h}_{i,n}(t) = a$ is subject to the total number of arrivals - the state of gang $i=5$. This restriction will be reflected through the first constraint $\ell=1$ of \eqref{eqn:constraint:linear} discussed in the following.

The last gang $i=5$ simulates the arrivals of customers, for which $\sum_{n\in[N_5]}s^{\phi,h}_{5,n}(t)$ is the total number of newly arrived customers in time slot $t$.
The total number of allocated customers should be no more than the total number of arrivals in each time slot. That is, the constraint $\ell=1$ of \eqref{eqn:constraint:linear} is specified by
\begin{itemize}
    \item for all $i=1,2,3,4$ and $(s,a)\in\bS_i\times\bA_i$, $f_{i,1}(s,a) = \min\bigl\{|\bS_i|-1-s,a\bigr\}$, and
    \item for $i=5$ and $(s',a')\in\bS_5\times \bA_5$, $f_{5,1}(s',a') = -s$.
\end{itemize}
It assures
\[\sum_{i=1}^4\sum_{n\in[N_i]}\min\bigl\{|\bS_i|-1-s^{\phi,h}(t),a^{\phi,h}_{i,n}(t) \bigr\}\leq \sum_{n\in[N_5]}s^{\phi,h}_{5,n}(t)\,~\forall t\in[T]_0.\]

Serving the customers occupies resource units from five different resource pools, incurring the other five constraints,  $\ell=2,3,\ldots,6$ of \eqref{eqn:constraint:linear}, each for which imposes the capacity constraint over resource pool $\ell-1$.
We specify
\begin{itemize}
    \item for $\ell=2,3,\ldots,6$, $i=1,2,3,4$, and $(s,a)\in\bS_i\times \bA_i$, $f_{i,\ell}(s,a) = w_{i,\ell-1}\min\bigl\{|\bS_i|-1,s+a\bigr\} - \frac{C_{\ell-1}}{(I-1)N_i^0}$, where $w_{i,\ell-1}\in\mathbb{R}_0$ is the number of resource units (possibly real numbers) in pool $\ell-1$ occupied by each customer being served by service $i$, and $C_{\ell-1}\in\mathbb{R}_0$ is the capacity of resource pool $\ell-1$ for $h=1$; and
    \item for $\ell=2,3,\ldots,6$, $i=5$, and $(s,a)\in\bS_i\times \bA_i$, $f_{5,\ell}(s,a) = 0$.
\end{itemize}
In this case, constraints $\ell=2,3,\ldots,6$ of \eqref{eqn:constraint:linear} are equivalent to
\[
\sum_{i=1}^4\sum_{n\in[N_i]}w_{i,\ell-1}\bigl(s^{\phi,h}_{i,n}(t) + \action^{\phi,h}_{i,n}(t)\bigr) \leq h C_{\ell-1},~\forall t\in[T]_0.
\]
In particular, 
\[[w_{i,\ell-1}]_{4 \times 5} = \left[
\begin{array}{lllll}
0&3&0&1&0\\
1&0&0&0&3\\
3&0&1&0&0\\
0&0&3&1&0
\end{array}
\right],
\]
and $(C_1,C_2,\ldots,C_5) = (8,5,9,7,8)$.

For $i=1,2,3,4$ and $(s^{\phi,h}_{i,n}(t),\action^{\phi,h}_{i,n}(t))=(s,a)\in\bS_i\times \bA_i$, the $s$ customers may complete their service and leave the system with departure probabilities reflected through the above-specified transition probabilities $p_i(s,a,s')$ for $s' < \min\bigl\{|\bS_i|-1,s+a\bigr\}$.
Once a customer leaves the system, its occupied resource units will be immediately released for future use.

When a customer is served (finishing service) by service $i=1,2,3,4$, the system, as a service provider/coordinator, earns $R_i = 110.3085,122.0775,120.3569,100.182$, respectively, dollars. 
Meanwhile, by using units of resource pool $\ell-1$, the system is charged $\mathcal{C}_{\ell-1}=1.3687,1.0789,0.6667,0.1523,1.9761$ for $\ell=2,3,\ldots,6$, respectively, per resource unit per time slot.
These numbers for $R_i$ and $\mathcal{C}_{\ell-1}$ are generated by pseudo-random generator of g++ 11.4.0 (Linux). 
It follows that, for \GSA triple $(i,s,a)$ ($i=1,2,3,4$, $s\in\bS_i$, and $a\in\{0,1\}$), the reward function
\begin{equation}\label{eqn:obj:simulation}
    r_i(s,a) = R_i\mathbb{E} D_i(t) - \sum_{\ell=2}^6\mathcal{C}_{\ell-1}w_{i,\ell-1} = R_i s d_i - \sum_{\ell=2}^6\mathcal{C}_{\ell-1}w_{i,\ell-1},  
\end{equation}
where recall that $D_i(t)$ is the number of departures (completions) of customers from service $i$ (identically and independently distributed for all bandit processes in gang $i$), and $d_i$ is the departure probability for each customer being served by service $i$.

\section{Joint Lipschitz continuity of $b^{\Lipschitza,U}$}\label{app:continuity_bU}

During this appendix, we consider the case with $\Lipschitza\in(0,1)$. Based on the linearity of $Q^{h,\Lipschitza}$ in $\bm{\xi}$, there exists $\tilde{\mathcal{Q}}^{h,\Lipschitza}(\bm{x})$ such that
\begin{equation}\label{eqn:app:conitinuity_b:1}
    b^{h,\Lipschitza,U}(\bm{x},\bm{\xi}) = \tilde{\mathcal{Q}}^{h,\Lipschitza}(\bm{x}) \min\{U,\bm{\xi}\},
\end{equation}
where $\tilde{\mathcal{Q}}^{h,\Lipschitza}(\bm{x})$ is independent to $\bm{\xi}$.
Based on the separate Lipschitz continuity of $Q^{1,\Lipschitza}$ in $\bm{x}\in\mathbb{R}^{|\mathscr{K}|}$ and $\bm{\xi}\in\mathbb{R}^{|\mathscr{K}\sum_{\kappa\in\mathscr{K}}N_{i^{\kappa}}}$,  $b^{\Lipschitza,U}=b^{h,\Lipschitza,U}|_{h=1}$ is also separately Lipschitz continuous for such $\bm{x}$ and $\bm{\xi}$.

We then show the joint, Lipschitz continuity of $b^{\Lipschitza,U}(\bm{x},\bm{\xi})$ in $\bm{x}\in\mathbb{R}^{|\mathscr{K}|}$ and $\bm{\xi}\in\mathbb{R}^{|\mathscr{K}\sum_{\kappa\in\mathscr{K}}N_{i^{\kappa}}+1}$.
Based on \eqref{eqn:app:conitinuity_b:1}, for any $\bm{x}_0\in\mathbb{R}^{|\mathscr{K}|}$ and $\bm{\xi}_0\in\mathbb{R}^{|\mathscr{K}\sum_{\kappa\in\mathscr{K}}N_{i^{\kappa}}+1}$,
\begin{equation}\label{eqn:app:conitinuity_b:2}
    \lim_{(\bm{x},\bm{\xi})\to (\bm{x}_0,\bm{\xi}_0)}b^{\Lipschitza,U}(\bm{x},\bm{\xi}) = \lim_{(\bm{x},\bm{\xi})\to (\bm{x}_0,\bm{\xi}_0)}\tilde{\mathcal{Q}}^{1,\Lipschitza}(\bm{x}) \min\{U,\bm{\xi}\} = \tilde{\mathcal{Q}}^{1,\Lipschitza}(\bm{x}_0)\min\{U,\bm{\xi}_0\} = b^{\Lipschitza,U}(\bm{x}_0,\bm{\xi}_0),
\end{equation}
where the second equality holds because $\tilde{\mathcal{Q}}^{1,\Lipschitza}(\bm{x})$ is continuous in $\bm{x}$ and independent to $\bm{\xi}$.
That is, $b^{\Lipschitza,U}(\bm{x},\bm{\xi})$ is jointly continuous in $\bm{x}\in\mathbb{R}^{|\mathscr{K}|}$ and $\bm{\xi}\in\mathbb{R}^{|\mathscr{K}\sum_{\kappa\in\mathscr{K}}N_{i^{\kappa}}+1}$.

For $\bm{x}_1,\bm{x}_2\in\mathbb{R}^{|\mathscr{K}|}$, $\bm{\xi}_1,\bm{\xi}_2\in\mathbb{R}^{|\mathscr{K}\sum_{\kappa\in\mathscr{K}}N_{i^{\kappa}}+1}$, $\Delta x=\lVert \bm{x}_1-\bm{x}_2\rVert$, and $\Delta\xi = \lVert \bm{\xi}_1-\bm{\xi}_2\rVert$,
\begin{multline}\label{eqn:app:conitinuity_b:3}
 \lVert b^{\Lipschitza,U}(\bm{x}_1,\bm{\xi}_1)-b^{\Lipschitza,U}(\bm{x}_2,\bm{\xi}_2)\rVert = \bigl\lVert \tilde{\mathcal{Q}}^{1,\Lipschitza}(\bm{x}_1) \min\{U,\bm{\xi}_1\} -  \tilde{\mathcal{Q}}^{1,\Lipschitza}(\bm{x}_2) \min\{U,\bm{\xi}_2\} \bigr\rVert   \\
 =\bigl\lVert \tilde{\mathcal{Q}}^{1,\Lipschitza}(\bm{x}_1)\min\{U,\bm{\xi}_1\} - \tilde{\mathcal{Q}}^{1,\Lipschitza}(\bm{x}_1) \min\{U,\bm{\xi}_2\} + \tilde{\mathcal{Q}}^{1,\Lipschitza}(\bm{x}_1)\min\{U,\bm{\xi}_2\} - \tilde{\mathcal{Q}}^{1,\Lipschitza}(\bm{x}_2)\min\{U,\bm{\xi}_2)\bigr\rVert\\
 \leq \bigl\lVert \tilde{\mathcal{Q}}^{1,\Lipschitza}(\bm{x}_1)(\min\{U,\bm{\xi}_1\} -\min\{U, \bm{\xi}_2\})\bigr\rVert + \bigl\lVert (\tilde{\mathcal{Q}}^{1,\Lipschitza}(\bm{x}_1)-\tilde{\mathcal{Q}}^{1,\Lipschitza}(\bm{x}_2))\min\{U,\bm{\xi}_2\}\bigr\rVert\\
 \leq \bigl\lVert \tilde{\mathcal{Q}}^{1,\Lipschitza}(\bm{x}_1)\bigr\rVert \bigl\lVert\min\{U,\bm{\xi}_1\} -\min\{U, \bm{\xi}_2\}\bigr\rVert + \bigl\lVert \tilde{\mathcal{Q}}^{1,\Lipschitza}(\bm{x}_1)-\tilde{\mathcal{Q}}^{1,\Lipschitza}(\bm{x}_2)\bigr\rVert\bigl\lVert\min\{U,\bm{\xi}_2\}\bigr\rVert
\\
\leq K_1 \Delta\xi +K_2 \Delta x,
\end{multline}
where there exist $K_1,K_2 < \infty$ independent to $\bm{x}_1,\bm{x}_2,\bm{\xi}_1,\bm{\xi}_2$,  and the last inequality is based on the boundness and Lipshitz continuity of $\tilde{\mathcal{Q}}^{1,\Lipschitza}(\bm{x})$ in $\bm{x}\in\mathbb{R}^{|\mathscr{K}|}$.
Since 
\[\bigl\lVert (\bm{x}_1;\bm{\xi}_1) - (\bm{x}_2;\bm{\xi}_2)\bigr\rVert \geq \frac{1}{2}\Bigl(\lVert\bm{x}_1-\bm{x}_2 \rVert + \lVert \bm{\xi}_1-\bm{\xi}_2\rVert \Bigr)=\frac{1}{2}(\Delta x + \Delta \xi),
\]
together with \eqref{eqn:app:conitinuity_b:3}, there exists $K<\infty$ independent from $\bm{x}_1,\bm{x}_2,\bm{\xi}_1,\bm{\xi}_2$ such that
\begin{equation}\label{eqn:app:conitinuity_b:4}
   \lVert b^{\Lipschitza,U}(\bm{x}_1,\bm{\xi}_1)-b^{\Lipschitza,U}(\bm{x}_2,\bm{\xi}_2)\rVert \leq  K\bigl\lVert (\bm{x}_1;\bm{\xi}_1) - (\bm{x}_2;\bm{\xi}_2)\bigr\rVert.
\end{equation}
It proves Lipschitz continuity of $b^{\Lipschitza,U}(\bm{x},\bm{\xi})$ jointly in $\bm{x}\in\mathbb{R}^{|\mathscr{K}|}$, $\bm{\xi}\in\mathbb{R}^{|\mathscr{K}\sum_{\kappa\in\mathscr{K}}N_{i^{\kappa}}+1}$.

\section{Unique solution to \eqref{eqn:app:convergence_Z:6}}
\label{app:existence_ODE}

\proof{Proof of Lemma~\ref{lemma:existence_ODE}.}
To prove the existence of solutions to \eqref{eqn:app:convergence_Z:6}, we use the Carath\'eodory conditions, as described in \cite{filippov}. Specifically, we invoke Theorems~1 and 2 in Chapter~1.  The Carath\'eodory conditions require,in our context,  that
$(\tau,\bm{x})\mapsto b^{\Lipschitza,U}(\bm{x},\bm{\xi}^{1}_{\tau/\sigma})$ is first defined on a set $[0,\bar{T}]\times D$, where $D$ is compact; that also,
\begin{enumerate}[label=(\Alph*)]
  \item it is continuous in $\bm{x}$ for almost all $t$;\label{cond:existence_ODE:1}
  \item it is measurable in $\tau$ for all $\bm{x}\in D$;
  \item $\norm{b^{\Lipschitza,U}(\bm{x},\bm{\xi}^{1}_{\tau/\sigma})}\le m(\tau)$, where $m(\tau)$ is integrable on $[0,\bar{T}]$.
\end{enumerate}
These conditions are relatively straightforward to verify, noting that
$\tau\mapsto \min(U, \bm{\xi}(\tau))$ is clearly bounded on $[0,\bar{T}]$. 

By \cite[Chapter 1, Theorem 1]{filippov}, on  the interval $\tau\in [0,\bar{T}]$, there exists a solution to \eqref{eqn:app:convergence_Z:6}, when the following conditions are satisfied.
\begin{enumerate}[label=(\Alph*)]\setcounter{enumi}{3}
    \item $b^{\Lipschitza,U}(\bm{x},\bm{\xi}^{1}_{\tau/\sigma})$ satisfies the Carath\'eodory conditions on $\tau\in[0,\bar{T}]$ and $\norm{\bm{x}-\bm{x}_0} \leq B$, where $\bm{x}_0$ is the initial condition $\bm{X}^{\sigma,U}_0 = \bm{x}_0$; and
    \item $\gamma(\bar{T})\leq B$, where $\gamma(\tau) \coloneqq \int_0^{\tau} m(\tau') d \tau'$.\label{cond:existence_ODE:5}
\end{enumerate}
We set $m(\tau) = M$ and $B=MT$, where $M < \infty$ is  an upper bound of $\norm{b^{\Lipschitza,U}(\bm{x},\bm{\xi})}$. Then the above two conditions are verified straightforwardly. 
We further invoke \cite[Chapter 1, Theorem 2]{filippov}.
For the compact set $D= \Bigl\{\bm{x}\in\mathbb{R}_0^{|\mathscr{K}|}~\Bigl|~\norm{\bm{x}-\bm{x}_0} \leq B\Bigr\}$, 
if there exists a summable function $\ell(\tau)$ such that, for any $\bm{x}_1,\bm{x}_2\in D$ and $\tau\in[0,\bar{T}]$,
\begin{equation}\label{eqn:app:existance_ODE:1}
\norm{b^{\Lipschitza,U}(\bm{x}_1,\bm{\xi}^{1}_{\tau/\sigma})-b^{\Lipschitza,U}(\bm{x}_2,\bm{\xi}^{1}_{\tau/\sigma})}\leq     \ell(\tau) \norm{\bm{x}_1-\bm{x}_2},
\end{equation}
then, in the domain $[0,\bar{T}]\times D$, there exists at most one solution of \eqref{eqn:app:convergence_Z:6}.
Based on the Lipschitz continuity of $b^{\Lipschitza,U}(\bm{x},\bm{\xi})$, we can easily verify \eqref{eqn:app:existance_ODE:1}.
Hence, there exists a unique solution of \eqref{eqn:app:convergence_Z:6} on the interval $\tau\in[0,\bar{T}]$.

\endproof

\section{Unique solution to \eqref{eqn:averaging_ODE}}\label{app:existence_ODE:averaging}
\proof{Proof of Lemma~\ref{lemma:existence_ODE:averaging}.}
Similar to the proof of Lemma~\ref{lemma:existence_ODE} in Appendix~\ref{app:existence_ODE}.
For any $\Lipschitza\in(0,1)$ and $U\in\mathbb{R}_0$, since $b^{\Lipschitza,U}$ satisfies Conditions~\ref{cond:existence_ODE:1}-\ref{cond:existence_ODE:5}, 
Conditions~\ref{cond:existence_ODE:1}-\ref{cond:existence_ODE:5} are also satisfied by replacing $b^{\Lipschitza,U}(\bm{x},\bm{\xi}^{1}_{\tau/\sigma})$ with $\bar{b}^{\Lipschitza,U}(\bm{x})$, where $\bar{b}^{\Lipschitza,U}(\bm{x}) = \mathbb{E}b^{\Lipschitza,U}(\bm{x},\bm{\xi}^{1}_{\tau/\sigma}) = \tilde{\mathcal{Q}}^{1,\Lipschitza}\mathbb{E}\min\{U,\bm{\xi}^{1}_{\tau/\sigma}\}$.
Based on the Lipschitz continuity of $\bar{b}^{\Lipschitza,U}(\bm{x})$, we can also verify the existence $\ell(t)$ for \eqref{eqn:app:existance_ODE:1}.
Hence, from \cite[Chapter 1, Theorems 1 and 2]{filippov}, there exists a unique solution of \eqref{eqn:averaging_ODE} on the interval $\tau\in[0,\bar{T}]$.

\endproof

\section{Unique solution to ODE for the limit case $\Lipschitza\downarrow 0$}

In this section, when $\Lipschitza\downarrow 0$, we will prove that there are unique solutions to \eqref{eqn:app:existence_ODE_limit:7}, \eqref{eqn:app:existence_ODE_limit:1}, and \eqref{eqn:app:existence_ODE_limit:27}.
Since $b^{h,U}(\bm{x},\bm{\xi}) = b^{h,U}(|\bm{x}|,\bm{\xi})$, in this section, we consider only the half hyperplane $\bm{x}\in\mathbb{R}_0^{|\mathscr{K}|}$.

Based on the definition of $f^{h,\Lipschitza}_{\kappa,\kappa'}(\bm{x},\bm{\xi})$ in Appendix~\ref{app:continuity_Q}, $\lim_{\Lipschitza\downarrow 0} f^{h,\Lipschitza}_{\kappa,\kappa'}(\bm{x},\bm{\xi}) = 0$. 

Based on \eqref{eqn:app:convergence_Z:3} and \eqref{eqn:app:convergence_Z:4}, the $\kappa$th element of $\beta^h(\bm{x},\bm{\xi})$ is
\begin{multline}\label{eqn:app:existence_ODE_limit:6}
    \beta^h_{\kappa}(\bm{x},\bm{\xi}) = \tilde{\Theta}^h(\bm{x})\bm{\xi} \\= \sum_{\kappa'\in\mathscr{K}}\Biggl(
\indicator\bigl\{x(t^{\kappa'}-1)=0\bigr\}\sum_{n\in\bigl[h\lceil x_{\kappa'}/h\rceil\bigr]}\xi(\kappa',\kappa,n)
 - \indicator\bigl\{x(t^{\kappa}-1)=0\bigr\}\sum_{n\in\bigl[h\lceil x_{\kappa}/h\rceil\bigr]}\xi(\kappa',\kappa,n)
\Biggr),
\end{multline}
where recall $x(t)$ is based on the definition in the beginning of Appendix~\ref{app:continuity_Q}.
Such a matrix function $\tilde{\Theta}^h(\bm{x}) = \lim_{\Lipschitza\downarrow 0}\tilde{\mathcal{Q}}^{h,\Lipschitza}(\bm{x})$ is  piece-wise constant and left-continuous in $\bm{x}\in\mathbb{R}_0^{|\mathscr{K}|}$ except those points with zero elements.

\subsection{Unique solution to \eqref{eqn:app:existence_ODE_limit:7}}\label{app:existence_ODE_limit:h=1}

\proof{Proof of Lemma~\ref{lemma:existence_ODE_limit:h=1}.}
For a very special case with $\bm{x}_0 = \bm{0}$, $\beta(\bm{0},\bm{\xi}) \equiv 0$. There is a unique solution $\bm{X}^{\sigma}_{\tau} = 0$ for all $\tau\in[0,\bar{T}]$.

We then consider the more general case with $\norm{\bm{x}_0}>0$.
We start with constructing a trajectory $\bm{\chi}^{\sigma}_{\tau}$ and then prove that the constructed trajectory solves \eqref{eqn:app:existence_ODE_limit:7}.


Recall the Poisson random process $\tau \mapsto M_t(\tau)$ for $t\in[T]_0$ defined in \eqref{eq:summed_ms}, for which $\tau_t(m)$ is the occurrence time of the $m$th event, and this event is labeled by $(\kappa_t(m),\zeta_t(m),n_t(m))$.
Based on \eqref{eq:xi_definition_new}, when $\tau=\tau_t(m)$, the random variable $\xi_{\tau}(\kappa,\kappa',n) > 0$ only if $\kappa = \kappa_t(m)$, $\zeta^{\kappa'} = \zeta_t(m)$, $n = n_t(m)$, $t^{\kappa'}=t^{\kappa}+1$, and $\tau$ is not in the silent period $\scrR$. 
For any such $(\kappa,\kappa',n)$ and $\tau\notin \scrR$, $\xi_{\tau'}(\kappa,\kappa',n)$ is a constant for all $\max\bigl\{\tau_t(m-1),\lfloor \tau\rfloor\bigr\}\leq\tau'< \tau_t(m)$.
Meanwhile, for the silent period $\tau\in\scrR$, $\xi_{\tau}(\kappa,\kappa',n) = 0$ for all $\kappa,\kappa'\in\mathscr{K}$ and $n\in[N_{i^{\kappa}}]$.
That is, for $\tau\in [0,\infty)$, $\bm{\xi}_{\tau}$ is piece-wise constant.

Recall the definition of $\beta(\bm{x},\bm{\xi})$ in \eqref{eqn:app:existence_ODE_limit:5}, we can expend it as
\begin{equation}
\beta(\bm{x},\bm{\xi})=\lim_{\Lipschitza \downarrow 0, U\to \infty}\sum_{\kappa'\in\mathscr{K}}\Bigl(Q^{1,\Lipschitza}(\kappa',\kappa,\bm{x},\bm{\xi})-Q^{1,\Lipschitza}(\kappa,\kappa',\bm{x},\bm{\xi})\Bigr),
\end{equation}
where $Q^{1,\Lipschitza}$ is defined in \eqref{eqn:app:convergence_Z:3}.
That is, 
\begin{equation}\label{eqn:define_theta}
\beta(\bm{x},\bm{\xi})=\sum_{\kappa'\in\mathscr{K}}\Bigl(\Theta(\kappa',\kappa,\bm{x},\bm{\xi})-\Theta(\kappa,\kappa',\bm{x},\bm{\xi})\Bigr),
\end{equation}
where 
\begin{multline}\label{eqn:expand_theta}
    \Theta(\kappa,\kappa',\bm{x},\bm{\xi})  \coloneqq \lim_{\Lipschitza \downarrow 0, U\to \infty}Q^{1,\Lipschitza}(\kappa,\kappa',\bm{x},\bm{\xi})
    =\indicator\biggl\{\sum_{\begin{subarray}~\kappa''\in \mathscr{K}:\\t^{\kappa''} <t^{\kappa}\end{subarray}}\lvert x_{\kappa''}\rvert =0,\sum_{\begin{subarray}~\kappa'' \in \mathscr{K}:\\t^{\kappa''} = t^{\kappa},\\\action^{\kappa''}=\dummyAction\end{subarray}}\lvert x_{\kappa''}\rvert=0\biggr\}\sum_{n\in\bigl[ \bigl\lceil  \lvert x_{\kappa}\rvert\bigr\rceil\bigr]}\xi(\kappa',\kappa,n) \\
+ \indicator\biggl\{\sum_{\begin{subarray}~\kappa''\in \mathscr{K}:\\t^{\kappa''} <t^{\kappa}\end{subarray}}\lvert x_{\kappa''}\rvert =0,\sum_{\begin{subarray}~\kappa'' \in \mathscr{K}:\\t^{\kappa''} = t^{\kappa},\\\action^{\kappa''}=\dummyAction\end{subarray}}\lvert x_{\kappa''}\rvert>0, i^{\kappa} = i^{\kappa'},s^{\kappa}=s^{\kappa'}, \action^{\kappa} = \dummyAction,t^{\kappa}=t^{\kappa'}\biggr\}\\
\times\sum_{\begin{subarray}~\kappa''\in\mathscr{K}:\\i^{\kappa''}=i^{\kappa},\\s^{\kappa''}=s^{\kappa},\\t^{\kappa}=t^{\kappa'}\end{subarray}}x_{\kappa''}\scra^{\orule}_{\zeta^{\kappa'}}(\frac{\bm{x}}{\sum_{i\in[I]}N_i^0}t^{\kappa'})\xi(0).
\end{multline}
At the right hand side of \eqref{eqn:expand_theta}, we plug in $\bm{\xi} = \bm{\xi}^1_{\tau}$, then the first and the second term potentially becomes non-zero only if $\tau$ is in and not in, respectively, the silent period $\scrR$.

When $\tau\notin\scrR$, the first term in \eqref{eqn:expand_theta} is piece-wise constant in $\bm{x}$, and the second term is zero.
That is, given a trajectory $\bm{\xi}^1_{\tau}$ for a non-silent period $[\tau_0, \tau_0+1)$ and $\bm{X}^{\sigma}_{\sigma\tau_0}=\bm{x}$, the differential equation $\beta(\bm{X}^{\sigma}_{\tau},\bm{\xi}^1_{\tau/\sigma})$ is piece-wise constant in $\tau\in[\sigma\tau_0,\sigma\tau_0+\sigma)$.
There exists a unique solution of such $\bm{X}^{\sigma}_{\tau}$ for $\tau\in[\sigma\tau_0,\sigma\tau_0+\sigma)$.

When $\tau\in \scrR$, the first term at the right hand side of \eqref{eqn:expand_theta} is zero, and the second term is non-zero only if $s^{\kappa} = s^{\kappa'}$, $i^{\kappa}=i^{\kappa'}$, and $t^{\kappa}=t^{\kappa'}$.
In other words, with given trajectory $\bm{\xi}^1_{\tau}$ for a silent period $\tau\in[\tau_1,\tau_1+1)$, for any given $i\in[I]$, $s\in \bS_i$, and $t\in[T]_0$,
$$\sum_{\kappa \in \mathscr{K}: i^{\kappa}=i,s^{\kappa}=s,t^{\kappa}=t}\beta_{\kappa}(\bm{X}^{\sigma}_{\tau},\bm{\xi}^1_{\tau/\sigma}) = 0,$$
during $\tau\in[\sigma\tau_1,\sigma\tau_1+\sigma)$.
It follows that, for any $i\in[I]$, $s\in\bS_i$, and $t\in[T]_0$, $\sum_{\kappa''\in\mathscr{K}:i^{\kappa''}=i,s^{\kappa''} = s,t^{\kappa''}=t}X^{\sigma}_{\tau,\kappa''}$ remains constant during $\tau\in [\sigma\tau_1,\sigma\tau_1+\sigma)$.
Thus, within the same period, based on the definition of $\bm{\scra}^{\orule}$ in \eqref{eqn:define:alpha:x}, $\bm{\scra}^{\orule}(\bm{X}^{\sigma}_{\tau}/\sum_{i\in[I]}N_i^0,t)$ is also constant, making the second term at the right hand side of \eqref{eqn:expand_theta} constant.
For such silent period $[\sigma\tau_1,\sigma\tau_1+\sigma)$, given $\bm{X}^{\sigma}_{\sigma\tau_1} = \bm{x}$, there exists a unique $\bm{X}^{\sigma}_{\tau}$ satisfying \eqref{eqn:app:existence_ODE_limit:7}.

Given $\bm{X}^{\sigma}_{0} = \bm{x}$, the first period $[0,\sigma)$ is a silent period, for which we have a unique $\bm{X}^{\sigma}_{\tau}$ that solves \eqref{eqn:app:existence_ODE_limit:7}.
Since the differential equation $|\beta(\bm{x},\bm{\xi})| < \infty$ for any $\bm{x}$ and $\bm{\xi}$, $\bm{X}^{\sigma}_{\sigma} = \lim_{\tau\to \sigma} \bm{X}^{\sigma}_{\tau}$. 
Then, given such $\bm{X}^{\sigma}_{\sigma}$, we have the unique solution for the second period $[\sigma, 2\sigma)$.
We can construct the unique solution for all future periods $[K\sigma, (K+1)\sigma)$ iteratively.
It proves the lemma.

\endproof

\subsection{Unique solution to \eqref{eqn:app:existence_ODE_limit:1} for the limit case $\Lipschitza\downarrow 0$}\label{app:existence_ODE_limit:h>1}

In this subsection, we consider the case with given $\bm{\xi}^{h}_{\tau}$, $\tau\geq 0$, where $\tau_t(m), \kappa_t(m), \zeta_t(m)$, and $n_t(m)$ all correspond to the case with given $h \in\mathbb{N}_+$.

\proof{Proof of Lemma~\ref{lemma:existence_ODE_limit:h>1}.}
For the special case with $\bm{x}_0 = \bm{0}$, $b^h(\bm{0},\bm{\xi}) \equiv 0$. Hence, there is a unique solution $\bm{X}^h_{\tau}=0$ for $\tau\in[0,\bar{T}]$.

We then consider the more general case with $\norm{\bm{x}_0}>0$.
Similar to the proof for Lemma~\ref{lemma:existence_ODE_limit:h=1}, we will show that $b^h(h\bm{X}^h_{\tau},\bm{x}^h_{\tau})$ is piece-wise constant in $\tau$.

Recall the definition of $\beta^h(\bm{x},\bm{\xi})$ in \eqref{eqn:app:existence_ODE_limit:3}, we can expend it as
\begin{equation}\label{eqn:define_theta:h}
\beta^h_{\kappa}(\bm{x},\bm{\xi})=\sum_{\kappa'\in\mathscr{K}}\Bigl(\Theta^h(\kappa',\kappa,\bm{x},\bm{\xi})-\Theta^h(\kappa,\kappa',\bm{x},\bm{\xi})\Bigr),
\end{equation}
where 
\begin{multline}\label{eqn:expand_theta:h}
    \Theta^h(\kappa,\kappa',\bm{x},\bm{\xi})  \coloneqq \lim_{\Lipschitza \downarrow 0}Q^{h,\Lipschitza}(\kappa,\kappa',\bm{x},\bm{\xi})\\
    =\indicator\biggl\{\sum_{\begin{subarray}~\kappa''\in \mathscr{K}:\\t^{\kappa''} <t^{\kappa}\end{subarray}}\lvert x_{\kappa''}\rvert =0,\sum_{\begin{subarray}~\kappa'' \in \mathscr{K}:\\t^{\kappa''} = t^{\kappa},\\\action^{\kappa''}=\dummyAction\end{subarray}}\lvert x_{\kappa''}\rvert=0\biggr\}\sum_{n\in\bigl[ h\bigl\lceil  \lvert x_{\kappa}/h\rvert\bigr\rceil\bigr]}\xi(\kappa',\kappa,n) \\
+ \indicator\biggl\{\sum_{\begin{subarray}~\kappa''\in \mathscr{K}:\\t^{\kappa''} <t^{\kappa}\end{subarray}}\lvert x_{\kappa''}\rvert =0,\sum_{\begin{subarray}~\kappa'' \in \mathscr{K}:\\t^{\kappa''} = t^{\kappa},\\\action^{\kappa''}=\dummyAction\end{subarray}}\lvert x_{\kappa''}\rvert>0, i^{\kappa} = i^{\kappa'},s^{\kappa}=s^{\kappa'}, \action^{\kappa} = \dummyAction,t^{\kappa}=t^{\kappa'}\biggr\}\\
\times\sum_{\begin{subarray}~\kappa''\in\mathscr{K}:\\i^{\kappa''}=i^{\kappa},\\s^{\kappa''}=s^{\kappa},\\t^{\kappa}=t^{\kappa'}\end{subarray}}x_{\kappa''}\scra^{\orule}_{\zeta^{\kappa'}}(\frac{\bm{x}}{h\sum_{i\in[I]}N_i^0},t^{\kappa'})\xi(0).
\end{multline}
At the right hand side of \eqref{eqn:expand_theta:h}, we plug in $\bm{\xi} = \bm{\xi}^h_{\tau}$, then the first and the second term potentially becomes non-zero only if $\tau$ is in and not in, respectively, the silent period $\scrR$.

When $\tau\notin\scrR$, the first term in \eqref{eqn:expand_theta:h} is piece-wise constant in $\bm{x}$, and the second term is zero.
It leads to the piece-wise constant $\beta(h\bm{X}^h_{\tau},\bm{\xi}^1_{\tau })$ in $\tau$ during the non-silent periods.
Given an initial condition, there exists a unique solution $\bm{X}^h_{\tau}$ for each non-silent period.

When $\tau\in \scrR$, the first term at the right hand side of \eqref{eqn:expand_theta:h} is zero, and the second term is non-zero only if $s^{\kappa} = s^{\kappa'}$, $i^{\kappa}=i^{\kappa'}$, and $t^{\kappa}=t^{\kappa'}$.
In other words, with given trajectory $\bm{\xi}^h_{\tau}$ for a silent period $\tau\in[\tau_1,\tau_1+1)$, for any given $i\in[I]$, $s\in \bS_i$, and $t\in[T]_0$,
$$\sum_{\kappa \in \mathscr{K}: i^{\kappa}=i,s^{\kappa}=s,t^{\kappa}=t}\beta^h_{\kappa}(h\bm{X}^{\sigma}_{\tau},\bm{\xi}^1_{\tau}) = 0,$$
during $\tau\in[\tau_1,\tau_1+1)$.
Thus $\sum_{\kappa''\in\mathscr{K}:i^{\kappa''}=i,s^{\kappa''} = s,t^{\kappa''}=t}X^h_{\tau,\kappa''}$ remains constant during $\tau\in [\tau_1,\tau_1+1)$.
Together with the definition of $\bm{\scra}^{\orule}$ in \eqref{eqn:define:alpha:x}, the second term at the right hand side of \eqref{eqn:expand_theta:h} is also constant.
For such silent period $[\tau_1,\tau_1+1)$, given $\bm{X}^h_{\tau_1} = \bm{x}$, there exists a unique $\bm{X}^h_{\tau}$ satisfying \eqref{eqn:app:existence_ODE_limit:1}. 

Then We can iteratively construct a unique $\bm{X}^h_{\tau}$ for all $\tau\in[0,\bar{T}]$ with given $\bm{X}^h_0=\bm{x}_0$. It proves the lemma.

\endproof

\subsection{Unique solution to the averaging process in the limit case $\Lipschitza\downarrow 1$}\label{app:existence_ODE_limit:averaging}

\proof{Proof of Lemma~\ref{lemma:existence_ODE_limit:averaging}.}
Define 
\begin{equation}\label{eqn:app:existence_ODE_limit:28}
    \bar{\bm{\xi}} \coloneqq \mathbb{E} \bm{\xi}_{\tau}^{1},
\end{equation}
where, since $\bm{\xi}_{\tau}^{1} $ is identically distributed for all $\tau\geq 0$, the right hand side remains a constant for all $\tau \geq 0$.
Also, for $\kappa,\kappa'\in\mathscr{K}$, the expectation $\mathbb{E}\xi_{\tau}(\kappa,\kappa',n)$ is the same for all $n\in[N^0_{i^{\kappa}}]$, for which
\begin{equation}
    \bar{\xi}_{\kappa,\kappa'} = \mathbb{E}\xi_{\tau}(\kappa,\kappa',n) \leq  \scrp_{t^{\kappa}}(\zeta^{\kappa},\zeta^{\kappa'})\in[0,1],
\end{equation}
with $\scrp_{t^{\kappa}}(\zeta^{\kappa},\zeta^{\kappa'})$ defined in \eqref{eqn:app:convergence_Z:2}.
From \eqref{eqn:app:existence_ODE_limit:25}, we rewrite 
\begin{equation}\label{eqn:app:existence_ODE_limit:29}
    \bar{\beta}(\bm{x}) = \lim_{\Lipschitza\downarrow 0} \tilde{\mathcal{Q}}^1(\bm{x})\bar{\bm{\xi}},
\end{equation}
where, taking expectation on both sides of \eqref{eqn:app:existence_ODE_limit:6} (with $h=1$), for $\kappa\in\mathscr{K}$, 
\begin{multline}\label{eqn:app:existence_ODE_limit:30}
      \bar{\beta}_{\kappa}(\bm{x}) = \Bigl(\tilde{\Theta}^1(\bm{x})\bar{\bm{\xi}}\Bigr)_{\kappa}
      = \sum_{\kappa'\in\mathscr{K}}\Biggl(
\indicator\bigl\{x(t^{\kappa'}-1)=0,x^{\mu}(t^{\kappa'})=0\bigr\}\lceil x_{\kappa'}\rceil\bar{\xi}_{\kappa',\kappa} \\
+ \indicator\bigl\{x(t^{\kappa'}-1)=0\bigr\}\bigl(1-\indicator\{x^{\mu}(t^{\kappa'})=0\}\bigr)\indicator\{t^{\kappa'}=t^{\kappa},i^{\kappa'}=i^{\kappa},s^{\kappa'}=s^{\kappa},\action^{\kappa'}=\dummyAction\}\\
\times\sum_{\begin{subarray}~\kappa''\in\mathscr{K}:\\i^{\kappa''}=i^{\kappa'},\\ s^{\kappa''}=s^{\kappa'},\\t^{\kappa''}=t^{\kappa'}\end{subarray}}x_{\kappa''}\scra^{\orule}_{\zeta^{\kappa}}(\frac{\bm{x}}{\sum_{i\in[I]}N_i^0},t^{\kappa})\bar{\xi}(0)\\
 - \indicator\bigl\{x(t^{\kappa}-1)=0,x^{\mu}(t^{\kappa})=0\bigr\}\lceil x_{\kappa}\rceil\bar{\xi}_{\kappa,\kappa'}\\
 - \indicator\bigl\{x(t^{\kappa}-1)=0\bigr\}\bigl(1-\indicator\{x^{\mu}(t^{\kappa})=0\}\bigr)\indicator\{t^{\kappa'}=t^{\kappa},i^{\kappa'}=i^{\kappa},s^{\kappa'}=s^{\kappa},\action^{\kappa}=\dummyAction\}\\
 \times\sum_{\begin{subarray}~\kappa''\in\mathscr{K}:\\i^{\kappa''}=i^{\kappa},\\ s^{\kappa''}=s^{\kappa},\\t^{\kappa''}=t^{\kappa}\end{subarray}}x_{\kappa''}\scra^{\orule}_{\zeta^{\kappa'}}(\frac{\bm{x}}{\sum_{i\in[I]}N_i^0},t^{\kappa'})\bar{\xi}(0)
 \Biggr).
\end{multline}

Similar to the proof of Lemma~\ref{lemma:existence_ODE_limit:h=1}, 
we show that $\bar{\beta}(\bm{X}^{\sigma}_{\tau})$ is piece-wise constant in $\tau$. 
Unlike the case for Lemma~\ref{lemma:existence_ODE_limit:h=1} or \ref{lemma:existence_ODE_limit:h>1}, in this averaging case with differential equation $\bar{\beta}(\bm{x})$, in general, there is no clear silent/non-silent led by $\bar{\bm{\xi}}$, because $\bar{\xi}_{\kappa,\kappa'}$ and $\bar{\xi}(0)$ may be non-zero together.

However, there are still two periods characterized by $\tilde{\Theta}^1(\bm{x})$.
For any $\bm{x}$,
if there exists $\kappa\in\mathscr{K}$ with $x^{\mu}(t^{\kappa})=0$, then, for any $\kappa'\in\mathscr{K}$,
$$\bigl(1-\indicator\{x^{\mu}(t^{\kappa'})\}\bigr)\indicator\{t^{\kappa'}=t^{\kappa},i^{\kappa'}=i^{\kappa},s^{\kappa'}=s^{\kappa},\action^{\kappa'}=\dummyAction\} = 0.$$
Hence, in the summand over $\kappa'\in\mathscr{K}$ at the right hand side of \eqref{eqn:app:existence_ODE_limit:30}, if the first or third terms is positive for any $\kappa'\in\mathscr{K}$, then the second and the fourth terms must be zero for all $\kappa'\in\mathscr{K}$.

Given $\bm{x}$ and $t$,
we refer to the case where $x^{\mu}(t)=0$ as the \emph{transition period} for $(\bm{x},t)$; while, if $x^{\mu}(t)>0$, it is the \emph{control period}.

Consider $\tau_0\geq 0$ and $t\in[T]$ for which, with $\bm{x}=\bar{\bm{x}}^{\sigma}_{\tau_0}$, $x(t)>0$ and $x(t-1)=0$.
If, for a $\kappa\in\mathscr{K}$ with $t^{\kappa}=t$, the process  $\{\bar{\bm{x}}^{\sigma}_{\tau}, \tau\geq 0\}$ is in the control period for $(\bar{\bm{x}}^{\sigma}_{\tau_0},t^{\kappa})$, then the first and third terms  at the right hand side of \eqref{eqn:app:existence_ODE_limit:30} are zero by plugging in $\bm{x}=\bar{\bm{x}}^{\sigma}_{\tau_0}$.
We only need to look at the second and fourth terms, which, similar to the discussions after \eqref{eqn:expand_theta}, are constant as $\tau$ increases until the nearest future $\tau_1>\tau_0$ such that $x^{\mu}(t^{\kappa})=0$ by plugging in $\bm{x}=\bar{\bm{x}}^{\sigma}_{\tau_1}$. 
The process $\{\bar{\bm{x}}^{\sigma}_{\tau}, \tau\geq 0\}$ turns into the transition period at $\tau=\tau_1$.

On the other hand, for $\tau_1\geq 0$, the process  $\{\bar{\bm{x}}^{\sigma}_{\tau}, \tau\geq 0\}$ is in the transition period for $(\bar{\bm{x}}^{\sigma}_{\tau},t^{\kappa})$ with $\tau=\tau_1$. At the right hand side of \eqref{eqn:app:existence_ODE_limit:30}, the second and fourth terms are zero. Such $\bar{\beta}(\bar{\bm{x}}^{\sigma}_{\tau})$ is piece-wise constant until $\tau_2>\tau_1$ such that $x(t)=0$ (plugging in $\bm{x}=\bar{\bm{x}}^{\sigma}_{\tau_2}$).
Because, in the transition period, all $\kappa\in\mathscr{K}$ with $\action^{\kappa}=\dummyAction$ will have zero derivative $\bar{\beta}_{\kappa}(\bar{\bm{x}}^{\sigma}_{\tau})=0$. 
The transition period will cause decrements of $x(t^{\kappa})$ until it turns to zero at $\tau=\tau_2$.

During the period $\tau\in[\tau_0,\tau_2)$, since $\bar{\beta}(\bar{\bm{x}}^{\sigma}_{\tau})$ is piece-wise constant in $\tau$, there exists a unique $\bar{\bm{x}}^{\sigma}_{\tau}$ that solves \eqref{eqn:app:existence_ODE_limit:27} with given $\bar{\bm{x}}^{\sigma}_{\tau_0} = \bm{x}$.
Based on the boundness of $\bar{\beta}$, $\bar{\bm{x}}^{\sigma}_{\tau_2} = \lim_{\tau\to\tau_2}\bar{\bm{x}}^{\sigma}_{\tau}$.
We can iteratively consider such $\tau=\tau_2$ as another initial point to continue the process. It provess the lemma.

\endproof

\section{Continuity in $\Lipschitza\in[0,1)$}
\label{app:continuity_trajectory_a}

\proof{Proof of Lemma~\ref{lemma:continuity_trajectory_a:stochastic}.}

In the discussion of this section, we consider the case with $U\geq 1$ all the time so that it does not affect the discussion on silent period where $\norm{\bm{\xi}^1_{\tau}} = 1$. 

We start with showing \eqref{eqn:continuity_trajectory_a:stochastic:1}.
Let $\bm{X}^{\sigma,\Lipschitza,U}_{\tau}(\bm{x}) $ represent the solution to \eqref{eqn:app:convergence_Z:6} with initial condition $\bm{X}^{\sigma,\Lipschitza,U}_0(\bm{x}) = \bm{x}$.  From Lemma~\ref{lemma:existence_ODE}, $\bm{X}^{\sigma,\Lipschitza,U}_{\tau}(\bm{x}) $ uniquely exists.

For $\bm{x}_1,\bm{x}_2\in\mathbb{R}_0^{|\mathscr{K}|}$ and $a_1,a_2\in [0,1)$, 
\begin{multline}\label{eqn:lemma:uniqueness:a>1:1}
    \norm{\bm{X}^{\sigma,a_1,U}_{\tau}(\bm{x}_1) - \bm{X}^{\sigma,a_2,U}_{\tau}(\bm{x}_2)} \leq \norm{\bm{x}_1-\bm{x}_2}\\ 
    + \norm{\int_0^{\tau} b^{a_1,U}\bigl(\bm{X}^{\sigma,a_1,U}_{\tau'}(\bm{x}_1), \bm{\xi}^{1}_{\tau'/\sigma}\bigr)-b^{a_2,U}\bigl(\bm{X}^{\sigma,a_2,U}_{\tau'}(\bm{x}_2), \bm{\xi}^{1}_{\tau'/\sigma}\bigr)d\tau'}\\
    \leq \norm{\bm{x}_1-\bm{x}_2} + \int_0^{\tau} \norm{\Bigl(\tilde{\mathcal{Q}}^{1,a_1}\bigl(\bm{X}^{\sigma,a_1,U}_{\tau'}(\bm{x}_1)\bigr)-\tilde{\mathcal{Q}}^{1,a_2}\bigl(\bm{X}^{\sigma,a_2,U}_{\tau'}(\bm{x}_2)\bigr)\Bigr)\bm{\xi}^{1}_{\tau'/\sigma}}d\tau'\\
     \leq \norm{\bm{x}_1-\bm{x}_2} + \int_0^{\tau} \norm{\tilde{\mathcal{Q}}^{1,a_1}\bigl(\bm{X}^{\sigma,a_1,U}_{\tau'}(\bm{x}_1)\bigr)-\tilde{\mathcal{Q}}^{1,a_2}\bigl(\bm{X}^{\sigma,a_2,U}_{\tau'}(\bm{x}_2)\bigr)}\norm{\bm{\xi}^{1}_{\tau'/\sigma}}d\tau'.
\end{multline}
We equally divide $[0,\tau)$ into $N$ periods $[0,\Delta)$, $[\Delta, 2\Delta)$, $\ldots$, $[(N-1)\Delta, N\Delta)$, where $\Delta = \frac{\tau}{N}$.
For $n=1,2,\ldots,N$, let 
\[
\overline{q}^{\Lipschitza}_n \coloneqq \sup_{\tau'\in [(n-1)\Delta,n\Delta)}\norm{\tilde{\mathcal{Q}}^{1,\Lipschitza}\bigl(\bm{X}^{\sigma,a_1,U}_{\tau'}(\bm{x}_1\bigr)-\tilde{\mathcal{Q}}^{1,\Lipschitza}\bigl(\bm{X}^{\sigma,a_2,U}_{\tau'}(\bm{x}_2)\bigr)},
\]
for which
\begin{multline}\label{eqn:lemma:uniqueness:a>1:6}
    \int_0^{\tau} \norm{\tilde{\mathcal{Q}}^{1,a_1}\bigl(\bm{X}^{\sigma,a_1,U}_{\tau'}(\bm{x}_1)\bigr)-\tilde{\mathcal{Q}}^{1,a_2}\bigl(\bm{X}^{\sigma,a_2,U}_{\tau'}(\bm{x}_2)\bigr)}\norm{\bm{\xi}^{1}_{\tau'/\sigma}}d\tau' \\
    \leq \sum_{n=1}^N \Bigl(\overline{q}^{a_1}_n+K_{\Lipschitza}\norm{a_1-a_2}\Bigr)\int_{(n-1)\Delta}^{n\Delta} \norm{\bm{\xi}^{1}_{\tau'/\sigma}}d\tau',
\end{multline}
with $K_{\Lipschitza}<\infty$ the Lipschitz constant for $\tilde{\mathcal{Q}}^{1,\Lipschitza}(\bm{x})$ in $\Lipschitza \in[0,1)$. The existence of such $K_{\Lipschitza}<\infty$, independent to $\bm{x}$, is ensured by \eqref{eqn:derivative_f_to_a} and \eqref{eqn:dirac-delta:6}.
Based on the boundness of $\bm{\xi}^1_{\tau}$, there exists $K_{\xi}<\infty$, independent to $N$, such that
\begin{equation}\label{eqn:lemma:uniqueness:a>1:7}
    \int_{(n-1)\Delta}^{n\Delta} \norm{\bm{\xi}^{1}_{\tau'/\sigma}}d\tau' \leq 
    \sigma K_{\xi}\Delta.
\end{equation}

For $n\in[N]$, from \eqref{eqn:lemma:uniqueness:a>1:1}, \eqref{eqn:lemma:uniqueness:a>1:6}, and \eqref{eqn:lemma:uniqueness:a>1:7}, we have
\begin{multline}\label{eqn:lemma:uniqueness:a>1:8}
 \norm{\bm{X}^{\sigma,a_1,U}_{n\Delta}(\bm{x}_1) - \bm{X}^{\sigma,a_1,U}_{n\Delta}(\bm{x}_2)} 
 \leq \norm{\bm{X}^{\sigma,a_2,U}_{(n-1)\Delta}(\bm{x}_1) - \bm{X}^{\sigma,a_2,U}_{(n-1)\Delta}(\bm{x}_2)} \\
 + \Bigl(\bar{q}^{a_1}_n+K_a\norm{a_1-a_2}\Bigr)\int_{(n-1)\Delta}^{n\Delta} \norm{\bm{\xi}^{1}_{\tau'/\sigma}d\tau'} \\
 \leq \norm{\bm{X}^{\sigma,a_1,U}_{(n-1)\Delta}(\bm{x}_1) - \bm{X}^{\sigma,a_2,U}_{(n-1)\Delta}(\bm{x}_2)} + \sigma K_{\xi} \Delta \Bigl(\bar{q}^{a_1}_n+K_a\norm{a_1-a_2}\Bigr).
\end{multline}

For each $n\in[N]$, we define $K^{\Lipschitza}_Q(n)$ as the minimum $K\geq 0$ such that
\begin{equation}
    \norm{\tilde{\mathcal{Q}}^{1,\Lipschitza}\bigl(\bm{X}^{\sigma,a_1,U}_{\tau}(\bm{x}_1)\bigr)-\tilde{\mathcal{Q}}^{1,\Lipschitza}\bigl(\bm{X}^{\sigma,a_2,U}_{\tau}(\bm{x}_2)\bigr)}\leq K \norm{\bm{X}^{\sigma,a_1,U}_{\tau}(\bm{x}_1)-\bm{X}^{\sigma,a_2,U}_{\tau}(\bm{x}_2)},
\end{equation}
for all $\tau\in[(n-1)\Delta,n\Delta)$.
For $\Lipschitza\in(0,1)$, based on the Lipschitz continuity of $\tilde{Q}^{1,\Lipschitza}(\bm{x})$ (led by the Lipschitz continuity of $b^{\Lipschitza,U}(\bm{x},\bm{\xi})=\tilde{\mathcal{Q}}^{1,\Lipschitza}(\bm{x})\bm{\xi}$ in $\bm{x}$), $K^{\Lipschitza}_Q(n) < \infty$.
Note that $K^{\Lipschitza}_Q(n)$ may be infinity in the limit case $\Lipschitza\downarrow 0$.

From \eqref{eqn:lemma:uniqueness:a>1:8}, 
\begin{multline}\label{eqn:lemma:uniqueness:a>1:9}
 \norm{\bm{X}^{\sigma,a_1,U}_{n\Delta}(\bm{x}_1) - \bm{X}^{\sigma,a_2,U}_{n\Delta}(\bm{x}_2)} \\
 \leq \norm{\bm{X}^{\sigma,a_1,U}_{(n-1)\Delta}(\bm{x}_1) - \bm{X}^{\sigma,a_2,U}_{(n-1)\Delta}(\bm{x}_2)} \\
 + \sigma K_{\xi} \Delta \Bigl(K^{a_1}_Q (n)\sup_{\tau\in[(n-1)\Delta,n\Delta)}\norm{\bm{X}^{\sigma,a_1,U}_{\tau}(\bm{x}_1) - \bm{X}^{\sigma,a_2,U}_{\tau}(\bm{x}_2)} + K_{\Lipschitza}\norm{a_1-a_2}\Bigr)\\
 \leq \norm{\bm{X}^{\sigma,a_1,U}_{(n-1)\Delta}(\bm{x}_1) - \bm{X}^{\sigma,a_2,U}_{(n-1)\Delta}(\bm{x}_2)} \\
 + \sigma K_{\xi} \Delta \Bigl(K^{a_1}_Q(n) \norm{\bm{X}^{\sigma,a_1,U}_{(n-1)\Delta}(\bm{x}_1) - \bm{X}^{\sigma,a_2,U}_{(n-1)\Delta}(\bm{x}_2)} + K^{\Lipschitza}_Q(n)\bar{X}_N + K_{\Lipschitza}\norm{a_1-a_2}\Bigr),
\end{multline}
where 
\[
\bar{X}_N\coloneqq \max_{n\in[N]}\biggl(\sup_{\tau\in[(n-1)\Delta,n\Delta)}\norm{\bm{X}^{\sigma,a_1,U}_{\tau}(\bm{x}_1) - \bm{X}^{\sigma,a_2,U}_{\tau}(\bm{x}_2)} -\norm{\bm{X}^{\sigma,a_1,U}_{(n-1)\Delta}(\bm{x}_1) - \bm{X}^{\sigma,a_2,U}_{(n-1)\Delta}(\bm{x}_2)}\biggr).
\]
From the continuity of $\bm{X}^{\sigma,\Lipschitza,U}_{\tau}(\bm{x})$ in $\tau$, $\lim_{N\rightarrow \infty} \bar{X}_N = 0$.

For $K\in\mathbb{R}_0\cup\{\infty\}$ and $n\in[N]$, we define a set 
\begin{equation}
    \mathscr{N}^{\Lipschitza,N}_n(K)\coloneqq \Bigl\{n' = n,n+1,\ldots,N~\Bigl|~K^{\Lipschitza}_Q(n') \geq K\Bigr\},
\end{equation}
$\mathscr{N}^{\Lipschitza,N}_N \equiv \emptyset$, and $\mathscr{N}^{\Lipschitza,N}_n(\infty)\coloneqq \lim_{K\rightarrow \infty}\mathscr{N}^{\Lipschitza,N}_n(K)$.

For any $\tau > 0$, $N\in\mathbb{N}_+$, and $\Delta = \tau/N$, based on \eqref{eqn:lemma:uniqueness:a>1:9},
\begin{multline}\label{eqn:lemma:uniqueness:a>1:11}
\norm{\bm{X}^{\sigma,a_1,U}_{\tau}(\bm{x}_1) - \bm{X}^{\sigma,a_2,U}_{\tau}(\bm{x}_2)} 
 = \norm{\bm{X}^{\sigma,a_1,U}_{N\Delta}(\bm{x}_1) - \bm{X}^{\sigma,a_2,U}_{N\Delta}(\bm{x}_2)} \\
    \leq \prod_{n\in[N]}(1 + \sigma K_{\xi}K^{a_1}_Q(n)\Delta) \norm{\bm{x}_1-\bm{x}_2} \\
    +\sum_{n=0}^{N-1}\prod_{n'=1}^n\Bigl(1 + \sigma K_{\xi}K^{a_1}_Q(N-n')\Delta\Bigr) \sigma K_{\xi}\Delta\Bigl(K^{a_1}_Q(N) \bar{X}_N +K_\Lipschitza\norm{a_1-a_2}\Bigr).
\end{multline}
We plug $\bm{x}_1=\bm{x}_2 = \bm{x}_0$ in \eqref{eqn:lemma:uniqueness:a>1:11}, for any given $K\in\mathbb{R}_0\cup\{\infty\}$,
\begin{multline}\label{eqn:lemma:uniqueness:a>1:12}
    \norm{\bm{X}^{\sigma,a_1,U}_{\tau}(\bm{x}_0) - \bm{X}^{\sigma,a_2,U}_{\tau}(\bm{x}_0)}\\ \leq
    \sum_{n=0}^{N-1}\Bigl(1 + \sigma K_{\xi}\underline{K}^{a_1,N}_Q(K)\Delta\Bigr)^{n-\bigl\lvert \mathscr{N}^{a_1,N}_{N-n}(K)\bigr\rvert}\Bigl(1 + \sigma K_{\xi}\Delta\max_{n'=N-n}^N K^{a_1}_Q(n')\Bigr)^{\bigl\lvert \mathscr{N}^{a_1,N}_{N-n}(K)\bigr\rvert} \\
    \times\sigma K_{\xi}\Delta\Bigl(K^{a_1}_Q(N) \bar{X}_N +K_\Lipschitza\norm{a_1-a_2}\Bigr)\\
     \leq  
 \sigma K_{\xi}\Delta\Bigl(K^{a_1}_Q(N) \bar{X}_N +K_\Lipschitza\norm{a_1-a_2}\Bigr)\sum_{n=0}^{N-1}\Bigl(1 + \sigma K_{\xi}\underline{K}^{a_1,N}_Q(K)\Delta\Bigr)^n\Bigl(1 + \sigma K_{\xi}\Delta\max_{n'=1}^N K^{a_1}_Q(n')\Bigr)^{\bigl\lvert \mathscr{N}^{a_1,N}_1(K)\bigr\rvert} \\
     =\Bigl(1 + \sigma K_{\xi}\Delta\max_{n'=1}^N K^{a_1}_Q(n')\Bigr)^{\bigl\lvert \mathscr{N}^{a_1,N}_1(K)\bigr\rvert}  \sigma K_{\xi}\Delta\Bigl(K^{a_1}_Q(N) \bar{X}_N +K_\Lipschitza\norm{a_1-a_2}\Bigr) \frac{\Bigl(1+\sigma K_{\xi}\underline{K}^{a_1,N}_Q(K)\Delta\Bigr)^N-1}{\sigma K_{\xi} \underline{K}^{a_1,N}_Q(K)\Delta}\\
     =\Bigl(1 + \sigma K_{\xi}\Delta\max_{n'=1}^N K^{a_1}_Q(n')\Bigr)^{\bigl\lvert \mathscr{N}^{a_1,N}_1(K)\bigr\rvert}  \Bigl(K^{a_1}_Q(N) \bar{X}_N +K_\Lipschitza\norm{a_1-a_2}\Bigr) \frac{\Bigl(1+\sigma K_{\xi}\underline{K}^{a_1,N}_Q(K)\Delta\Bigr)^N-1}{\underline{K}^{a_1,N}_Q(K)},
\end{multline}
where 
\begin{equation}
    \underline{K}^{a_1,N}_Q(K)\coloneqq \max_{n\in[N]: n\notin \mathscr{N}^{a_1,N}_1(K)}K^{a_1}_Q(n).
\end{equation}

Let $\underline{K}^{a_1}_Q(K)\coloneqq \lim_{N\rightarrow \infty} \underline{K}^{a_1,N}_Q(K)$.
Obviously, $\underline{K}^{a_1}_Q(K)< K$.

Plugging $\Delta = \frac{\tau}{N}$ in \eqref{eqn:lemma:uniqueness:a>1:12} and taking $N\rightarrow \infty$, we obtain 
\begin{multline}\label{eqn:lemma:uniqueness:a>1:13}
  \lim_{a_2\to a_1}\lim_{a_1\downarrow 0}\frac{\norm{\bm{X}^{\sigma,a_1,U}_{\tau}(\bm{x}_0) - \bm{X}^{\sigma,a_2,U}_{\tau}(\bm{x}_0)}}{\norm{a_1-a_2}}\\ \leq \lim_{a_2\to a_1}\lim_{a_1\downarrow 0}\frac{1}{\norm{a_1-a_2}}
  \lim_{N\rightarrow \infty}\Bigl(1 + \sigma K_{\xi}\max_{n'=1}^N K^{a_1}_Q(n')\tau\frac{1}{N}\Bigr)^{\bigl\lvert \mathscr{N}^{a_1,N}_1(K)\bigr\rvert}  \lim_{N\rightarrow \infty}\Bigl(K^{a_1}_Q(N) \bar{X}_N +K_\Lipschitza\norm{a_1-a_2}\Bigr) \\
  \times\frac{\Bigl(1+\sigma K_{\xi}\underline{K}^{a_1,N}_Q(K)\tau\frac{1}{N}\Bigr)^N-1}{\underline{K}^{a_1,N}_Q(K)}\\
  = \lim_{a_2\to a_1}\lim_{a_1\downarrow 0} \lim_{N\rightarrow \infty}\Bigl(1 + \sigma K_{\xi}\max_{n'=1}^N K^{a_1}_Q(n')\tau\frac{1}{N}\Bigr)^{\bigl\lvert \mathscr{N}^{a_1,N}_1(K)\bigr\rvert} K_{\Lipschitza}\frac{\exp\{\sigma K_{\xi} \underline{K}^{a_1}_Q(K)\tau\}-1}{\underline{K}^{a_1}_Q(K)}.
  \\  \leq \lim_{a_2\to a_1}\lim_{a_1\downarrow 0}\exp\Bigl\{\sigma\tau K_{\xi}\lim_{N\to\infty}\max_{n'=1}^N K^{a_1}_Q(n')\frac{1}{N}\Bigl\lvert \mathscr{N}^{a_1,N}_1(K)\Bigr\rvert\Bigr\} K_{\Lipschitza} \frac{\exp\{\sigma \tau K_{\xi}\underline{K}^{a_1}_Q(K)\}-1}{\underline{K}^{a_1}_Q(K)} ,
\end{multline}
where $\lim_{a_1\downarrow 0} \bm{X}^{\sigma,a_1,U}(\bm{x}_0)$ indicates the unique solution to $\dot{\bm{X}}^{\sigma,0,U}_{\tau}(\bm{x}_0) = \lim_{a_1\downarrow 0} b^{a_1,U}(\bm{X}^{\sigma,0,U}(\bm{x}_0),\bm{\xi}^1_{\tau/\sigma})$ with initial condition $\bm{x}_0$.

Based on the definitions in \eqref{eqn:app:convergence_Z:3} and \eqref{eqn:app:convergence_Z:4}, when $\Lipschitza\downarrow 0$, $\tilde{\mathcal{Q}}^{1,\Lipschitza}(\bm{x}) $ becomes piece-wise constant in $\bm{x}$.
Hence, for any $\bar{T}<\infty$ and $a\downarrow 0$, there are only finitely many discontinuous points $\tau\in[0,\bar{T}]$ for $\bm{X}^{\sigma,0,U}_{\tau}$, which are also discontinuous points of $\tilde{\mathcal{Q}}^{1,0}(\bm{X}^{\sigma,0,U}_{\tau}(\bm{x}_0))$ in $\tau$.
Apart from these discontinuous points, $\tilde{\mathcal{Q}}^{1,0}(\bm{X}^{\sigma,0,U}_{\tau}(\bm{x}_0)$ is piece-wise constant.
That is, for any $K\in \mathbb{R}_0\cup\{\infty\}$, $\lim_{a_1\downarrow 0}\underline{K}^{a_1}_Q(K) =0$.

For $a_1\in[0,1)$, we define
\begin{multline}
\overline{K}^{a_1}_Q \coloneqq \sup_{\tau'\in[0,\tau)}\lim_{\delta \downarrow 0}\max\Biggl\{\frac{\norm{\tilde{\mathcal{Q}}^{1,a_1}(\bm{X}^{\sigma,a_1,U}_{\tau'}(\bm{x}_0)+\delta)-\tilde{\mathcal{Q}}^{1,a_1}(\bm{X}^{\sigma,a_1,U}_{\tau'}(\bm{x}_0))}}{\norm{\delta}},\\
\frac{\norm{\tilde{\mathcal{Q}}^{1,a_1}(\bm{X}^{\sigma,a_1,U}_{\tau'}(\bm{x}_0)-\delta)-\tilde{\mathcal{Q}}^{1,a_1}(\bm{X}^{\sigma,a_1,U}_{\tau'}(\bm{x}_0))}}{\norm{\delta}}\Biggr\}
\eqqcolon \sup_{\tau'\in[0,\tau]} q^{a_1}(\tau'),
\end{multline}
 and
\begin{equation}
\alpha^{a_1}(K) \coloneqq \lim_{N\to\infty}\frac{1}{N}\Bigl\lvert \mathscr{N}^{a_1,N}_1(K)\Bigr\rvert\leq \frac{1}{\tau}\int_0^{\tau}\indicator\bigl\{q^{a_1}(\tau')\geq K\bigr\}d\tau'.
\end{equation}
For any $a_1\in[0,1)$,
\begin{equation}
    \overline{K}^{a_1}_Q \geq \lim_{a_2\to a_1}\lim_{N\to\infty}\max_{n'=1}^N K^{a_1}_Q(n').
\end{equation}

We plug $K = \overline{K}^{a_1}_Q$ in \eqref{eqn:lemma:uniqueness:a>1:13},
\begin{multline}\label{eqn:lemma:uniqueness:a>1:14}
  \lim_{a_2\to a_1}\lim_{a_1\downarrow 0}\frac{\norm{\bm{X}^{\sigma,a_1,U}_{\tau}(\bm{x}_0) - \bm{X}^{\sigma,a_2,U}_{\tau}(\bm{x}_0)}}{\norm{a_1-a_2}} \leq  
    \lim_{a_2\to a_1}\exp\Bigl\{\sigma K_{\xi} \lim_{a_1\downarrow 0}\lim_{N\to\infty}\max_{n'=1}^N K^{a_1}_Q(n')\alpha^{a_1}(\overline{K}^{a_1}_Q)\tau\Bigr\}K_a \sigma \tau K_{\xi}\\
    \leq K_a \sigma \tau K_{\xi}\exp\Biggl\{\sigma K_{\xi}\int_0^{\tau} \lim_{a_2\to a_1}\lim_{N\to\infty}\max_{n'=1}^N K^{0}_Q(n') \indicator\Bigl\{q^{0}(\tau')=\overline{K}^{0}_Q\Bigr\} d\tau'\Biggr\}\\
    \leq K_a \sigma \tau K_{\xi}\exp\Biggl\{\sigma K_{\xi}\int_0^{\tau} \overline{K}^0_Q\indicator\Bigl\{q^{0}(\tau')=\overline{K}^{0}_Q\Bigr\} d\tau'\Biggr\}\\
    \leq K_a \sigma \tau K_{\xi}\exp\Biggl\{\sigma K_{\xi} \Bigl(\int_{\begin{subarray}~\tau'\in [0,\tau)\\ \tau'\notin\inter{\scrR^+}\end{subarray}}K_Q\biggl(\norm{\frac{d^-}{d \tau'}\tilde{\mathcal{Q}}^{1,0}(\bm{X}^{\sigma,0,U}_{\tau'}(\bm{x}_0))} +\norm{\frac{d^+}{d \tau'}\tilde{\mathcal{Q}}^{1,0}(\bm{X}^{\sigma,0,U}_{\tau'}(\bm{x}_0))}\biggr) d\tau'\\
    +\int_{\tau'\in\inter{\scrR^+}}\norm{\bm{\scra}^{\orule}\bigl(\frac{\bm{X}^{\sigma,0,U}_{\tau'}}{\sum_{i\in[I]}N^0_i},t(\bm{X}^{\sigma,0,U}_{\tau'})\bigr)}d\tau'\Bigr)\Biggr\},
\end{multline}
where 
\begin{multline}
    \scrR^+\coloneqq \biggl\{\tau'\in[0,\tau) \biggl| \\\forall \kappa\in\mathscr{K}, \indicator\Bigl\{\sum_{\begin{subarray}~\kappa'\in\mathscr{K}:\\t^{\kappa'}<t^{\kappa}\end{subarray}}\lvert X^{\sigma,0,1}_{\tau',\kappa'}(\bm{x}_0)\rvert=0,
\sum_{\begin{subarray}~\kappa'\in\mathscr{K}:\\t^{\kappa'}=t^{\kappa},\\a^{\kappa'} = \dummyAction\end{subarray}}\lvert X^{\sigma,0,1}_{\tau',\kappa'}(\bm{x}_0)\rvert=0\Bigr\} \sum_{\kappa'\in\mathscr{K}}\sum_{n\in\bigl\lceil\lvert X^{\sigma,0,1}_{\tau',\kappa}(\bm{x}_0)\rvert \bigr\rceil} \lvert\xi^1_{\tau'}(\kappa,\kappa',n) \rvert=0\biggr\},
\end{multline}
$t(\bm{x})\coloneqq \Bigl\{\max \bigl\{t\in[T]_0 \bigl| \sum_{\kappa\in\mathscr{K}:t^{\kappa}<t}|x_{\kappa}| = 0\bigr\}\cup \{0\}\Bigr\}$, 
$K_Q\coloneqq \sup_{\tau'\in[0,\tau),\tau'\notin\inter{\scrR^+}}\frac{1}{\norm{b^1(\bm{X}^{\sigma,0,U}_{\tau'}(\bm{x}_0),\bm{\xi}^1_{\tau'/\sigma})}}<\infty$, $d^-/d \tau'$ and $d^+/d\tau'$ are the left and right, respectively, derivatives in $\tau'$, and $\inter{\scrR^+}$ is the interior of $\scrR^+$.

Since $\tilde{\mathcal{Q}}^{1,0}(\bm{X}^{\sigma,0,U}_{\tau'})=\lim_{\Lipschitza\downarrow 0}\tilde{\mathcal{Q}}^{1,\Lipschitza}(\bm{X}^{\sigma,\Lipschitza,U}_{\tau'})$ has only finitely many discontinuous points in $\tau'\in[0,\tau)$ and is always bounded, $\int_0^{\tau}\norm{\frac{d^-}{d \tau'}\tilde{\mathcal{Q}}^{1,0}(\bm{X}^{\sigma,0,U}_{\tau'})} + \norm{\frac{d^+}{d \tau'}\tilde{\mathcal{Q}}^{1,0}(\bm{X}^{\sigma,0,U}_{\tau'})}d\tau'$ is bounded for any $\tau\in[0,\bar{T}]$. 
There exists $K<\infty$ such that, based on \eqref{eqn:lemma:uniqueness:a>1:14},
we obtain
\begin{equation}\label{eqn:lemma:uniqueness:a>1:15}
 \lim_{a_2\downarrow 0}\frac{\norm{\bm{X}^{\sigma,0,U}_{\tau}(\bm{x}_0) - \bm{X}^{\sigma,a_2,U}_{\tau}(\bm{x}_0)}}{\norm{a_2}} \leq   
 K_a \sigma \tau K_{\xi}\exp\Bigl\{\sigma K_{\xi}(K_Q K+\lvert \scrR^+\rvert)\Bigr\} < \infty,
\end{equation}
where $K_a,K_{\xi}$, $K_Q$, and $K$ are independent to $a$ and $U$.
It leads to the right continuity of $\bm{X}^{\sigma,\Lipschitza,U}_{\tau}$, for any $\tau\in[0,\bar{T}]$, in $\Lipschitza = 0$.

Based on the definition of $b^{\Lipschitza,U}(\bm{x},\bm{\xi}) = \tilde{\mathcal{Q}}^{1,\Lipschitza}(\bm{x})\min\{U,\bm{\xi}\}$, for any given $\bm{\xi}^{1}_{\tau}$ ($\tau\geq 0$) and $\Lipschitza\in[0,1)$, there exists $U_0<\infty$ such that, for all $U\geq U_0$, $\bm{X}^{\sigma,\Lipschitza,U}_{\tau} = \bm{X}^{\sigma,\Lipschitza,U_0}_{\tau}$ for all $\tau\in[0,\bar{T}]$. 
It shows the left continuity of $\bm{X}^{\sigma,\Lipschitza,U}_{\tau}$ in $U\rightarrow \infty$, leading to \eqref{eqn:continuity_trajectory_a:stochastic:1}.

It remains to prove \eqref{eqn:continuity_trajectory_a:averaging:1}.
Along similar lines, for $a_1,a_2\in[0,1)$ and any $N\in\mathbb{N}_+\cup\{\infty\}$, there exist $K_\Lipschitza,K_{\xi},K_Q,K<\infty$ such that
\begin{multline}\label{eqn:continuity_trajectory_a:averaging:2}
\lim_{a_2\to a_1}\lim_{a_1\downarrow 0}\frac{\norm{\bar{\bm{x}}^{a_1,U}_{\tau} - \bar{\bm{x}}^{a_2,U}_{\tau}} }{\norm{a_1-a_2}}\\
 \leq K_{\Lipschitza} \sigma \tau K_{\xi}\exp\Biggl\{\sigma K_{\xi} \Bigl(K_Q\int_{\begin{subarray}~\tau'\in [0,\tau)\\ \tau'\notin\inter{\scrR^+}\end{subarray}}\bigl(\norm{\frac{d^-}{d \tau'}\tilde{\mathcal{Q}}^{1,0}(\bar{\bm{x}}^{0,U}_{\tau'})}\\
 +\norm{\frac{d^+}{d \tau'}\tilde{\mathcal{Q}}^{1,0}(\bar{\bm{x}}^{0,U}_{\tau'})} \bigr)d\tau'+\int_{\tau'\in\scrR^+}\norm{\scra^{\orule}\bigl(\frac{\bar{\bm{x}}^{0,1}_{\tau'}}{\sum_{i\in[I]}N^0_i},t(\bar{\bm{x}}^{0,1}_{\tau'})\bigr)} d\tau'\Bigr)\Biggr\}\\
 \leq K_{\Lipschitza} \sigma \tau K_{\xi}\exp\Biggl\{\sigma K_{\xi}(K_Q K+\lvert \scrR^+\rvert)\Biggr\}.
\end{multline}

It proves the lemma.

\endproof

\section{Proof of Lemma~\ref{lemma:sim}}\label{app:lemma:sim}
\proof{Proof of Lemma~\ref{lemma:sim}.}
For any $\orule=(\phi^h: h\in\mathbb{N}_+)\in\Psi^0 $ (\partialref{cond:weak_stab}{Lipschitz–limit regularity}),
let $\psi\in\lPhih$ be such that, for all $t\in [T]_0$ and $\zeta\in\calJ$, 
\begin{equation}\label{eqn:define_psi}
    \bbP\Bigl\{\action^{\psi,h}_{i^{\zeta},n}(t)=\action^{\zeta} \Bigl| s^{\psi,h}_{i^{\zeta},n}(t) = s^{\zeta},\bm{Y}^{\psi,h}(t)=\bm{y}\Bigr\}  = \alpha^{\orule}_{\zeta}(\bm{y},t),
\end{equation}
where recall $\alpha^{\orule}_{\zeta}$ is defined in \eqref{eqn:define:alpha}.

For any $t\in[T]$, and $\kappa\in\calJ$ with $t^{\kappa}=t$, if 
$hX^h_{\tau^h(t),\kappa} \in \mathbb{N}_0$, then it is considered as the number of bandit processes in GSA triple $\zeta^{\kappa}$ at time $t^{\kappa}=t$. 
In particular, for $\tau_0 \geq \tau^h(t)$, if $\sum_{\kappa\in\mathscr{K}: t^{\kappa}=t,\action^{\kappa}=\dummyAction} hX^h_{\tau_0} = 0$ and and $\lim_{\tau\uparrow \tau_0}\sum_{\kappa\in\mathscr{K}:t^{\kappa}=t,\action^{\kappa}=\dummyAction} hX^h_{\tau} >0$, then we say the process $\bm{X}^h_{\tau}$ starts a transition period at $\tau=\tau_0$ for $t$.
In this case, observing the definition of $h\bm{X}^h_{\tau}$, for any $\kappa'\in\mathscr{K}$ with $t^{\kappa'}=t+1$, $i^{\kappa'}=i^{\kappa}$, and $\action^{\kappa'}=\dummyAction$, each unit reduced from $hX^h_{\tau,\kappa}$, representing a bandit process, will be added in $hX^h_{\tau,\kappa'}$, during $\tau\in[\tau_0,\tau^h(t+1))$, with probability $\scrp_{t}(\zeta^{\kappa},\zeta^{\kappa'}) = \bbP\Bigl\{s^{\orule,h}_{i^{\kappa},n}(t+1) = s^{\kappa'}\Bigl| s^{\orule,h}_{i^{\kappa},n}(t) = s^{\kappa},\action^{\orule,h}_{i^{\kappa},n}(t) = \action^{\kappa}\Bigr\}=\bbP\Bigl\{s^{\psi,h}_{i^{\kappa},n}(t+1) = s^{\kappa'}\Bigl| s^{\psi,h}_{i^{\kappa},n}(t) = s^{\kappa},\action^{\psi,h}_{i^{\kappa},n}(t) = \action^{\kappa}\Bigr\}=p_{i^{\kappa}}(s^{\kappa},\action^{\kappa},s^{\kappa'})$ (the definition of $\scrp$ in \eqref{eqn:app:convergence_Z:2}).
In this context, for any $\kappa\in\mathscr{K}$ with $t^{\kappa}=t+1$ and $\action^{\kappa} = \dummyAction$, $hX^h_{\tau^h(t+1),\kappa} \in \mathbb{N}_0$ and again  represents the number of bandit processes in GSA triple $\zeta^{\kappa}$ at time $t^{\kappa}=t+1$.

We define $\mathcal{P}\coloneqq \bigl[p(\zeta,i',s')\bigr]_{|\mathcal{J}|\times\sum_{i\in[I]}|\bS_i|}$ with
\begin{equation}
    p(\zeta,i',s') \coloneqq \begin{cases}
        p_{i^{\zeta}}(s^{\zeta},\action^{\zeta},s'),&\text{if }i' = i^{\zeta},\\
        0,&\text{otherwise}.
    \end{cases}
\end{equation}

Let $\bm{\mathcal{Z}}^h_{\tau}(t) \coloneqq (\mathcal{Z}^h_{\tau,\kappa}: \kappa\in\mathscr{K},t^{\kappa}=t)$ for all $\tau\in[0,\bar{T})$ and $t\in[T]_0$.
We have
\begin{multline}\label{eqn:lemma:sim:2}
    \Biggl\lVert \mathbb{E} \Bigl[\scru(\bm{\mathcal{Z}}^h_{\tau^h(t+1)},t+1)\Bigl|~\bm{\mathcal{Z}}^h_{\tau_0}\Bigr]
    -\mathbb{E} \Bigl[\bm{Y}^{\psi,h}(t+1) \Bigl|~\bm{Z}^{\psi,h}(t)\Bigr]\Biggr\rVert \leq \norm{\mathcal{P}^T \Bigl(\bm{\mathcal{Z}}^h_{\tau_0}(t)-\bm{Z}^{\psi,h}(t)\Bigr)}\\
    \leq \norm{\mathcal{P}}\norm{\bm{\mathcal{Z}}^h_{\tau_0}(t)-\bm{Z}^{\psi,h}(t)},
\end{multline}
where $\scru$ is defined in \eqref{eqn:define:scru}.
Also, from law of large numbers, for any $\delta>0$,
\begin{equation}\label{eqn:lemma:sim:3}
    \lim_{h\to\infty}\bbP\Biggl\{\Bigl\lVert \scru(\bm{\mathcal{Z}}^h_{\tau^h(t+1)},t+1)-
    \mathbb{E} \scru(\bm{\mathcal{Z}}^h_{\tau^h(t+1)},t+1)\Bigr\rVert > \delta\Bigl|~\bm{\mathcal{Z}}^h_{\tau_0}\Biggr\} = 0,
\end{equation}
and
\begin{equation}\label{eqn:lemma:sim:4}
\lim_{h\to\infty}\bbP\Biggl\{\Bigl\lVert \bm{Y}^{\psi,h}(t+1)  -
    \mathbb{E} \Bigl[\bm{Y}^{\psi,h}(t+1) \Bigr]\Bigr\rVert > \delta\Bigl|~\bm{Z}^{\psi,h}(t)\Biggr\} = 0.
\end{equation}
From \eqref{eqn:lemma:sim:2}-\eqref{eqn:lemma:sim:4}, for any $\delta>0$,
\begin{equation}\label{eqn:lemma:sim:5}
    \lim_{h\to\infty}\bbP\Biggl\{\Bigl\lVert \scru(\bm{\mathcal{Z}}^h_{\tau^h(t+1)},t+1)
    -\bm{Y}^{\psi,h}(t+1)\Bigr\rVert > \delta\Bigl|~\bm{\mathcal{Z}}^h_{\tau_0},\bm{Z}^{\psi,h}(t)\Biggr\}
    \leq \lim_{h\to\infty}\bbP\Biggl\{\norm{\mathcal{P}}\norm{\bm{\mathcal{Z}}^h_{\tau_0}(t)-\bm{Z}^{\psi,h}(t)} > \delta'\Biggr\},
\end{equation}
for some $0<\delta' < \delta$.

Upon $\tau=\tau^h(t+1)$, $\sum_{\kappa\in\mathscr{K}: t^{\kappa}=t+1,\action^{\kappa}=\dummyAction} hX^h_{\tau} >0$.
It triggers the decision making period: for all $\kappa\in\mathscr{K}$ with $t^{\kappa}=t+1$, 
\begin{multline}\label{eqn:lemma:sim:6}
\beta^h_{\kappa}(h\bm{X}^h_{\tau},\bm{\xi}^h_{\tau}) = \sum_{\begin{subarray}~\kappa'\in\mathscr{K}:\\t^{\kappa'}=t+1,\\i^{\kappa'}=i^{\kappa},\\s^{\kappa'}=s^{\kappa}\end{subarray}} \Bigl(\indicator\{\action^{\kappa}\neq \dummyAction, \action^{\kappa'} = \dummyAction\}\sum_{\begin{subarray}~\kappa''\in\mathscr{K}:\\t^{\kappa''}=t+1,\\i^{\kappa''}=i^{\kappa},\\s^{\kappa''}=s^{\kappa}\end{subarray}}X^h_{\tau,\kappa''}\scra^{\orule}_{\zeta^{\kappa}}(\frac{\bm{X}^h_{\tau}}{\sum_{i\in[I]}N^0_i},t+1) \\
- \indicator\{\action^{\kappa}=\dummyAction,\action^{\kappa'}\neq \dummyAction\}\sum_{\begin{subarray}~\kappa''\in\mathscr{K}:\\t^{\kappa''}=t+1,\\i^{\kappa''}=i^{\kappa},\\s^{\kappa''}=s^{\kappa}\end{subarray}}X^h_{\tau,\kappa''}\scra^{\orule}_{\zeta^{\kappa'}}(\frac{\bm{X}^h_{\tau}}{\sum_{i\in[I]}N^0_i},t^{\kappa'})\Bigr)\xi^h(0),
\end{multline}
where, based on definitions in \eqref{eqn:define:alpha} and \eqref{eqn:define:alpha:x}, $\scra^{\orule}_{\zeta}(\bm{X}^h_{\tau}/\sum_{i\in[I]}N^0_i,t+1)=\alpha^{\orule}_{\zeta}(\bm{y},t+1)=\bbP\Bigl\{\action^{\psi,h}_{i^{\zeta},n}(t)=\action^{\zeta} \Bigl| s^{\psi,h}_{i^{\zeta},n}(t) = s^{\zeta},\bm{Y}^{\psi,h}(t)=\bm{y}\Bigr\} $ when $\bm{y}=\scru(\bm{X}^h_{\tau}/\sum_{i\in[I]}N^0_i,t+1)$.
Such decision making period continues until for some $\tau_1 > \tau^h(t+1)$ such that $\sum_{\kappa\in\mathscr{K}: t^{\kappa}=t+1,\action^{\kappa}=\dummyAction} hX^h_{\tau,\kappa}$ becomes zero, and then $\bm{X}^h_{\tau}$ enters the transition period for $t+1$.

During the decision making period $\tau\in [\tau^h(t+1),\tau_1)$, based on \eqref{eqn:lemma:sim:6}, exactly 
$h\sum_{i\in[I]}N^0_i\scru(\bm{\mathcal{Z}}^h_{\tau^h(t+1)},t+1)\scra^{\orule}_{\zeta^{\kappa'}}(\bm{\mathcal{Z}}^h_{\tau},t^{\kappa'})$
is moved from $hX^h_{\tau^h(t+1),\kappa}$ to $hX^h_{\tau_1,\kappa'}$.
While, for each $\zeta\in\calJ$ and $\bm{Y}^{\psi,h}(t+1)=\bm{y}$, each bandit process in state $s^{\zeta}$ has probability $\alpha^{\orule}_{\zeta}(\bm{y},t+1)$ to take action $\action^{\zeta}$.
By law of large numbers, along similar lines as the analysis for the transition period,
\begin{multline}\label{eqn:lemma:sim:7}
    \lim_{h\to\infty}\Bigl\{\norm{\bm{\mathcal{Z}}^h_{\tau_1}(t+1) - \bm{Z}^{\psi,h}(t+1)} > \delta\Bigl| \scru(\bm{\mathcal{Z}}^h_{\tau^h(t+1)},t+1), Y^{\psi,h}_{\zeta}(t+1)\Bigr\}\\
    \leq \lim_{h\to\infty} \bbP\Bigl\{\max_{\zeta\in\calJ}\norm{\scra^{\orule}_{\zeta}(\bm{Z}^h_{\tau^h(t+1)},t+1)}\norm{\scru(\bm{\mathcal{Z}}^h_{\tau^h(t+1)},t+1)
    -Y^{\psi,h}_{\zeta}(t+1)}  > \delta'\Bigr\}\\
    \leq \lim_{h\to\infty} \bbP\Bigl\{\norm{\scru(\bm{\mathcal{Z}}^h_{\tau^h(t+1)},t+1)
    -Y^{\psi,h}_{\zeta}(t+1)} > \delta'\Bigr\},
\end{multline}
where $0<\delta''<\delta'<\delta$ and the last inequality is based on \eqref{eqn:lemma:sim:5}.

The decision making and transition periods continue alternatively until $\tau=\tau^h(T+1)$. Hence, given $\bZ^{\psi,h}(0) = \bm{\mathcal{Z}}^h_0 = \bm{x}_0/\sum_{i\in[I]}N^0_i$ that satisfies \eqref{eqn:sim:1}, 
for any $\delta>0$, $t\in[T]$,   
\begin{equation}\label{eqn:lemma:sim:8}
    \lim_{h\to \infty}\bbP\Bigl\{\norm{\bZ^{\psi,h}_{\zeta}(t)- \bm{\mathcal{Z}}^h_{\tau^h(t)}(t)} > \delta\Bigr\} = 0.
\end{equation}

We then discuss the relationship between $\bZ^{\psi,h}(t)$ and $\bZ^{\orule,h}(t)$ as $h\to \infty$.

For any $h\in\mathbb{N}_+$, $t\in[T]_0$, and $\bm{z}\in\Delta_{\calJ}$, we define 
\begin{equation}\label{eqn:define_expect_z}
    \orule^h_t(\bm{z}) \coloneqq \mathbb{E}\Bigl[\bm{Z}^{\orule,h}(t+1) \Bigl| \bm{Z}^{\orule,h}(t) = \bm{z}\Bigr],
\end{equation}
and
\begin{equation}
    \psi^h_t(\bm{z}) \coloneqq \mathbb{E}\Bigl[\bm{Z}^{\psi,h}(t+1) \Bigl| \bm{Z}^{\psi,h}(t) = \bm{z}\Bigr].
\end{equation}
We can then write 
\begin{equation}
    \bm{Z}^{\orule,h}(t+1) = \orule^h_t\bigl(\bm{Z}^{\orule,h}(t)\bigr) +\Delta^{\orule,h}_t\bigl(\bm{Z}^{\orule,h}(t)\bigr),
\end{equation}
and 
\begin{equation}
    \bm{Z}^{\psi,h}(t+1) = \psi^h_t\bigl(\bm{Z}^{\psi,h}(t)\bigr) +\Delta^{\psi,h}_t\bigl(\bm{Z}^{\psi,h}(t)\bigr),
\end{equation}
where $\Delta^{\orule,h}_t\bigl(\bm{z}\bigr)$ and $\Delta^{\psi,h}_t\bigl(\bm{z}\bigr)$ are random variables with zero mean.


By law of the large numbers, for any $t\in[T]_0$ and $\bm{z}\in\Delta_{\calJ}$
\begin{equation}\label{eqn:lemma:sim:11}
    \lim_{h\to \infty}\Delta^{\psi,h}_t(\bm{z}) = 0.
\end{equation}
We then think about $\lim_{h\to\infty}\Delta^{\orule,h}_t(\bm{z})$.
Based on the definition in \eqref{eqn:define_expect_z}, 
the $\zeta$th element of $\orule^h_t(\bm{z})$ is
\begin{equation}\label{eqn:lemma:sim:10}
    \orule^h_{t,\zeta}(\bm{z})=\bbE\Bigl[Y^{\orule,h}_{i^{\zeta},s^{\zeta}}(t+1)\alpha^{\orule,h}_{\zeta}\bigl(\bm{Y}^{\orule,h}(t+1),t+1\bigr)~\Bigl|~\bZ^{\orule,h}(t)=\bm{z}\Bigr].
\end{equation}
For any $\delta>0$, 
\begin{multline}\label{eqn:lemma:sim:9}
    \lim_{h\to\infty}\bbP\Bigl\{\norm{\Delta^{\orule,h}_{t,\zeta}(\bm{z})}>\delta~\Bigl|~\bZ^{\orule,h}(t)=\bm{z}\Bigr\}\\
    \leq 
    \lim_{h\to\infty}\bbP\Bigl\{\norm{Y^{\orule,h}_{i^{\zeta},s^{\zeta}}(t+1)\alpha^{\orule,h}_{\zeta}\bigl(\bm{Y}^{\orule,h}(t+1),t+1\bigr)-\bbE\Bigl[\bm{Y}^{\orule,h}_{i^{\zeta},s^{\zeta}}(t+1)\alpha^{\orule,h}_{\zeta}\bigl(\bm{Y}^{\orule,h}(t+1),t+1\bigr)\Bigr]}>\delta \\~\Bigl|~\bZ^{\orule,h}(t)=\bm{z}\Bigr\}\\
    \leq \lim_{\delta'\downarrow 0}\lim_{h\to\infty}\bbP\Biggl\{\norm{Y^{\orule,h}_{i^{\zeta},s^{\zeta}}(t+1)\alpha^{\orule,h}_{\zeta}\bigl(\bm{Y}^{\orule,h}(t+1),t+1\bigr)-\bbE\Bigl[\bm{Y}^{\orule,h}_{i^{\zeta},s^{\zeta}}(t+1)\alpha^{\orule,h}_{\zeta}\bigl(\bm{Y}^{\orule,h}(t+1),t+1\bigr)\Bigr]}>\delta \\~\Biggl|~\norm{\bm{Y}^{\orule,h}_{i^{\zeta},s^{\zeta}}(t+1)-\bbE\bigl[\bm{Y}^{\orule,h}_{i^{\zeta},s^{\zeta}}(t+1)\bigl|\bZ^{\orule,h}(t)=\bm{z}\bigr]}\leq\delta',\bZ^{\orule,h}(t)=\bm{z}\Biggr\},
\end{multline}
where the last inequality is based on the law of large numbers.
Together with the Lipschitz continuity of $\alpha^{\orule,h}_{\zeta}$ in the first argument for all $h\in\mathbb{N}_+^{\infty}$ (Proposition~\ref{prop:stable_cond:1} and \partialref{cond:weak_stab}{Lipschitz–limit regularity} of $\orule\in\Phi^0$),
the right hand side of \eqref{eqn:lemma:sim:9} becomes $0$.
That is,
\begin{equation}\label{eqn:lemma:sim:11:2}
    \lim_{h\to\infty}\Delta^{\orule,h}_{t,\zeta}(\bm{z}) = 0.
\end{equation}

For $\delta>0$, based on \eqref{eqn:lemma:sim:11} and \eqref{eqn:lemma:sim:11:2},
\begin{multline}\label{eqn:lemma:sim:12}
    \lim_{h\to \infty}\bbP\Biggl\{\norm{\bm{Z}^{\orule,h}(t+1)-\bm{Z}^{\psi,h}(t+1)}>\delta\Biggr\} \\
    \leq \lim_{h\to\infty}\bbP\Biggl\{\norm{\orule^h_t\Bigl(\bm{Z}^{\orule,h}(t)\Bigr)-\orule^h_t\Bigl(\bm{Z}^{\psi,h}(t)\Bigr)}+\norm{\orule^h_t\Bigl(\bm{Z}^{\psi,h}(t)\Bigr)-\psi^h_t\Bigl(\bm{Z}^{\psi,h}(t)\Bigr)} >\delta'\Biggr\}, 
\end{multline}
for some $\delta' < \delta$.

Let $\bm{y}_1\coloneqq \bbE\bigl[\bm{Y}^{\orule,h}_{i^{\zeta},s^{\zeta}}(t+1)\bigl|\bZ^{\orule,h}(t)=\bm{z}\bigr]$ and $\bm{y}_2 \coloneqq \bbE\bigl[\bm{Y}^{\orule,h}_{i^{\zeta},s^{\zeta}}(t+1)\bigl|\bZ^{\orule,h}(t)=\bm{z}+\bm{\Delta}\bigr]$ for some $\bm{\Delta}$.
Similar to \eqref{eqn:lemma:sim:9}, based on the Lipschitz continuity of $\bm{\alpha}^{\orule,h}$ in its first argument for all $h\in\mathbb{N}_+^{\infty}$ (Proposition~\ref{prop:stable_cond:1} and \partialref{cond:weak_stab}{Lipschitz–limit regularity} of $\orule\in\Phi^0$), we obtain
\begin{multline}
  \lim_{h\to\infty}\norm{\orule^h_{t,\zeta}(\bm{z})-\orule^h_{t,\zeta}(\bm{z}+\bm{\Delta})}
  \leq \lim_{h\to \infty}\norm{y_{1,i^{\zeta},s^{\zeta}}\alpha^{\orule,h}_{\zeta}(\bm{y}_1,t+1) - y_{2,i^{\zeta},s^{\zeta}}\alpha^{\orule,h}_{\zeta}(\bm{y}_2,t+1)}\\
  \leq K\lim_{h\to\infty}\norm{\bm{y}_1-\bm{y}_2}\leq K\norm{\bm{\Delta}},
\end{multline}
where $K$ is the Lipschitz constant of $y_{i^{\zeta},s^{\zeta}}\alpha^{\orule}_{\zeta}(\bm{y},t+1)$ with respect to $\bm{y}$.
That is, there exists $K<\infty$ (independent to $h$) such that
\begin{equation}\label{eqn:lemma:sim:13}
    \lim_{h\to\infty} \norm{\orule^h_t\Bigl(\bm{Z}^{\orule,h}(t)\Bigr)-\orule^h_t\Bigl(\bm{Z}^{\psi,h}(t)\Bigr)}\leq K\Bigl\lVert \bm{Z}^{\orule,h}(t)-\bm{Z}^{\psi,h}(t)\Bigr\rVert.
\end{equation}
Also, recall the definition \eqref{eqn:define_psi},
for any $\bm{z}\in\Delta_{\calJ}$, 
\begin{equation}\label{eqn:lemma:sim:14}
    \lim_{h\to\infty} \norm{\orule^h_t(\bm{z}) - \psi^h_t(\bm{z})} = 0.
\end{equation}
Plugging \eqref{eqn:lemma:sim:13} and \eqref{eqn:lemma:sim:14} into \eqref{eqn:lemma:sim:12}, there exists $\delta_1 < \delta$ such that
\begin{multline}\label{eqn:lemma:sim:15}
   \lim_{h\to \infty}\bbP\Biggl\{\norm{\bm{Z}^{\orule,h}(t+1)-\bm{Z}^{\psi,h}(t+1)}>\delta\Biggr\} 
   \leq \lim_{h\to \infty} \bbP\Biggl\{ K\norm{\bm{Z}^{\orule,h}(t)-\bm{Z}^{\psi,h}(t)} > \delta_1\Biggr\}\\
   \leq \lim_{h\to \infty} \bbP\Biggl\{ K\norm{\bm{Z}^{\orule,h}(t)-\bm{Z}^{\psi,h}(t)} > \delta_1\Biggl| \norm{\bm{Z}^{\orule,h}(t)-\bm{Z}^{\psi,h}(t)} \leq \delta_1/K\Biggr\} \\
   + \lim_{h\to \infty} \bbP\Biggl\{\norm{\bm{Z}^{\orule,h}(t)-\bm{Z}^{\psi,h}(t)} > \delta_1/K\Biggr\}\\
   \leq \lim_{h\to \infty} \bbP\Biggl\{\norm{\bm{Z}^{\orule,h}(t)-\bm{Z}^{\psi,h}(t)} > \delta_1/K\Biggr\}.
\end{multline}
Together with given initial condition $\bm{Z}^{\orule,h}(0)=\bm{Z}^{\psi,h}(0)=\bm{z}_0$, for all $t\in[T]_0$,
\begin{equation}
    \lim_{h\to \infty}\bbP\Biggl\{\norm{\bm{Z}^{\orule,h}(t)-\bm{Z}^{\psi,h}(t)}>\delta\Biggr\} =0.
\end{equation}
Since \eqref{eqn:lemma:sim:8} holds for all $t\in[T]_0$, it proves the lemma.

\endproof

\section{Proof of Lemma~\ref{lemma:action_ALP_satisfied}}
\label{app:lemma:action_ALP_satisfied}
\proof{Proof of Lemma~\ref{lemma:action_ALP_satisfied}.}
For the \ALP policy family, the difference between $\bm{A}^{\psi,h}(t)$ and $\boalphaLP(\bm{x}^*,t)$ is mainly caused by two parts:
\begin{enumerate}[label=(\alph*)]
    \item the fractional part of the numerators $\oalphaLP_{i,s,\action}(\bm{x}^*,t)Y^{\psi,h}_{i,s}(t) h\sum_{i\in[I]}N^0_i$ ($(i,s,\action)\in\mathcal{J}$) in Step~\ref{step:1}, and
    \item the adapted proportion of bandit processes $\max_{m\in[M]}\Delta Z_m$ in Step~\ref{step:5} where $M$ and $\Delta Z_m$ are $(j,t)$-dependent.
\end{enumerate}
In particular, there exist $K_1,K_2,K_3\in\mathbb{R}_+$ such that
\begin{multline}\label{eqn:alp:1}
    \Bigl\lvert \bigl(\alpha^{\psi,h}_{i,s,\action}(\bm{Y}^{\psi,h}(t),t)-\oalphaLP_{i,s,\action}(\bm{x}^*,t)\bigr)Y^{\psi,h}_{i,s}(t)\bigr)\Bigr\rvert \\
    \leq K_1\max_{(i',s',\action')\in\mathcal{J}:Y^{\psi,h}_{i',s'}(t)>0}\frac{\biggl\lvert\oalphaLP_{i',s',\action'}(\bm{x}^*,t)h\sum_{i''\in[I]}N^0_{i''} Y^{\psi,h}_{i',s'}(t) - \Bigl\lfloor \oalphaLP_{i',s',\action'}(\bm{x}^*,t)h\sum_{i''\in[I]}N^0_{i''} Y^{\psi,h}_{i',s'}(t)\Bigr\rfloor\biggr\rvert }{h\sum_{i''\in[I]}N^0_{i''} } \\
    +K_2 \max_{i\in[I],s\in\bS_i: Y^{\psi,h}_{i,s}(t)>0}\Biggl\lvert Y^{\psi,h}_{i,s}(t) -\sum_{\action'\in\bA_i:\action'\neq a^0(i,s)}\frac{\bigl\lfloor \oalphaLP_{i,s,\action'}(\bm{x}^*,t)h\sum_{i'\in[I]}N^0_{i'} Y^{\psi,h}_{i,s}(t)\bigr\rfloor}{h\sum_{i'\in[I]}N^0_{i'}} - \oalphaLP_{i,s,a^0(i,s)}(\bm{x}^*,t)Y^{\psi,h}_{i,s}(t)\Biggr\rvert\\
    + K_3 \max_{(i',s',\action')\in\mathcal{J}:Y^{\psi,h}_{i',s'}(t)>0}\max_{m\in[M_{\ell}(i',s',\action')]}\Delta Z_{m,\ell}(i',s',\action') \\
    \leq  \frac{o(h)}{h} + K_3 \max_{(i',s',\action')\in\mathcal{J}:Y^{\psi,h}_{i',s'}(t)>0}\max_{m\in[M_{\ell}(i',s',\action')]}\Delta Z_{m,\ell}(i',s',\action'),
\end{multline}
where $(i,s,a^0(i,s))$ is the picked up \GSA triple in Step~\ref{step:2} for each $(i,s)$, and $\Delta Z_{m,\ell}(i,s,\action)$ are those defined in \eqref{eqn:define_deltaZ} and updated in Step~\ref{step:5}.

Since $\bm{x}^*$ is an optimal solution to \eqref{eqn:obj:linear programming}-\eqref{eqn:constraint:linear programming:4}, it satisfies \eqref{eqn:constraint:linear programming:4}.
If 
\begin{equation}\label{eqn:alp:2}
    \lim_{h\rightarrow \infty}\lVert\bm{Y}^{\psi,h}(t)-\bm{y}^{\phi(\bm{x}^*),h}(t)\rVert = 0,
\end{equation}
then from the law of large numbers, the $L_{\ell}(t)$ obtained at the beginning of Step~\ref{step:3} converges to the left hand side of \eqref{eqn:constraint:linear programming:4} with $\bm{x}=\bm{x}^*$.
In this case, we obtain $\lim_{h\rightarrow \infty} \max\{0,L_{\ell}(t)\} = 0$, and hence $\lim_{h\rightarrow \infty} \Delta Z_{m,\ell}(i,s,\action) = 0$ for all $(i,s,\action)\in\mathcal{J}$ and $m\in [M_{\ell}(i,s,\action)]$.
Together with \eqref{eqn:alp:1}, it leads to \eqref{eqn:assumption:action} with plugged in $\phi = \psi$ and $\bm{x} = \bm{x}^*$.

Given the initial condition $\bm{Y}^{\psi,h}(0)=\bm{y}^0 = \bm{y}^{\phi(\bm{x}^*),h}(0)$, together with \eqref{eqn:alp:1}, \eqref{eqn:assumption:action} (plugged in $\phi = \psi$ and $\bm{x}=\bm{x}^*$) holds for $t=0$.
By invoking Proposition~\ref{prop:asym_opt:LP} with specified $T=0$, we get \eqref{eqn:alp:2} for $t=1$.
Assume that \eqref{eqn:assumption:action} and \eqref{eqn:alp:2} hold for $t\leq T'$. 
By invoking Proposition~\ref{prop:asym_opt:LP} for $T=T'$, we can iteratively obtain  \eqref{eqn:alp:2} for $t=T'+1$, leading to \eqref{eqn:assumption:action} for $t=T'+1$.
Hence, \eqref{eqn:assumption:action} holds for any $T<\infty$.

\endproof

\section{Proof of Lemma~\ref{lemma:ALP_stable}}
\label{app:lemma:ALP_stable}
The proof follows similar lines as those for proving Propositions~\ref{prop:stable:fluid-balance} and \ref{prop:stable:fluid-balance:plus}.
\proof{Proof of Lemma~\ref{lemma:ALP_stable}.}

By Step~\ref{step:3}, the LP-Approx policy is formed, for which 
$$
\alpha^{\text{LP-Approx}}_{i,s,\action}(\bm{y},t) = \lim_{h\to\infty} \frac{Z^{\text{LP-Approx},h}_{i,s,\action}(t)}{Y^{\text{LP-Approx},h}_{i,s}(t)}\Bigl|_{\bm{Y}^{\text{LP-Approx},h}(t)=\bm{y}} = \begin{cases}
    \oalphaLP_{i,s,\action}(\bm{x}^*,t),\text{if } \action \neq \action^0(i,s),\\
    1-\sum_{\begin{subarray}~\action'\in\bA_i:\\\action'\neq \action^0(i,s)\end{subarray}}\oalphaLP_{i,s,\action'}(\bm{x}^*,t),&\text{otherwise}.
\end{cases}
$$

The \ALP policy family, denoted by $\psi$, is constructed based on such LP-Approx with further Steps~\ref{step:4}-\ref{step:done_ell}.

Recall that, in Step~\ref{step:5}, $j=1,2,\ldots,|\calJ|$, are the rankings of all the GSA triples in $\calJ$, and, in this step, the algorithm considers fixed $j$ and $\ell\in[L]$.
For the $j$th GSA, $(i(j),s(j),a(j))$, the algorithm further ranks its associated $M_{\ell}\bigl(i(j),s(j),a(j)\bigr)$ actions. 
In Step~\ref{step:5}, we use $\action_m$ to represent the $m$th action of the $j$th GSA triple. 
In this context, $(i(j),s(j),\action_m)$ again forms a GSA in $\calJ$.

Based on the definition in Step~\ref{step:5}, given $\bm{Y}^{\psi,h}(t)=\bm{y}$, for each $(j,\ell)\in [|\calJ|]\times[L]$ and $m\in \Bigl[M_{\ell}\bigl(i(j),s(j),a(j)\bigr)\Bigr]$, $\Delta Z_{m,\ell}\bigl(i(j),s(j),a(j)\bigr)$ is updated at most once during the algorithm.
If, for some $(j,\ell)$ and $m$, $\Delta Z_{m,\ell}\bigl(i(j),s(j),a(j)\bigr)$ is never updated during the algorithm, then we set $\Delta Z_{m,\ell}\bigl(i(j),s(j),a(j)\bigr)=0$ without loss of generality.

For each $(j,\ell)\in [|\calJ|]\times[L]$ and $m\in \Bigl[M_{\ell}\bigl(i(j),s(j),a(j)\bigr)\Bigr]$, we define 
$$\bar{\Delta} Z_{m,j,\ell}(\bm{y})\coloneqq \lim_{h\to \infty}\frac{\Delta Z_{m,\ell}\bigl(i(j),s(j),a(j)\bigr)}{Y^{\psi,h}_{i(j),s(j)}(t)}\Bigl|_{\bm{Y}^{\psi,h}(t)=\bm{y}}.$$
Along similar lines as the proof of Propositions~\ref{prop:stable:fluid-balance} (in Appendix~\ref{app:stability}), 
such $\bar{\Delta} Z_{m,j,\ell}(\bm{y})$ satisfies the following conditions.
\begin{enumerate}\setcounter{enumi}{4}
    \item $\bar{\Delta} Z_{m,j,\ell}(\bm{y})$ exists and is Lipschitz continuous in $ \bigl\{\bm{y}\in\deltay\bigl| y_{i(j),s(j)}>0\bigr\}$.\label{cond:deltaZ:1}
    \item $\bar{\Delta} Z_{m,j,\ell}(\bm{y})$ is continuous in $\bigl\{\bm{y}\in\deltay\bigl| y_{i(j),s(j)}=0\bigr\}$.  \label{cond:deltaZ:2}
    \item $\bar{\Delta} Z_{m,j,\ell}(\bm{y})$ either exists as a finite number or tends to positive infinity. \label{cond:deltaZ:3}
    \item If $\bar{\Delta} Z_{m,j,\ell}(\bm{y}) < \infty$, then there exists $\delta y >0$ such that, for any $y_{i(j),s(j)}\in[0,\delta y]$, $ \bar{\Delta} Z_{m,j,\ell}(\bm{y})$ is constant.\label{cond:deltaZ:4}
\end{enumerate}

By Step~\ref{step:3}, $Z^{\psi,h}_{i,s,\action}(t)$ ($(i,s,\action)\in\calJ$) are initialized to be $Z^{\text{LP-Approx},h}_{i,s,\action}(t)$, which are constants, independent to $\bm{Y}^{\psi,h}(t)$.
In Step~\ref{step:5}, for each $(j,\ell)\in [|\calJ|]\times[L]$ and $m\in \Bigl[M_{\ell}\bigl(i(j),s(j),a(j)\bigr)\Bigr]$, upon updating $\Delta Z_{m,\ell}\bigl(i(j),s(j),a(j)\bigr)$, the algorithm also updates $Z^{\psi,h}_{i(j),s(j),\action_m}(t)$ with $ Z^{\psi,h}_{i(j),s(j),\action_m}(t) +\min\Bigl\{Z^{\psi,h}_{i(j),s(j),\action_{m+1}}(t), \Delta Z_{m,\ell}\bigl(i(j),s(j),a(j)\bigr)\Bigr\}$ and update $Z^{\psi,h}_{i(j),s(j),\action_{m+1}}(t)$ with $ Z^{\psi,h}_{i(j),s(j),\action_{m+1}}(t) -\min\Bigl\{Z^{\psi,h}_{i(j),s(j),\action_{m+1}}(t), \Delta Z_{m,\ell}\bigl(i(j),s(j),a(j)\bigr)\Bigr\}$.
Upon every such update in Step~\ref{step:5}, based on Conditions~\ref{cond:deltaZ:1}-\ref{cond:deltaZ:4}, the increment/decrement part satisfies that
$$
\lim_{h\to\infty} \frac{1}{Y^{\psi,h}_{i(j),s(j)}(t)}\min\Bigl\{Z^{\psi,h}_{i(j),s(j),\action_{m+1}}, \Delta Z_{m,\ell}\bigl(i(j),s(j),a(j)\bigr)\Bigr\}\Bigl|_{\bm{Y}^{\psi,h}(t)=\bm{y}}  
$$
is Lipschiz continuous in $\bm{y}\in \deltay$.
Hence, at the end of the algorithm, 
$$
\alpha^{\psi}_{i(j),s(j),a(j)}(\bm{y},t) = \lim_{h\to\infty}\frac{Z^{\psi,h}_{i(j),s(j),a(j)}(t)}{Y^{\psi,h}_{i(j),s(j)}(t)}\Bigl|_{\bm{Y}^{\psi,h}(t)=\bm{y}}
$$
is also Lipschitz continuous in $\bm{y}\in \deltay$.
It proves the proposition.

\endproof

\bibliographystyle{alpha}
\bibliography{references/IEEEabrv,references/bib1}
\end{document}